\documentclass[a4paper]{article}

\usepackage{microtype}
\usepackage{amsmath,amssymb, amsthm, accents}
\usepackage{xstring}
\usepackage{xcolor}
\usepackage{latexsym}
\usepackage[latin1]{inputenc}
\usepackage{enumitem}
\usepackage[mathscr]{euscript}

\usepackage{tikzit}
\tikzstyle{bullet}=[fill={rgb,255: red,6; green,6; blue,6}, draw=black, shape=circle, minimum size=.02cm, scale=0.6]
\tikzstyle{circle}=[fill=white, draw=black, shape=circle, inner sep=0pt]

\tikzstyle{new edge style 0}=[-]

\usepackage{graphicx}
\usepackage{color}
\usepackage{mathtools}

\usepackage{url}
\usepackage[colorlinks=true, linkcolor={blue!50!black}, pdfhighlight=/O, ocgcolorlinks=true]{hyperref}
\hypersetup{bookmarksdepth=3}

\usepackage[colorinlistoftodos]{todonotes}

\usepackage{tikz}
\usepackage{tikz-qtree}
\usepackage{tikz-cd}
\usetikzlibrary{decorations.pathmorphing}

\usepackage[normalem]{ulem}
\usepackage[affil-it]{authblk}

\usepackage[a4paper, total={5.5in,8.5in}]{geometry}

\numberwithin{equation}{section}
\theoremstyle{plain}
\newtheorem{theorem}[equation]{Theorem}
\newtheorem{lemma}[equation]{Lemma}
\newtheorem{proposition}[equation]{Proposition}
\newtheorem{corollary}[equation]{Corollary}
\newtheorem{introtheorem}{Theorem}

\theoremstyle{definition}

\newtheorem{definition}[equation]{Definition}
\newtheorem{construction}[equation]{Construction}
\newtheorem{example}[equation]{Example}
\newtheorem{remark}[equation]{Remark}
\newtheorem{convention}[equation]{Convention}

\newtheorem{observation}[equation]{Observation}

\newtheorem{warning}[equation]{Warning}
\newtheorem{propdef}[equation]{Proposition/definition}

\setlist[enumerate]{label=(\arabic*), leftmargin=*}
\setenumerate{label=(\arabic*), leftmargin=*}
\setlist[itemize]{label=$\vcenter{\hbox{\footnotesize$\bullet$}}$, leftmargin=*}
\setitemize{label=$\vcenter{\hbox{\footnotesize$\bullet$}}$, leftmargin=*}

\newcommand{\mb}[1]{\mathbf{#1}}
\newcommand{\mm}[1]{\mathrm{#1}}
\newcommand{\mf}[1]{\mathfrak{#1}}
\newcommand{\mc}[1]{\mathcal{#1}}

\newcommand{\ol}[1]{\overline{#1}}

\newcommand{\cat}[1]{
\StrLen{#1}[\mystrlen]
\ifnum\mystrlen=1 \mathscr{#1}
\else \mathrm{#1}
\fi}
\newcommand{\scat}[1]{\mb{#1}}
\newcommand{\oocat}[1]{
\StrLen{#1}[\mystrlen]
\ifnum\mystrlen=1 \mathscr{#1}
\else \mathbf{#1}
\fi}

\newcommand{\colim}{\operatornamewithlimits{\mathrm{colim}}}

\newcommand{\holim}{\operatornamewithlimits{\mathrm{holim}}}
\newcommand{\Hom}[0]{\mm{Hom}}
\newcommand{\Map}[0]{\mm{Map}}
\newcommand{\rto}[1]{\stackrel{#1}{\rt}}

\newcommand{\rt}[0]{\longrightarrow}
\newcommand{\lt}[0]{\longleftarrow}
\newcommand{\wgt}[1]{\langle #1\rangle}

\newcommand{\longsquiggly}{\xymatrix{{}\ar@{~>}[r]&{}}}

\def\dar{\ar@{*{\hspace{-1pt}}*{\hspace{1pt}\text{-}}>}}
\def\ddar{\ar@{*[left]{\hspace{-1pt}}*[left]+<-4.5pt, 0pt>{\hspace{1pt}\text{-}}>}}

\newcommand{\Fun}[0]{\cat{Fun}}

\newcommand{\gr}[0]{\mm{gr}} %%% graded as adjective/index
\newcommand{\tot}[0]{\mm{tot}}
\newcommand{\cl}[0]{\mm{cl}}
\newcommand{\mix}[0]{\mm{mix}}

\renewcommand{\Bar}{\operatorname{Bar}}%R:Changed it to Bar because of all the B's around...

\DeclareMathOperator{\End}{End}

\DeclareMathOperator{\Tot}{Tot}
\DeclareMathOperator{\Tw}{Tw}

\DeclareMathOperator{\id}{id}
\DeclareMathOperator{\bcirc}{\hat{\circ}}
\DeclareMathOperator{\botimes}{\hat{\otimes}}

\newcommand{\Alg}[0]{\cat{Alg}}
\newcommand{\Mod}[0]{\cat{Mod}}
\newcommand{\MC}[0]{\mm{MC}}

\newcommand{\CC}[0]{\mathscr{C}}

\newcommand{\PP}[0]{\mathscr{P}}

\newcommand{\uCC}[0]{\mm{u}\mathscr{C}}
\newcommand{\cPP}[0]{\mm{c}\mathscr{P}}
\newcommand{\Sym}[0]{\mathrm{Sym}}

\newcommand{\antishriek}{{\scriptstyle \text{\rm !`}}}

\newcommand{\Gr}{\operatorname{Gr}} %%% functor taking ass graded
\newcommand{\curv}{\operatorname{curv}}
\newcommand{\sgn}{\operatorname{sgn}}
\newcommand{\dR}{\mm{dR}}

\newcommand{\Spec}{\operatorname{Spec}}
\newcommand{\blend}{\operatorname{blend}}

\newcommand{\Loid}{\mathfrak{L}}%canonical name for a Lie algebroid...
\newcommand{\Loide}{\mathfrak{H}}
\newcommand{\Loi}{\mathfrak{F}}
\newcommand{\tang}{\mathfrak{t}}%canonical name for fake tangent
\newcommand{\keranc}{\mf{n}}%canonical name for kernel of anchor
\newcommand{\kerance}{\mf{m}}%canonical name for kernel of anchor

\newcommand{\As}[0]{\mathsf{As}}

\newcommand{\ucoas}{\mathrm{u}\mathsf{coAs}}
\newcommand{\Lie}[0]{\mathsf{Lie}}

\newcommand{\Pois}[0]{\mathsf{Pois}}

\newcommand{\cocom}{\mathsf{coCom}}
\newcommand{\ucocom}{\mathsf{ucoCom}}
\newcommand{\cLie}{\mathsf{cLie}}

\title{Lie algebroids are curved Lie algebras}

\author[]{Damien Calaque\thanks{\href{mailto:damien.calaque@umontpellier.fr}{damien.calaque@umontpellier.fr}}}
\author[]{Ricardo Campos\thanks{\href{mailto:ricardo.campos@umontpellier.fr}{ricardo.campos@umontpellier.fr}}}
\author[]{Joost Nuiten \thanks{\href{mailto:joost-jakob.nuiten@umontpellier.fr}{joost-jakob.nuiten@umontpellier.fr}}}
\affil[]{IMAG, Univ. Montpellier, CNRS, Montpellier, France.}

\begin{document}

\maketitle
\begin{abstract}
%We show that Lie algebroids form an $\infty$-category equivalent to certain curved Lie algebras. For this we develop 
%a context under which the homotopy theory of curved Lie algebras comes from seeing them as algebras over a complete operad. 
%
%\joostline{how about this slight reformulation: (D: I'm happy with both :-) )}
We show that there is an equivalence of $\infty$-categories between Lie algebroids and certain kinds of curved Lie algebras. For this we develop a method to study the $\infty$-category of curved Lie algebras using the homotopy theory of algebras over a complete operad. 

\end{abstract}

\tableofcontents

\section{Introduction}

Differential graded (dg) Lie algebras have shown to be of great importance in deformation theory and rational homotopy theory 
\cite{SchlessingerStasheff2012,QuillenAnnals}. On the deformation theory side, this culminated with a theorem 
of Lurie and Pridham \cite{DAGX,Pridham} stating an equivalence (of $\infty$-categories) between dg-Lie algebras and pointed 
formal deformation problems (also known as formal moduli problems, or FMPs). 

On the side of rational homotopy theory, (reduced) dg-Lie algebras are models for rational $1$-connected pointed spaces. 
In both of these cases, dg-Lie algebras arise from the same procedure: given a pointed space or pointed formal moduli problem, the loop space at the basepoint has the structure of  group and the corresponding dg-Lie algebra is the `tangent space' of this group (this perspective on rational homotopy theory is made more precise in \cite{LurieDAGXIII}). In particular, the datum of a basepoint plays a crucial role in the appearance of dg-Lie algebras.

Dg-Lie algebras sit inside the larger category of \emph{curved Lie algebras}, which are graded Lie algebras with a ``differential'' that does not square to zero, but whose square is controlled by a curvature element: $d^2=[\theta,-]$. It is a well-accepted idea that such curved Lie algebras, or more generally curved $L_\infty$-algebras, correspond to geometric objects without a fixed basepoint. Indeed, let us point out that in the seminal paper \cite{Kontsevich2003} Kontsevich considers formal graded pointed $Q$-manifolds, which are equivalent to $L_\infty$-algebras; the unpointed version is known to correspond to curved $L_\infty$-algebras.
Based on this idea, there have been attempts to approach unbased rational homotopy theory by using curved dg-Lie algebras \cite{maunder2015unbased, maunder2017Koszul, ChuangLazarevMannan2016}. 

A similar philosophy has been used on the deformation theory side, where people encountered the need for a version of 
deformation theory under or over a given space \cite{calaque2018formal}. For instance, Costello \cite{costellowittenII} uses curved Lie algebras, or rather curved $L_\infty$-algebras, as models for certain formal derived (differentiable) stacks that appear in field theories, which he calls $L_\infty$-spaces. These formal derived stacks can be seen as formal thickenings of a given manifold $X$, and the corresponding curved $L_\infty$-algebras live over the de Rham algebra $\Omega^*(X)$. In the setting of Costello's formal derived geometry, Grady and Gwilliam \cite{grady_gwilliam} have shown that every \emph{Lie algebroid} 
over a manifold $X$ can be viewed as an $L_\infty$-space. Recall that Lie algebroids are to Lie groupoids what Lie algebras are to Lie groups. It is therefore not so surprising to see them appearing in the context of ``unbased'' deformation theory.

In the context of derived \emph{algebraic} geometry, the $L_\infty$-spaces of Costello roughly correspond to the so-called perfect families of affine formal derived stacks over $X_{\dR}$, as defined and studied in \cite{CPTVV2017}. Indeed, these are formal thickenings of $X$ sitting between $X$ itself and its de Rham stack as
$$\begin{tikzcd}
X\arrow[r] & Y \arrow[r] & X_{\dR}.
\end{tikzcd}$$
From the perspective of $X_{\dR}$, one can view $Y$ as a formal thickening that does not quite come equipped with a basepoint (i.e.\ a section of the second map). One therefore expects $Y$ to give rise to a curved $L_\infty$-algebra over $\Omega^*(X)$, i.e.\ to an $L_\infty$-space.

From the point of view of $X$, one can view $Y$ as the quotient of $X$ by the formal groupoid $X\times_Y X \rightrightarrows X$. Consequently, $Y$ should give rise 
to a Lie algebroid on $X$ (see e.g.\ \cite{CalaqueCaldararuTu2014,gaitsgory_rozenblyum_vol2, Yu2017Dolbeault}).
In the algebraic context, the third author proved \cite{nuiten2019koszul} that dg-Lie algebroids are indeed equivalent to 
formal moduli problems under $X=\mathrm{Spec}(A)$ (see also \cite{calaque2018formal}), for $A$ a connective commutative differential 
graded algebra (cdga). Together with previous works on formal derived geometry \cite{gaitsgory_rozenblyum_vol2,CPTVV2017}, 
this suggests that dg-Lie algebroids over $X$ do not just give rise to curved Lie algebras over $X_{\dR}$, but that their $\infty$-categories should be very closely related. 
Indeed, the main objective of this paper is to show that there is an equivalence (of $\infty$-categories) 
between dg-Lie algebroids over $A$ and certain curved Lie algebras over the complete filtered de Rham algebra of $A$:

\begin{introtheorem}[See Theorem \ref{thm:comparison tangent case}]\label{thm:main intro}
	Let $k$ be a field of characteristic zero and suppose that $A$ is a smooth algebra or a cofibrant cdga over $k$, locally of finite presentation. Then there is an equivalence of $\infty$-categories
	$$\begin{tikzcd}
		\curv\colon \cat{Lie\ algebroids}(A/k)\arrow[r] & \cat{Curved\ Lie\ algebras}_{\dR(A)},
	\end{tikzcd}$$
	between Lie algebroids over $A$ and curved Lie algebras over the de Rham algebra $\dR(A)$ equipped with the Hodge filtration, satisfying a normalizing assumption.
		This equivalence sends a Lie algebroid $L \stackrel{\rho}{\twoheadrightarrow} T_A$ to $\curv(L) = \ker(\rho) \otimes_A \dR(A)$, see Section \ref{sec:explicit}.
\end{introtheorem}
This result relies on having a well-behaved homotopy theory for Lie algebroids and curved Lie algebras. The $\infty$-category of Lie algebroids can be described efficiently in terms of model categories, but for curved Lie algebras this issue is more subtle.

Indeed, even though curved algebraic structures have already appeared in many areas (matrix factorizations \cite{CaldararuTu}, deformation quantization \cite{CattaneoFelder}, Floer theory \cite{Fukaya2003,FOOO}\mbox{, ...)}, 
their homotopy theory is still a subject of ongoing study.
In particular, curved Lie algebras have no underlying cochain complexes and hence do not admit an obvious homotopy theory in terms of quasi-isomorphisms. To make sense of Theorem \ref{thm:main intro}, the first question that needs to be answered is therefore: ``what is a good homotopy theory for curved Lie algebras?''.

Various approaches to the homotopy theory of curved objects have been presented in the literature, each suiting different purposes \cite{AmorimTu2020,Bellier-MillesDrummond-Cole2020, DotsenkoShadrinVallette2018, HirshMilles2012, positselskiweaklycurved}.
A secondary purpose of this paper is to develop a homotopy theory for curved algebras (including homotopy transfer theorem) which is suitable for the study of derived deformation theory and in particular for Theorem \ref{thm:main intro}.
 The basic idea will be to endow objects with a complete filtration and control their homotopy theory by the associated graded. 

The curved Lie algebras appearing in Theorem \ref{thm:main intro} will then come with a complete filtration; geometrically, this means that they correspond to formal stacks sitting in between $X$ and its Hodge stack $X_\mm{Hodge}$; the latter is a stack over $\hat{\mathbb{A}}^1/\mathbb{G}_m$ controlling the Hodge filtration on de Rham cohomology, whose special fiber is the shifted tangent bundle $T[1]X$. This geometric picture is substantiated by Theorems \ref{thm:intro curved vs uncurved} and \ref{thm:intro general main} below.

\medskip

We point out that Theorem \ref{thm:main intro} does not apply to the situations usually considered in differential geometry: the algebra of functions on a smooth manifold $\mc{C}^\infty(M)$ is not in the conditions of Theorem \ref{thm:comparison tangent case} and the notion of Lie and $L_\infty$-algebroids varies slightly, as Lie algebroids are typically required to arise from vector bundles \cite{mackenzie87,pradines67}.
Nevertheless the same methods can be adapted to prove a differential-geometric version of the main theorem:

\begin{introtheorem}[See Theorem \ref{thm:comparison differentiable case}]
	Let $M$ be a differentiable manifold. The $\curv$ construction establishes an equivalence of $\infty$-categories
	$$\begin{tikzcd}
		\curv\colon \cat{Lie\ algebroids}(M)\arrow[r] & \cat{Vector\ bundle\ curved\ Lie\ algebras}_{\Omega^*(M)},
	\end{tikzcd}$$
between (differential-geometric) $L_\infty$-algebroids over $M$ and those curved $L_\infty$-algebras $\mf{g}$ over the de Rham complex $\Omega^*(M)$ of $M$ that are of the form $\mf{g}\simeq \Omega^*(M)\otimes_{\mc{C}^\infty(M)} E$, with $E$ a bounded above graded vector bundle on $M$.
\end{introtheorem}

\subsection*{Outline and main results}

Contrary to dg-Lie algebras, curved Lie algebras and curved $L_\infty$-algebras do not directly form a model category. This can be explained by the fact that curved Lie (or curved $L_\infty$-)algebras do not arise as algebras in cochain complexes over a `curved Lie' operad.

Nevertheless, part of the theory of curved Lie algebras works as if such an operad of curved Lie algebras existed, and its Koszul dual cooperad were the linear dual of the operad governing unital commutative algebras. For example, there is a ``bar construction'' sending a curved Lie algebra $\mf{g}$ to the cocommutative coalgebra $\Sym^c(\mf{g}[1])$, as well as a natural notion of $\infty$-morphisms.
Furthermore, while curved Lie algebras are not cochain complexes but instead have a pre-differential that does not square to zero, the ones ``appearing in nature'' do typically carry a natural complete filtration such that the pre-differential squares to zero on the associated graded.

\medskip

$\rhd$ The approach carried out in \textbf{Section \ref{sec:complete operads}} aims to formalize the heuristic above.
The first step is to work in the underlying category of \emph{complete filtered complexes}, i.e.\ cochain complexes equipped with a decreasing complete filtration. 
In this category one can define obvious notions of complete operads and their algebras; we show that algebras over a complete filtered operad form a model category such that weak equivalences are maps inducing quasi-isomorphisms at the level of the associated graded (Theorem \ref{thm:Model str}).
Equivalently, one can think of such objects as graded mixed complexes as appearing in \cite{CPTVV2017}, see Section \ref{sec:graded mixed}. Most of the operadic calculus from \cite{LodayVallette2012}, such as the (co)operadic bar-cobar constructions, $\infty$-morphisms and homotopy transfer, generalises to the complete filtered setting.

In particular, applying this machinery to a filtered version of the counital cocommutative cooperad $\ucocom$, we obtain a complete operad $\cLie_{\infty} \coloneqq \Omega(\ucocom\{1\})$. The algebras over this operad, which we call \emph{mixed-curved $L_\infty$-algebras}, differ from ordinary curved $L_\infty$-algebras in the sense that their pre-differential comes with a splitting as $d+\ell_1$ where $d$ squares strictly to zero and $\ell_1$ is filtration increasing.
Still, the model category of mixed-curved $L_\infty$-algebras can be fruitfully used to study the $\infty$-category of curved $L_\infty$-algebras and $\infty$-morphisms between them (Definition \ref{def:cat of classical curved Lie}). In short, we prove the following result:

\begin{introtheorem}[See Section \ref{sec:Koszul morphism} and Corollary \ref{cor:main result}]\label{thmC}
The complete operad $\cLie_\infty$ admits a model $\cLie$, whose algebras are mixed-curved Lie algebras. The $\infty$-category $\oocat{cLie}^\mm{mix}$ associated to the model category of mixed-curved Lie algebras is equivalent to the $\infty$-category of mixed-curved $L_\infty$-algebras with $\infty$-morphisms between them.

The (simplicial) $\infty$-category $\oocat{cLie}$ of curved Lie algebras is then given by a pullback of $\infty$-categories, each of which arises from a model category
$$
\oocat{cLie}\simeq \oocat{cLie}^\mm{mix}\times_{\oocat{Mod}_k^\mm{cpl}} \oocat{Mod}_k^\gr.
$$
Here $\oocat{Mod}_k^\mm{cpl}$ denotes the $\infty$-category of complete complexes and $\oocat{Mod}_k^\gr$ denotes the $\infty$-category of graded complexes. Consequently, $\oocat{cLie}$ is a presentable $\infty$-category.
\end{introtheorem}

While we focus on curved Lie algebras in this paper, our framework can equally well handle other types of curved algebras. For example, it applies to curved associative algebras which arise, for instance, from vector bundles with non-flat connections, see Section \ref{sec:Other types of curved algebras}.

Perhaps surprisingly, our framework produces non-trivial results even when restricted to uncurved Lie (or $L_\infty$) algebras, since there can be non-trivial curved morphisms between uncurved objets. In particular, by expressing the Maurer--Cartan space of an $L_\infty$-algebra $\mathfrak g$ as the space of curved maps $0 \rightsquigarrow \mathfrak g$, in Section \ref{sec:rogers} we recover some results due to Dolgushev and Rogers.

We point out that our approach is not directly comparable to \cite{Bellier-MillesDrummond-Cole2020}. While Bellier-Mill\`es and Drummond-Cole also consider filtered objects, their operads \emph{themselves} have a curvature and a pre-differential not squaring to zero. 

Many of the results above  extend naturally if we replace the ground field $k$ by a cdga $A$.
%, at least if we assume that the underlying modules are cofibrant.
However, the curved Lie algebras appearing in Theorem \ref{thm:main intro} do not quite fit into this framework, because their pre-differential is required to interact with the de Rham differential on $\dR(A)$. In other words, it becomes important to view $\dR(A)$ as a \emph{graded mixed} cdga (with weight-grading given by the form degree, as in \cite{CPTVV2017}), and to consider curved Lie algebras that interact with the graded mixed structure.

\medskip

$\rhd$ The goal of \textbf{Section \ref{sec: filtered algebras}} is then to carry out a similar analysis as before, but for curved algebras in modules over a graded mixed cdga $B$. 
The main insight, spelled out in Lemma \ref{lem:operad for mixed-curved Loo over B}, is that mixed-curved $L_\infty$-algebras over such $B$ are also governed by a complete filtered $B$-operad $\cLie_{\infty, B}$, constructed as a distributive law $\cLie_{\infty, B}\cong B\circ \cLie_\infty$.

While $\cLie_{\infty, B}$ is not obtainable as a cobar construction (it is not even augmented), the upshot of Section \ref{sec: filtered algebras} is that we still have (somewhat \emph{ad hoc}) bar-cobar resolutions, $\infty$-morphisms and crucially, a version of the Homotopy Transfer Theorem \ref{thm:htt-B}. This opens the way to a generalization of Theorem \ref{thmC} (Theorem \ref{thm:homotopy theories of curved lie over B}), which allows us to study classical curved Lie algebras over $B$ via a pullback of $\infty$-categories:
$$
\oocat{cLie}_B\simeq \oocat{cLie}_B^\mm{mix}\times_{\oocat{Mod}_B^\mm{cpl}} \oocat{Mod}_{B_{\gr}}^\gr.
$$

$\rhd$ Starting from \textbf{Section \ref{sec:Curved Lie algebroids}}, our goal is to study the homotopy theory of Lie algebroids over a cdga $A$. 
Notice that Lie algebroids over a fixed base are not algebras over an operad, so that the usual methods of constructing a model structure on them do not quite work.
In \cite{nuiten2019homotopicalalgebra}, the third author showed that Lie (or equivalently $L_\infty$) algebroids carry a semi-model structure for which the weak equivalences are quasi-isomorphisms.

In fact,  we can go further than \cite{nuiten2019homotopicalalgebra} and study Lie algebroids which are themselves curved.
Such type of objects have been considered for instance in \cite{baarsmaphd}. %\ricardo{https://dspace.library.uu.nl/handle/1874/386311 ?}\damien{Have a look at Theorem 5.1.1 of this reference.}.
Similar to the previous sections, the $\infty$-category $\oocat{cLie}(A/k)$ of curved $L_\infty$-algebroids over $A$ can be conveniently studied using a mixed variant of curved $L_\infty$-algebroids, which can be organized into a (semi) model category. Most results of Section \ref{sec:Curved Lie algebroids} are extensions of the results of the previous section to Lie algebroids, and can be summarized as follows.
\begin{enumerate}
\item The category of mixed-curved $L_\infty$-algebroids over $A$ carries a semi-model structure whose weak equivalences are $A$-module maps inducing quasi-isomorphisms on the associated graded (Theorem \ref{thm:model structure lie algebroids}).

\item While there are no bar or cobar constructions for curved $L_\infty$-algebroids, there is a ``bar-cobar'' resolution $\Loid\mapsto Q(\Loid)$ on mixed-curved $L_\infty$-algebroids such that structure preserving maps of mixed-curved $L_\infty$-algebroids $Q(\Loid) \to \Loide$ correspond to $\infty$-morphisms $\Loid \rightsquigarrow \Loide$ (Proposition \ref{prop:cobar resolution for loo algebroids}).

\item The association $(d,\ell_1) \mapsto d+\ell_1$ induces an equivalence $\oocat{cLie}(A/k)^\mm{gr-mix} \simeq \oocat{cLie}(A/k) $ from the $\infty$-category of graded mixed-curved Lie algebroids to the $\infty$-category of curved Lie algebroids (Proposition \ref{prop:classical curved Lie algebroid=graded mixed}).
\end{enumerate}
Finally, we also give a description of curved $L_\infty$-algebroids using uncurved objects (which is already interesting for curved $L_\infty$-algebras over the base field $k$):
\begin{introtheorem}[See Theorem \ref{thm:curved Lie algebroids conceptually}]\label{thm:intro curved vs uncurved}
%There is a Rees-like construction\damien{This is not quite a Rees construction (it's Koszul dual to it). I'm afraid this'll confuse the reader} $\mc{R}(T_A)=T_A\ltimes A\wgt{-1}[-1]$, such that there is an equivalence  $$\oocat{cLie}(A/k) \simeq \oocat{Lie}(A/k)^\gr{\scriptstyle / \mc{R}(T_A)}$$
%between the  $\infty$-categories of curved $L_\infty$-algebroids and graded \underline{uncurved} $L_\infty$-algebroids over $\mc{R}(T_A)$.
%
%\joostline{how about this: D: I like it, though I would replace $\hbar$-de Rham complex by Rees algebra of the de Rham complex. }
There is an equivalence of $\infty$-categories
$$\oocat{cLie}(A/k) \simeq \oocat{Lie}(A/k)^\gr{\scriptstyle / \mc{R}(T_A)}$$
between curved $L_\infty$-algebroids and graded \underline{uncurved} $L_\infty$-algebroids over a certain graded Lie algebroid $\mc{R}(T_A)$, whose Chevalley--Eilenberg complex is the Rees algebra of the de Rham complex of $A$.
\end{introtheorem}

\medskip

$\rhd$ Finally, in \textbf{Section \ref{sec:The equivalence between curved Lie algebras and Lie algebroids}} we prove the main theorems.
In fact, we deduce them from a more general result characterizing the category of all curved $L_\infty$-algebroids over complete filtered algebras of the form $C^*(\tang)$, where $\tang\to T_A$ is a complete $L_\infty$-algebroid over $A$ and $C^*(\tang)$ denotes its Chevalley--Eilenberg complex (with the Hodge filtration).
\begin{introtheorem}[See Theorem \ref{thm:comparison}]\label{thm:intro general main}
	Let $A$ be a nonpositively graded cdga and let $\tang$ be a complete $L_\infty$-algebroid over $A$ such that $F^0(\tang)=0$ and each $F^i(\tang)$ is finitely generated quasiprojective as an $A$-module. Then there is an equivalence of $\infty$-categories
	$$\begin{tikzcd}
		\curv\colon \oocat{cLie}(A/k)_{\scriptstyle/\tang}\arrow[r, "\sim"] & \oocat{cLie}_{C^*(\tang)}.
	\end{tikzcd}$$
Taking $\tang = T_A$ the terminal Lie algebroid on $A$ and restricting to uncurved $L_\infty$ algebroids we recover precisely the statement of our main Theorem \ref{thm:main intro}.
\end{introtheorem}
In light of Theorem \ref{thm:intro curved vs uncurved} and the relation between Lie algebroids and formal stacks, this suggests a more geometric interpretation of the $\infty$-category of all curved $L_\infty$-algebras over $\dR(A)$ in terms of formal stacks over the Hodge stack.

\subsection*{Notations and conventions}
Throughout, differentials have degree $1$ and filtrations are decreasing. In the body of the text, in the absence of additional adjectives, all objects are assumed differential graded (dg) by default i.e., they live over the category of cochain complexes over a field $k$ of characteristic zero.
So for instance, when we refer to a Lie algebra, this is synonymous to a dgla, whereas a classical curved Lie algebra will be described as a graded Lie algebra with a degree $+1$ endomorphism and a degree $2$ curvature element satisfying some properties.

By default our operads are unital. The equivalence from unital augmented operads to non-unital operads sending an operad to the kernel of the augmentation map is denoted by  $\PP \mapsto \overline{\PP}$. 
On the other hand, cooperads are by default assumed to be non-counital. Given a cooperad $\CC$, we denote the corresponding counital coaugmented cooperad by $\CC_+ = \CC \oplus I$, see Convention \ref{conv:augmentation}. All other operadic terminology and conventions are in line with \cite{LodayVallette2012}.

In line with the unitality assumptions, $\Sym$ denotes the free unital commutative algebra, i.e. $\Sym V = k \oplus V \oplus V\otimes V \oplus \dots$.

%An operad $\PP$ (resp. cooperad $\CC$) is said to be $1$-reduced if $\PP(0) = \overline{\PP}(1)=0$ (resp. $\CC(0)=\CC(1)=0$).

Everywhere in the paper, $A$ will denote a cdga over $k$ (over which Lie algebroids live), while $B$ will be a graded mixed cdga, whose main example is the de Rham algebra $B=\dR(A)$ equipped with the Hodge filtration.

%Since both curved Lie algebras and Lie algebroids play a prominent but distinct role in this paper, we will reserve a uppercase letter for the first ($\Loid$) and 
%a lowercase letter for the latter ($\mathfrak g$). 

Finally, we use a roman typestyle for ordinary categories and a bold font for $\infty$-categories, while we reserve a sans serif typestyle for named (co)operads.
For instance, $\Lie$ will denote the Lie operad, while $\Alg_\Lie$ will denote the model category of Lie algebras and $\oocat{Alg}_\Lie\simeq \oocat{Alg}_{\Lie_\infty}$ the corresponding $\infty$-category. We will not distinguish between a simplicially enriched category and the corresponding $\infty$-category.

\subsection*{Acknowledgments}
 We thank Joan Bellier-Mill\`es for carefully reading and pointing out several inaccuracies on Section \ref{sec:complete operads} and  Chris Rogers for comments that led to Section \ref{sec:rogers}.
 We also thank Bruno Vallette, Victor Roca, and Jim Stasheff for useful comments on the first version of the paper.
 This project has received funding from the European Research Council (ERC) under the European Union's Horizon 2020 research and innovation programme (grant agreement No 768679) and from the grant ANR-20-CE40-0016 HighAGT.

%{\color{red}
%-------------
%
%To Delete\\
%
%
%Differentials go up.
%
%Objects are dg by default.
%
%Operads (resp. cooperads) are by default non-unital (resp. non-counital).
%
%Everywhere in the paper, $A$ denotes a cofibrant $k$-cdga, which is quasi-free, with generators in non-positive degrees. while $B$ denotes a mixed cdga.\ricardo{Not really, need to decide what to do with it.}
%
%The notation $\Sym$ denotes the free unital commutative algebra, in particular $\Sym_0 V = k$.
%
%\textbf{Dictionary:}
%
%Filtered is assumed decreasing filtration
%
%Mixed implies filtered and means that that there is some other ``pre-differential'' $d'$ (not necessarily squaring to zero) increasing the filtration degree.
%
%Perturbed implies mixed and means that $(d+d')^2=0$.
%
%Graded mixed is mixed but furthermore we're given a splitting of the filtration. In particular the extra pre-differential decomposes $d'= d'_1 + d'_2+...$ where the index represents the increase of the filtration degree.
%
%\textbf{Notation}
%
%$\infty$-cats are mathrm
%
%Lie algebroids are $L$ while curved lie algebras are $\mathfrak g$.
%}

\section{Complete filtered operadic homotopy theory}\label{sec:complete operads}

The goal of this section is to develop the appropriate homotopical framework in which to consider curved Lie algebras. We will do this by studying operads and their algebras in the complete filtered setting: we show that given a filtered operad $\PP$, $\PP$-algebras form a model category and satisfy a form of the Homotopy Transfer Theorem in such a way that the associated $\infty$-category is equivalent to the one of $\PP_\infty$-algebras and $\infty$-morphisms. We then discuss the complete operads $\cLie$ and $\cLie_\infty$ governing respectively mixed-curved Lie algebras and mixed-curved $L_\infty$-algebras, which can be used to study curved Lie algebras in the usual sense.

\subsection{Recollections on filtered complexes}
Given a field $k$ of characteristic zero, a \emph{filtered complex} is a $\mathbb{Z}$-indexed sequence of cochain complexes of $k$-vector spaces and inclusions between them
$$\begin{tikzcd}
\dots \arrow[r, hookrightarrow] & F^1V\arrow[r, hookrightarrow] & F^0V \arrow[r, hookrightarrow] & F^{-1}V\arrow[r, hookrightarrow] & \dots
\end{tikzcd}$$
We will typically denote a filtered complex by its colimit $V\coloneqq \colim_{n\to -\infty} F^nV$ and think of each $F^nV$ as a subcomplex of $V$. A filtered complex $V$ is said to be \emph{nonnegatively filtered} if $F^0V=V$. Given a filtered complex, one can shift its filtration weight by $p$ to get another filtered complex $V\wgt{p}$
$$
F^q V\wgt{p}\coloneqq F^{p+q}V.
$$
We will denote by $\Mod_k^\mm{filt}$ the category of filtered complexes, with maps between them given by maps of cochain complexes preserving filtration weights. 
	
A filtered complex is \emph{complete} if the filtration is complete and Hausdorff, i.e.\ if the map $V \to \displaystyle\lim_{\leftarrow n} V/F^nV$ is an isomorphism. The inclusion of the complete complexes into all filtered complexes admits a left adjoint
\[\begin{tikzcd}
	\widehat{(-)}\colon \cat{Mod}_k^\mm{filt}\arrow[r, yshift=0.4pc] \arrow[r, yshift=-0.4pc, hookleftarrow] & \Mod_k^\mm{cpl}\colon \iota 
\end{tikzcd}\]
sending a filtered complex to its completion. In particular, the category of complete complexes admits all colimits, which are computed in the category of filtered complexes and then completed. For example, infinite coproducts of complete complexes are given by \emph{completed direct sums}. 

A \emph{(weight-)graded complex} is simply a $\mathbb{Z}$-indexed family of complexes $\{V\wgt{p}\}_{p\in \mathbb{Z}}$. There are functors between complete complexes and graded complexes
$$\begin{tikzcd}
\Mod_k^\gr\arrow[r, rightarrow, "\Tot"] & \Mod_k^\mm{cpl}\arrow[r, "\Gr"] & \Mod_k^\gr.
\end{tikzcd}$$
The functor $\Tot$ sends a graded complex to its \emph{total complex}, i.e.\ to
$$
\Tot(V) \coloneqq \widehat{\bigoplus}_p V\wgt{p}, \qquad\qquad F^q\Tot(V) = \prod_{n\leq 0} V\wgt{n-q}.
$$
The second functor takes the associated graded $\Gr^p(V)\coloneqq F^pV\big/F^{p+1}V$. Note that for any filtered complex $V$, the map to its completion $V\rt \widehat{V}$ induces an isomorphism on the associated graded.
\begin{definition}\label{def:admissible mono}
A map of (complete) filtered complexes $V\rt W$ is called a \emph{surjection} if it is a surjection is every filtration weight. It is called an \emph{admissible monomorphism} if each $F^pV\rt F^pW$ is a monomorphism and furthermore, each map 
$$
F^pV\oplus_{F^{p+1}V} F^{p+1}W\rt F^pW
$$ 
is a monomorphism.
\end{definition}
\begin{remark}\label{rem:exact category}
Let us say that $0\rt V\rt W\rt Z\rt 0$ is a \emph{short exact sequence} of complete complexes if it is short exact in each filtration weight. A map is an admissible monomorphism (a surjection) precisely if it is the first (second) map in such a short exact sequence. With this notion of short exact sequence, the category of complete complexes becomes an exact category in the sense of Quillen \cite{QuillenKTheory}.
\end{remark}
\begin{remark}
Note that, even though we are working over a field, not every inclusion is admissible, i.e.\ fits into a short exact sequence: for example, take $k\rt k'$ where the codomain is just $k$, in filtration degree $1$ instead of $0$.
\end{remark}
\begin{remark}\label{rem:surjections have sections}
Suppose that $p\colon W\rt Z$ is a surjection of complete complexes. Without differentials, $p$ admits a section: indeed, without differentials we can simply choose a basis for $Z$. Each basis vector has a certain (maximal) filtration weight, and choosing inverse images with the same weight provides a filtration-preserving section.
\end{remark}
\begin{lemma}\label{lem:graded exact}
The functor $\Gr\colon \Mod_k^\mm{cpl}\rt \Mod_k^\gr$ preserves (infinite) direct sums and products, filtered colimits and is exact (for the exact structure as in Remark \ref{rem:exact category}). In particular, it preserves pushouts along admissible monomorphisms and pullbacks along maps that are surjective in each filtration weight.
\end{lemma}
\begin{proof}
The first part is readily verified and the second part  is true for any functor between Quillen exact categories preserving exact sequences.
\end{proof}
The category of filtered complexes is a closed symmetric monoidal category via the tensor product
\[
F^r(V\otimes W)\coloneqq\sum_{p+q=r}F^pV\otimes F^qW.
\]
One easily sees that the internal mapping object is the filtered complex given in weight $r$ by maps that increase filtration weight by (at most) $r$:
\[
F^r\Hom(V,W) = \big\{\text{filtration preserving maps }V\rt W\wgt{r}\big\}.
\]
We will also refer $\Hom(V, W)$ (with the above filtration) as the \emph{filtered mapping complex}.
\begin{proposition}\label{prop:SMC}
	The category of complete cochain complexes carries a closed symmetric monoidal structure $\botimes$ such that the completion functor $\widehat{(-)}\colon \Mod_k^\mm{filt}\rt \Mod_k^\mm{cpl}$ is symmetric monoidal.
\end{proposition}
\begin{proof}
	This follows from the fact that for any filtered complex $V$ and any complete complex $W$, the filtered mapping complex $\Hom(V, W)$ is itself already complete; indeed, it is a limit
	$$
	\lim\limits_{m\to -\infty}\lim\limits_{n\to \infty} \Hom(F^mV, W/F^nW)
	$$ 
	of complexes with (complete) filtrations vanishing in sufficiently high degrees.
\end{proof}
\begin{remark}
	The category of complexes is a full monoidal subcategory of complete complexes, by endowing a complex $V$ with the trivial filtration $F^{1}V=0$ and $F^rV = V$ for $r\leq 0$. We tend to tacitly view complexes as complete complexes in this way. In particular, $\Mod_k^\mm{cpl}$ is tensored and enriched over cochain complexes; the complex of maps $V\rt W$ is simply $F^0\Hom(V, W)$.
\end{remark}
\begin{remark}\label{rem:graded is monoidal}
The category of (weight-)graded cochain complexes has a similar closed symmetric monoidal structure, where 
$$
(V\otimes W)\wgt{r}=\bigoplus_{p+q=r} V\wgt{p}\otimes W\wgt{q}\quad \text{and}\quad \Hom(V, W)\wgt{p}=\prod_q \Hom(V\wgt{q}, W\wgt{q+p}).
$$
The functors $\Tot\colon \Mod_k^\gr\rt \Mod_k^\mm{cpl}$ and $\Gr\colon \Mod_k^\mm{cpl}\rt \Mod_k^\gr$ are both symmetric monoidal and furthermore preserve internal mapping objects, i.e.\
$$
\Tot\big(\Hom(V, W)\big) \cong \Hom\big(\Tot(V), \Tot(W)\big), \qquad\Gr\big(\Hom(V, W)\big) \cong \Hom\big(\hspace{-1pt}\Gr(V), \Gr(W)\big).
$$
To see the second isomorphism, note that without differential one can decompose $V=\widehat{\bigoplus}_\alpha k\wgt{p_\alpha}[n_\alpha]$ as a completed sum of copies of $k$, in various degrees and filtration weights. Indeed, to do this one simply has to choose a basis for the associated graded of $V$ and lift all basis vectors to $V$ itself. The above isomorphism then takes the form
$$
\Gr\Big(\prod_\alpha W\wgt{-p_\alpha}[-n_\alpha]\Big)\cong \prod_\alpha \Big(\Gr(W)\wgt{-p_\alpha}[-n_\alpha]\Big)
$$
which holds because taking the associated graded preserves products.
\end{remark}
We can then consider operads over this symmetric monoidal category.
\begin{definition}[Complete operads]
	A \emph{complete operad} $\PP$ is an (by default \emph{unital}, symmetric) operad in the category of complete complexes, i.e.\ a unital algebra in symmetric sequences of complete complexes, with respect to the composition product $\bcirc$. Explicitly, $\PP$ comes with composition maps
	$$
	\gamma\colon \big(\PP\bcirc\PP\big)(n)=\widehat{\bigoplus_{k}} \PP(k)\botimes_{\Sigma_k}\left(\widehat{\bigoplus_{i_1+\dots +i_k=n}}\mm{Ind}_{\Sigma_{i_1}\times\dots \times \Sigma{i_k}}^{\Sigma_n}\Big(\PP(i_1) \botimes\dots \botimes\PP(i_k)\Big)\right)\rt \PP(n)
	$$
	from a completed direct sum of completed tensor products.
	Given a complete operad $\PP$, a \emph{complete $\PP$-algebra} is a complete complex $A$ equipped with a map $\PP\bcirc A\rt A$ satisfying the usual associativity and unitality conditions.
\end{definition}
\begin{remark}[Filtered operads]
In a similar way, one can define (not necessarily complete) filtered operads and algebras over them, using the category of symmetric sequences of filtered complexes, equipped with the (non-completed) composition product $\circ$. A complete operad is then equivalently a filtered operad whose underlying symmetric sequence is complete: the structure map $\PP\circ\PP\rt \PP$ then extends uniquely to the completion.

Likewise, if $\PP$ is a filtered operad, then there is an equivalence between filtered $\PP$-algebras $A$ whose underlying filtered complex is complete, and complete algebras over the completion $\widehat{\PP}$: indeed, the structure map $\PP\circ A\rt A$ extends uniquely to the completion $\widehat{\PP}\bcirc A\rt A$.
\end{remark}
\begin{example}
	It will be important (particularly in Section \ref{sec:HTT}) that due to Proposition \ref{prop:SMC} we can define the \emph{endomorphisms operad} of a filtered complex $A$ to be the filtered operad $\End_A$ given by $\End_A(n) \coloneqq \Hom(A^{\otimes n}, A)$, such that filtered algebras over $\PP$ can be identified with filtered operad maps $\PP\rt \End_A$. When $A$ is complete, this is a complete operad, which is isomorphic to the complete endomorphism operad with operations $A\botimes\dots \botimes A\rt A$.
\end{example}

Since we consider objects with non-trivial arity $0$ pieces, we need to be precise about the version of cooperads we consider, namely in what concerns partial vs total cocompositions and conilpotence.

Recall that given a symmetric sequence $\CC$ with partial cocompositions $\Delta_i\colon \CC(k)\to \CC(m)\otimes \CC(k-n+1), i=1,\dots,m$ (appropriately $\Sigma$-equivariant and coassociative), one can iterate them to obtain a total cocomposition $\Delta \colon \CC(k) \to \prod_{i_1+\dots +i_n=k} \CC(n)\otimes \CC(i_1)\otimes \dots\otimes \CC(i_n)$, making $\CC$ a coalgebra with respect to $\circ$, as the argument dual to \cite[Proposition 5.3.8]{LodayVallette2012} shows. Notice however that in the absence of a unit one cannot obtain the partial cocomposition from the total one. 

\begin{definition}\label{def:complete cooperad}
	A \emph{complete cooperad} $\CC$ is a symmetric sequence of complete complexes equipped with (by default \emph{non-counital}) $\Sigma$-equivariant coassociative partial cocompositions $\Delta_i\colon \CC(k)\to \CC(m)\otimes \CC(k-n+1), i=1,\dots,m$. In particular, $\CC$ is a coalgebra in symmetric sequences of complete complexes with respect to the completed composition product $\bcirc$.

A \emph{(conilpotent) $\CC$-coalgebra} is a complete complex $C$ together with a map $C\rt \CC\bcirc C$ satisfying the usual associativity constraint.
	
	A complete cooperad $\CC$ is said to be \emph{complete conilpotent} if the $n$-fold cocomposition determines a map
$$
		(\Delta, \Delta^2, \Delta^3, \dots)\colon \cat{C}\rt \widehat{\bigoplus_{n\geq 2}}\, \cat{C}\bcirc \dots \bcirc \cat{C}
$$
	into the \emph{completed sum} of completed composition products. In other words, an operation in $\cat{C}(n)$ can be decomposed into infinitely many trees, but their sum converges with respect to the filtration.
\end{definition}
\begin{warning}\label{war:complete conilpotence}
	Recall that (strict) conilpotency in the usual sense is the requirement that the coradical filtration be exhaustive, $\colim_n \mathrm{corad}^n(\CC) = \CC$.	
	The condition of complete conilpotency  is \emph{weaker} than conilpotency in the usual sense: for each element $c\in \CC$ and $r\geq 0$, there exists an $n$ such that $\Delta^n(c)$ is of filtration degree $r$, but $\Delta^n(c)$ need not vanish for large $n$. However, note that the associated graded $\Gr(\CC)$ is a cooperad in weight-graded complexes which \emph{is} conilpotent in the usual sense.
\end{warning}

%Let us emphasize that by convention, all complete cooperads $\CC$ are \emph{conilpotent} in the sense that their iterated decompositions define a map of the form \eqref{diag:decomposition}. 

\begin{convention}\label{conv:augmentation}[Augmented and non-unital operads]
There is an equivalence of categories between unital augmented operads and non-unital operads, given by quotienting out the unit in arity one. We will denote this construction by $\PP \mapsto \overline{\PP}$.
We take the convention that our operads are unital unless otherwise specified.
On the other hand, we take the convention that cooperads are non-counital.
 Given a cooperad $\CC$, we denote the corresponding counital coaugmented cooperad by $\CC_+ = \CC \oplus I$.
	
The reason for this choice is that non-unitaly is convenient to define conilpotent cooperads, and for the constructions in Section \ref{sec:bar-cobar}, but slightly inconvenient when talking about \emph{algebras} over operads: indeed, when $\PP$ is a non-unital complete operad and $A$ is a complete complex, the free $\PP$-algebra on $A$ is given by $\PP\circ A$, not $\overline{\PP}\circ A$.
\end{convention}

\subsection{Model structures on filtered complexes and algebras}
In this section we show that the category of complete complexes can be endowed with a model structure whose weak equivalences are maps inducing a quasi-isomorphism at the level of the associated graded. Furthermore, this model structure transfers to a model structure on algebras over operads:
\begin{theorem}\label{thm:Model str}
	Let $\PP$ be a complete operad. The category of complete algebras over $\PP$ admits a cofibrantly generated model category structure such that:
	\begin{itemize}
		\item Weak equivalences are maps inducing quasi-isomorphisms on the associated graded.
		\item Fibrations are maps that induce surjections in each filtration weight $p$.%\ricardo{Which is equivalent to requiring it only on the associated graded. Do we have a preference for one statement? Does this change this proof?}
	\end{itemize}
	In particular every $\PP$-algebra is fibrant.
\end{theorem}
In particular, taking $\PP$ to be the trivial operad gives a model structure on complete complexes. 
\begin{proof}
The category of (not necessarily complete) filtered complexes admits a cofibrantly generated model structure in which weak equivalences (resp.\ fibrations) are graded quasi-isomorphisms (resp.\ surjections in each filtration degree): this is the special case of \cite[Theorem 3.14]{cirici2019model} where $r=0$. 
Now consider the adjoint pair
$$\begin{tikzcd}
\widehat{\mm{Free}}\colon \cat{Mod}_k^\mm{filt}\arrow[r, yshift=0.4pc] & \cat{Alg}_{\PP}^\mm{cpl}\colon \mm{forget}\arrow[l, yshift=-0.4pc]
\end{tikzcd}$$ 
	between complete $\PP$-algebras and (not necessarily complete) filtered complexes. To check that the model structure transfers along this adjunction, it suffices to provide a functorial path object in complete $\PP$-algebras. This is just the classical argument from Hinich \cite{hinich1997homological}: if $A$ is a complete $\PP$-algebra, then $A\botimes \Omega[\Delta^1]$ is a complete $\PP$-algebra as well, where $\Omega[\Delta^1]$ is considered as a commutative algebra with the trivial filtration. This factors the diagonal as 
	$$
	A\rt A\botimes\Omega[\Delta^1]\rt A\times A.
	$$
	Since the associated graded functor from complete complexes to graded complexes is symmetric monoidal, one finds a factorization of graded algebras $\Gr(A)\rt \Gr(A)\otimes \Omega[\Delta^1]\rt \Gr(A)\times \Gr(A)$, where $\Omega[\Delta^1]$ is in weight $0$. This is clearly a weak equivalence, followed by a surjection.
\end{proof}
\begin{lemma}\label{lem:fibrations and w.e. on associated graded}
Let $f\colon V\rt W$ be a map of complete complexes.
\begin{enumerate}
\item $f$ induces a quasi-isomorphism on the associated graded if and only if it induces a quasi-isomorphism in each filtration degree.
\item $f$ induces a surjection on the associated graded if and only if it induces a surjection in each filtration degree.
\end{enumerate}
\end{lemma}
\begin{proof}
In both cases, the `if' part is immediate. For (1), let $C$ denote the mapping cone of $f$. Since taking the associated graded commutes with taking mapping cones, the associated graded of $C$ is acyclic. In particular, the sequence $\dots\hookrightarrow F^pC\hookrightarrow F^{p-1}C\hookrightarrow \dots$ consists of acyclic cofibrations, so that each inclusion $F^pC\hookrightarrow C$ into the colimit is an acyclic cofibration. Consequently, each $C\big/F^pC$ is acyclic, so the completion $C=\holim C\big/F^pC$ is acyclic as well. This implies that each $F^pC$ is acyclic, so that $f$ is a quasi-isomorphism in each degree.

For (2), by induction on $q\geq 1$ using the snake lemma, one sees that $f$ fits into short exact sequences $0\to K^q\rt F^pV/F^{p+q}V\rt F^pW/F^{p+q}W\to 0$ such that the natural map $K^{q+1}\rt K^q$ is surjective. Taking the limit $q\to \infty$ and using that $F^pV$ and $F^pW$ are complete, one then obtains an exact sequence $F^pV\rt F^pW\rt \smash{\lim\limits_{\leftarrow q}{}^1} K^q=0$. This implies that $f$ is surjective in each filtration degree.
\end{proof}

\begin{lemma}\label{lem:cofibrations of complete complexes}
A map of complete complexes (over a field $k$) $V\rt W$ is a cofibration if and only if it is an admissible monomorphism (Definition \ref{def:admissible mono}).
\end{lemma}
\begin{proof}
Suppose that $V\rt W$ is a cofibration and let $V[0, 1]$ denote the mapping cone of $V$. Using the lifting property against the trivial fibration $V[0,1]\rt 0$, one sees that without differential, $V\rt W$ is given by a summand inclusion $W\cong V\oplus W/V$. Note that by Remark \ref{rem:surjections have sections}, such summand inclusions are exactly the admissible monomorphisms.

Conversely, suppose that $V\rt W$ is an admissible monomorphism, i.e.\ a summand inclusion without the differential. Let $p\colon Y\rt X$ be an acyclic fibration of complete complexes, with fiber $Z$. We then have a short exact sequence of filtered mapping complexes
	$$
	0\rt \Hom(W/V, Z)\rt \Hom(W, Y)\rt \Hom(V, Y)\times_{\Hom(V, X)} \Hom(W, X)\rt 0
	$$
	and we have to check that the right map is a trivial fibration (in filtration degree $0$). Using Lemma \ref{lem:fibrations and w.e. on associated graded}, it suffices to verify that the associated graded of $\Hom(W/V, Z)$ is acyclic. By Remark \ref{rem:graded is monoidal}, it suffices to verify that $\Hom(\Gr(W/V), \Gr(Z))$ is acyclic, which is immediate because $Z$ was graded-acyclic.
\end{proof}
\begin{remark}\label{rem:model for derived complete}
Consider the category $\Fun(\mathbb{Z}, \Mod_k)$ of sequences of cochain complexes $\dots \rt F^1V\rt F^0V\rt \dots$, equipped with the projective model structure. One can verify that an object is cofibrant in this model structure if and only if it is a filtered complex.
The obvious fully faithful inclusion $\Mod_k^\mm{cpl}\hookrightarrow \Fun(\mathbb{Z}, \Mod_k)$ is a right Quillen functor, which furthermore preserves cofibrant objects. In particular, it induces a fully faithful functor of $\infty$-categories. The essential image can be seen to consist of those sequences of complexes such that $\holim_{i\rt \infty} F^iV\simeq 0$. Consequently, $\Mod_k^\mm{cpl}$ is a model for the $\infty$-category of (derived) complete complexes.
\end{remark}
\begin{corollary}\label{cor:monoidal model}
	The model structure on complete cochain complexes is monoidal model with respect to the completed tensor product of Proposition \ref{prop:SMC}. Furthermore, the functor $V\hat\otimes (-)$ preserves graded-quasi isomorphisms for any object $V$.
\end{corollary}
\begin{proof}
	It suffices to verify the pushout-product axiom. Let $V_1\rt W_1$ and $V_2\rt W_2$ be two cofibrations. Using Lemma \ref{lem:cofibrations of complete complexes}, their pushout-product is a summand inclusion without differential, hence a cofibration, whose cokernel is just $W_1/V_1\otimes W_2/V_2$. In particular, if one of the two maps is furthermore a graded quasi-isomorphism, then this cokernel is graded acyclic.
\end{proof}
\begin{remark}
As a consequence of Corollary \ref{cor:monoidal model}, one finds that for any complete dg-algebra $B$, the model category of complete $B$-modules is tensored over complete complexes (via the tensor product $\botimes_k$). In particular, for any two complete $B$-modules $M, N$, there is complete mapping complex $\Hom_B(M, N)$, given in filtration degree $p$ by the maps of $B$-modules increasing filtration weight by (at most) $p$.
\end{remark}
In fact, Lemma \ref{lem:cofibrations of complete complexes} has an analogue for complete modules over a complete dg-algebra $B$. To this end, let us recall the following terminology:
\begin{definition}
Let $B$ be a complete dg-algebra over $k$. A complete $B$-module $M$ is called \emph{quasiprojective} if without differential, it is the retract of a free complete $B$-module $B\botimes_k V$.
\end{definition}
Since taking the associated graded is symmetric monoidal (Remark \ref{rem:graded is monoidal}), every complete $B$-module has an underlying weight-graded module over the weight-graded dg-algebra $\Gr(B)$. The category of such weight-graded modules admits a model structure, whose weak equivalences (fibrations) are quasi-isomorphisms (surjections) of complexes in each weight. Using this we have the following characterization of cofibrations of complete $B$-modules in terms of their associated graded:
\begin{proposition}\label{prop:cofibrations of complete modules}
Let $B$ be a complete dg-algebra and $f\colon M\rt N$ a map of complete $B$-modules. Then $f$ is a cofibration if and only if it is an admissible monomorphism with quasiprojective cokernel $N\big/M$, such that $\Gr(N\big/M)$ is a cofibrant module over $\Gr(B)$.
\end{proposition}
\begin{proof}
Let us first prove that all cofibrations indeed have the listed properties. By the small object argument, $f\colon M\rt N$ is a retract of a transfinite composition of pushouts of maps of the form $B\botimes_k V\rt B\botimes_k W$, where $V\rt W$ is an admissible monomorphism of filtered complexes. Such maps are themselves admissible monomorphisms, with quasiprojective cokernel. Furthermore, Lemma \ref{lem:graded exact} and Remark \ref{rem:graded is monoidal} imply that taking the associated graded sends this to the retract of a transfinite composition of pushouts of maps $\Gr(B)\otimes \Gr(V)\rt \Gr(B)\otimes \Gr(W)$; such a map is a cofibration of weight-graded modules over $\Gr(B)$. It follows that the cokernel $\Gr(N\big/M)\cong \Gr(N)\big/\Gr(M)$ is a cofibrant module over $\Gr(B)$.

Conversely, suppose that $f$ has the listed properties and let $p\colon Y\rt X$ be an acyclic fibration of complete $B$-modules, with kernel $B$. Since the cokernel of $f$ is quasiprojective, there exists a splitting $N\cong M\oplus N\big/M$ without differentials. This implies that there is a short exact sequence of $B$-linear filtered mapping complexes
	$$
	0\rt \Hom_B(N/M, Z)\rt \Hom_B(N, Y)\rt \Hom_B(M, Y)\times_{\Hom_B(M, X)} \Hom_B(N, X)\rt 0.
	$$
	We have to check that the right map is a trivial fibration in filtration degree $0$. By Lemma \ref{lem:fibrations and w.e. on associated graded}, it suffices to verify that the associated graded of $\Hom_B(N/M, Z)$ is acyclic. An argument very similar to Remark \ref{rem:graded is monoidal}, using that $N/M$ is quasiprojective to write it as a retract of a completed sum of shifted copies of $B$, shows that there is an isomorphism
	$$
	\Gr\Big(\Hom_B\big(N/M, Z\big)\Big)\cong \Hom_{\Gr(B)}\big(\Gr(N/M), \Gr(Z)\big).
	$$
	Since $\Gr(Z)$ is acyclic and $\Gr(N/M)$ is a cofibrant module over $\Gr(B)$ by assumption, we conclude that $\Hom_B(N/M, Z)$ is indeed graded acyclic.
\end{proof}

\subsubsection{Complete filtered complexes versus graded mixed complexes}\label{sec:graded mixed}
Recall that by definition, the associated graded functor $\Gr\colon \Mod^{\mm{cpl}}_k\rt \Mod_k^{\gr}$ preserves and detects weak equivalences. Since it preserves exact sequences and all direct sums and direct products, the induced functor between (stable) $\infty$-categories preserves limits and colimits and detects equivalences. It follows formally from this that one can identify the $\infty$-category of complete complexes with algebras in weight-graded complexes over a certain monad. In fact, there is a well-known way to describe complete filtered complexes concretely in terms of weight-graded complexes with additional algebraic structure:
\begin{definition}\label{def:graded mixed complex}
A \emph{graded mixed complex} is a weight graded complex $V$, equipped with operations $\delta_k\colon V\rt V$ for $k\geq 1$ of weight $k$ and degree $1$, such that
$$
d\circ \delta_k + \delta_k\circ d + \sum_{i+j=k} \delta_i\delta_j=0.
$$
We will write $\Mod_k^{\mm{gr-mix}}$ for the category of graded mixed complexes and maps between them that preserve the weights and 
strictly commute with the operations $\delta_k$.
\end{definition}
\begin{remark}\label{remark:strict graded mixed complexes}
The above definition of a graded mixed complex is also known as a (shifted) multicomplex, and differs from the graded mixed complexes appearing in e.g.\ \cite{CPTVV2017}, where instead the \emph{strict} notion of graded mixed complex is used, corresponding to the situation where $\delta_k=0$ for $k\geq 2$. 
%Second, the operations $\delta_k$ are taken to be of cohomological degree $-1$ (and weight $k$); this comes down to replacing $V$ by the weight-graded complex obtained by shifting the degree up by $2k$ in weight $k$.\damien{The second point is not accurate (this was in PTVV, but in CPTVV it is the same weight grading and degree convention as ours. }
In fact, the inclusion of the category of strict graded mixed complexes into the one of graded mixed complexes becomes an equivalence of $\infty$-categories 
after localizing at the weak equivalences (i.e.\ weightwise quasi-isomorphisms). Indeed, strict 
graded mixed complexes are weight graded modules over $k\langle \delta_1\rangle/\delta_1^2$, where $\delta_1$ has both degree $1$ and weight $1$, while 
graded complexes are weight graded modules over its quasi-free resolution $k\langle \delta_i|i\leq 1\rangle$, where $\delta_i$ has degree $1$ and weight $i$, and 
$$
d(\delta_k)=-\sum_{i+j=k} \delta_i\delta_j.
$$ 
\end{remark}
\begin{definition}\label{def:infty-morphisms of graded mixed complex}
An $\infty$-morphism between two graded mixed complexes, denoted by a wiggly arrow $V\rightsquigarrow W$, is a collection of maps $\varphi_k\colon V\rt W$ for $k\geq 0$ 
of weight $k$ and degree $0$, such that
$$
\sum_{i+j=k}(\delta_i\varphi_j+\varphi_i\delta_j)=0,
$$
where we use the convention that $\delta_0=d$. We will write $\Mod_k^{\mm{gr-mix},\infty}$ for the category of graded mixed complexes and $\infty$-morphisms between them.
\end{definition}
\begin{remark}\label{remark:almost model}
The category $\Mod_k^{\mm{gr-mix},\infty}$ is almost a model category (see e.g.\ \cite[\S4.1]{vallette2014homotopy}), in the sense that all axioms 
but the bicompleteness one are satistfied, though finite products and pullbacks of fibrations exist: an $\infty$-morphism $\varphi=(\varphi_k)_{k\geq0}$ is 
a (co)fibration (resp.~a weak equivalence) if $\varphi_0$ is a (co)fibration (resp.~a weak equivalence). One easily sees that every object is then both fibrant and cofibrant. 
Moreover, one can actually prove that the faithful (but not fully faithful) functor $\Mod_k^{\mm{gr-mix}}\to\Mod_k^{\mm{gr-mix},\infty}$ given by the identity on 
objects induces an equivalence of $\infty$-categories after localizing at weak equivalences. 
\end{remark}
Led by the above Remark \ref{remark:almost model}, we define the $\infty$-category $\scat{Mod}_k^{\mm{gr-mix}}$ as the simplicial category 
whose objects are graded mixed complexes, and with $n$-simplices in the space of morphisms from $V$ to $W$ being $\infty$-morphisms 
$V\rightsquigarrow W\otimes\Omega[\Delta^n]$. It then follows that the 
$\infty$-functors $\Mod_k^{\mm{gr-mix}}[\text{w.e.}^{-1}]\to \Mod_k^{\mm{gr-mix},\infty}[\text{w.e.}^{-1}]\to \scat{Mod}_k^{\mm{gr-mix}}$ are equivalences. 

\medskip

If $V$ is a graded mixed complex, then the total complex $\Tot(V)$ comes equipped with a differential $d_{\tot}=d+\sum_{k\geq 1} \delta_k$. 
Moreover, every $n$-simplex of $\infty$-morphisms $\varphi=(\varphi_k)_{k\geq0}\colon V\rightsquigarrow W\otimes\Omega[\Delta^n]$ leads to an $n$-simplex of 
filtered morphisms $\varphi_{\tot}=\sum_{k\geq0}\varphi_k\colon \Tot(V)\rightsquigarrow \Tot(W)\otimes\Omega[\Delta^n]$. Hence we have a functor 
$$\begin{tikzcd}
\Tot\colon \scat{Mod}_k^{\mm{gr-mix}}\arrow[r] & \scat{Mod}_k^\mm{cpl}.
\end{tikzcd}$$
between simplicial categories.
\begin{proposition}\label{prop:graded mixed complexes}
The functor $\Tot\colon  \scat{Mod}_k^{\mm{gr-mix}}\rt \scat{Mod}_k^\mm{cpl}$ is an equivalence of $\infty$-categories.
\end{proposition}
This result, well-known to experts, provides the blueprint for our discussion in Section \ref{sec:curved algebras}, where we show how curved (filtered) 
$L_\infty$-algebras are equivalent to a kind of `graded mixed-curved $L_\infty$-algebra'. We will also recover it in Example \ref{ex:graded mixed} from 
an operadic perspective.
\begin{proof}
The $\infty$-functor is fully faithful by definition, hence we just have to prove that it is essentially surjective. For every complete filtered complex $W$, 
we can choose a splitting of the filtration on $W$ (without differential) and obtain an isomorphism of complete filtered graded vector spaces $W\cong \Tot(\Gr(W))$. 
Decomposing the differential on $W$ into its homogeneous components of weight $k$ as $d_W=d+\sum_{k\geq 1} \delta_k$, this determines a graded mixed 
structure on $\Gr(W)$, which is such that $W\cong\Tot\big(\Gr(W),\delta_1,\dots)$ as complete filtered complexes. 
\end{proof}
%In other words, taking total complexes and splitting the filtration allows one to pass between graded mixed complexes and filtered complexes. 
%However, notice that this does not respect (strict) morphisms: a morphism of complete filtered complexes may not be compatible with the choice of splitting. 

%This issue disappears at the homotopical level; more precisely, the category of graded mixed complexes carries a model structure, where a map is a weak equivalence if the underlying map of weight-graded complexes is a quasi-isomorphism and a fibration if it is a surjection. Taking total complexes then refines to a right Quillen functor

\subsection{Bar, cobar, and twisting morphisms}\label{sec:bar-cobar}

In this section we will see that the classical bar and cobar constructions between operads and conilpotent cooperads, as well as the twisting morphism yoga generalize in a fairly straightforward way to the complete setting. A closely related discussion appears in \cite[Chapter 2]{DotsenkoShadrinVallette2018}. Recall from \ref{conv:augmentation} that cooperads are not assumed to be counital.
%All propositions/definitions are referred to \cite{LodayVallette2012}, since the proofs in the filtered case are obvious adaptations from the ordinary case.

%However, we wish to apply these constructions to cooperads such as $\ucocom$ of counital cocommutative coalgebras which is not conilpotent. In the second part of this section we show that there are refinements of these constructions in the filtered case that cover a larger category of cooperads.
%
%
%\subsubsection{The case of conilpotent cooperads}

\begin{propdef}
	Let $\PP$ be a complete augmented operad and $\CC$ be a complete cooperad.
	
	\begin{enumerate}
		\item The \emph{convolution Lie algebra} of $\CC$ and $\PP$ is the complete Lie algebra 
		$$\mathfrak g = \prod_{n\geq 0} \Hom(\CC(n), \overline{\PP}(n))$$
		with filtration induced by the internal $\Hom$. The bracket is induced from the pre-Lie product $f\star g = \gamma_{(1)}^\PP\circ (f\hat \otimes g)\circ \Delta_{(1)}^\CC$, where the indices $_{(1)}$ indicate infinitesimal (co)composition.
		
		\item A \emph{twisting morphism} $\phi \colon \CC \to \PP$ is a Maurer--Cartan element of the Lie algebra $\mathfrak g$ (in particular, it takes values in $\ol{\PP}$).
		 The set of twisting morphisms is denoted $\Tw(\CC,\PP) = \{f\in F^0\mathfrak g_1 \ | \ \partial f + \frac{1}{2}[f,f]=0\}$. 
	\end{enumerate}
	
\end{propdef}

\begin{proof}
	It is easy to see that the filtrations are preserved by the Lie bracket and
	$\mathfrak g$ is complete since the direct product preserves completeness.
	See \cite[Section 6.4]{LodayVallette2012} for a treatment of twisting morphisms in the unfiltered case.
\end{proof}

Let $E$ be a complete symmetric sequence. The free operad generated by $E$, denoted $\mathcal T(E)$ is the completion of the vector space spanned by trees labeled by elements of $E$, where the filtration level of an $E$-labeled tree the sum of the filtration levels of each vertex. Composition is given by grafting trees.

Similarly, the cofree complete conilpotent cooperad on $E$ is denoted by $\mathcal T^c(E)$;  it differs only from $\mathcal T(E)$ by the unit in arity $1$, and has cocomposition given by ungrafting trees.

% The filtration level on a tree in the free operad $\mathcal T(E)$ is obtained by summing over the filtration levels of every vertex. 

\begin{propdef}\label{propdef:bar operad}
	The \emph{bar construction} of a complete operad $\PP$ is the complete (non-unital) conilpotent cooperad $$\Bar{\PP}=(\mathcal T^c(\overline{\PP}[1]),  d_{\PP}+d_{\gamma^{(1)}_\PP})$$ with the filtration induced by the cofree conilpotent cooperad functor and differential arising from the differential on $\PP$ and the bar differential, contracting edges of trees. 
	
	Similarly, the \emph{cobar construction} of a complete cooperad $\CC$ is the complete (unital augmented) operad
	$$\Omega{\CC}=(\mathcal T(\CC[-1]), d_{\CC}+d_{\Delta^{(1)}_\CC}).$$
	These functors form an adjoint pair $\Omega\colon \cat{Coop}^\mm{conil}\leftrightarrows \cat{Op}^{\mm{aug}} \colon \Bar$ between complete augmented operads and \emph{complete conilpotent} cooperads. Furthermore the counit of the adjunction $\Omega \Bar \stackrel{\sim}\Rightarrow \mathrm{id}_{\cat{Op}^{\mm{aug}}}$ is a weak equivalence. 
\end{propdef}
\begin{proof}
	The proof of the adjunction follows from showing that there are natural bijections $\Hom_{\text{Op}^\mm{aug}}(\Omega \CC, \PP) \cong \Tw(\CC, \PP) \cong \Hom_{\text{Coop}^\mm{conil}}(\CC, \Bar \PP)$, which is a straightforward adaptation from the unfiltered case \cite[Theorem 6.5.10]{LodayVallette2012} (in fact, the left bijection also exists when $\CC$ is not conilpotent).

	For the second part, notice that the functor associated graded commutes with the bar and cobar constructions (since it preserves tensor products).
	Ignoring degrees, elements of $ \Omega \Bar(\Gr(\PP))$ can be seen as trees whose vertices are themselves (``inner'') trees whose vertices are labeled by $\PP$. Taking a second filtration by the number of inner edges (which is the bar filtration) we recover at the level of the associated graded only the piece of the differential corresponding to the one from $\mathscr P$ and a second one making an inner edge into an outer edge. 
	One checks that the associated graded retracts into $\mathscr P$ by constructing a homotopy that makes an outer edge into an inner edge (cf.\ \cite[Proposition 3.1.12]{Fresse2004} for the unfiltered case).
\end{proof}
Notice that, in addition to the complete filtration coming from $\CC$, $\Omega\CC$ admits another (decreasing) filtration given by the number of vertices in $\mathcal T \CC$. This filtration will be referred to as the \emph{cobar filtration}.
Dually, the number of vertices in $\mathcal T^c \PP$ induces an exhausting increasing filtration on  $\Bar\PP$ that will be called the \emph{bar filtration}. Taking the associated spectral sequence one can show that $\Bar$ preserves weak equivalences of complete operads. 
\begin{definition}
A twisting morphism $\CC\to \PP$ is said to be \emph{Koszul} if the induced map $\Omega \CC\to \PP$ is a weak equivalence of complete operads.
\end{definition}

In particular, the projection $\Bar \PP \to \PP$ is a Koszul twisting morphism by Proposition/definition \ref{propdef:bar operad}. Recall that a twisting morphism $\CC\to \PP$ gives rise to a bar and cobar construction at the level of (co)algebras:

\begin{propdef}\label{propdef:bar algebras}
	Let $\phi \colon \CC \to \PP$ be a Koszul twisting morphism between complete (co)operads. 
	The \emph{bar construction} of a $\PP$-algebra $A$, denoted $\Bar_\phi A$  and the \emph{cobar construction} of a conilpotent $\CC$-coalgebra $C$, denoted $\Omega_\phi C$ are the quasi-free (co)algebras 
	\[
	\Bar_\phi A = (\CC_+\circ A, d_A+d_\phi) \qquad\qquad\qquad \Omega_\phi C = (\PP \circ C, d_C+d^\phi).
	\]
	See Convention \ref{conv:augmentation} about free algebras and our convention conerning unitality. The differentials are induced by those on $A$ and $C$, together with the (co)bar differentials as in \cite[Section 11.2]{LodayVallette2012}. An \emph{$\infty$-morphism} between two $\PP$-algebras% (resp $\CC$-coalgebras)
	, denoted by a wiggly arrow $\rightsquigarrow$, is by definition a morphism of $\CC$-coalgebras
	% ($\PP$-algebras) 
	between the respective bar 
	%(resp. cobar) 
	constructions.
	
	These (co)bar constructions define an adjoint pair $\Omega_\phi\colon \cat{Coalg}_{\CC}\leftrightarrows \cat{Alg}_{\PP}\colon \Bar_\phi$, and the counit of the adjunction $\Omega_\phi \Bar_\phi A\to A$ is a weak equivalence if $\CC$ is a complete conilpotent cooperad.
\end{propdef}

\begin{proof}
	The functors are adjoint by the same argument as \cite[Proposition 11.3.2]{LodayVallette2012}. For $\Omega_\phi \Bar_\phi A\to A$ being a weak equivalence, we argue as in the unfiltered case \cite[Theorem 11.3.6]{LodayVallette2012}, but unlike there we require a proof that does not used that $\CC$ and $\PP$ are connected weight graded. Since the functor $\Gr\colon \Mod_k^\mm{filt}\rt \Mod_k^{\gr}$ preserves colimits and tensor products and detects weak equivalences, it suffices to prove this at the graded level (so we can forget about filtrations, while the weight-grading will play no role); in particular, $\CC$ is now conilpotent in the usual sense (see Warning \ref{war:complete conilpotence}).
		
	We start by showing that for the universal Koszul twisting morphism $\iota\colon \CC \to \Omega \CC$, the map  $\Omega_\iota \Bar_\iota A\to A$  is a weak equivalence.
	Ignoring differentials, $\Omega_\iota \Bar_\iota A$ takes the form  $\Omega \CC \circ \CC_+ \circ A$. One can take a filtration on $\Omega_\iota \Bar_\iota A$ given by the sum of the coradical filtrations on all $\CC_+$ and $\CC[-1]$ pieces appearing. On the associated graded, the only non-internal piece of the differential that survives is the counital part that takes an element $p$ in one of the $\CC_+$ pieces, replaces it by $p\circ 1$ and ``moves''  $p$ to the $\Omega \CC$ while increasing its degree by $1$. 
		There is a natural contracting homotopy to this differential that takes any rightmost $c\in \CC[-1] \subset \Omega\CC$ connected only to units $1\in \CC_+$ and moves it to the $\CC_+$ side. It follows that the only surviving piece corresponds to $k\circ k\circ A = A$.
	
	\medskip
	
	Secondly, let us show that the map $\Omega_\iota \Bar_\iota A\to \Omega_\phi \Bar_\phi A$ is a weak equivalence. 
	Ignoring differentials, this corresponds to showing that the map $\Omega\CC \to \PP$ induces an equivalence $\Omega \CC \circ \CC_+ \circ A \stackrel{\sim}{\rt} \PP \circ \CC_+ \circ A$. 
	This time, one takes a filtration consisting of the total coradical filtration on the $\CC_+$ part (ignoring the $\Omega\CC$ and $\PP$ pieces). Now, on the associated graded, we obtain precisely $\Omega \CC \circ \CC_+ \circ A \stackrel{\sim}{\rt} \PP \circ \CC_+ \circ A$ with only the internal differentials, which is a quasi-isomorphism since $\phi$ is Koszul.
	
	\medskip
	
	Finally the result follows from the 2-out-of-3 property since the map $\Omega_\iota \Bar_\iota A\to A$  is precisely the composite $\Omega_\iota \Bar_\iota A\to \Omega_\phi \Bar_\phi A\to  A$.
\end{proof}

\subsection{Homotopy Transfer Theorem}\label{sec:HTT}

In this section we prove a version of the Homotopy Transfer Theorem for complete complexes. The proof of the theorem itself is fairly standard and does not actually make use of the completeness of the filtered complexes. However, to obtain some classical consequences such as the construction of higher Massey products on the homology of an algebra there are some obstructions coming from the underlying category of filtered vector spaces. As we will see, essentially all obstructions vanish under the assumption that the filtered complexes involved are complete.

We start by noticing that the usual notion of homotopy equivalence of cochain complexes extends naturally to the filtered setting.

\begin{definition}
	A map $f\colon W \rt V$ between complete complexes is a \emph{filtered homotopy equivalence} if there exists a map of filtered complexes $g\colon V \to W$ and filtration-preserving homotopies $h_V\colon V \rt V$, $h_W\colon W \rt W$ of degree $-1$ such that 
	\[
	dh_V + h_Vd = \id_V - f \circ g\qquad\text{and}\qquad dh_W + h_Wd = \id_W - g \circ f.
	\]
A \emph{ homotopy retract} consists of filtered maps $i,p$ of degree $0$ and $h$ of degree $1$
\begin{equation*}
\begin{tikzcd}[cells={nodes={}}]
V \arrow[yshift=-.5ex,swap]{r}{i}  
& W \arrow[yshift=.5ex,swap]{l}{p} 
\arrow[loop right, distance=2em, 
%start anchor={[yshift=1ex]east}, end anchor={[yshift=-1ex]east}
]{}{h} 
\end{tikzcd}
\end{equation*}
such that $ip-\operatorname{id}_W= [d_W,h]$. It is called a \emph{deformation retract} if furthermore $pi=\operatorname{id}_V$.
\end{definition}
\begin{lemma}\label{lem:acyclic fib is def retract}
Let $p\colon W\twoheadrightarrow V$ be an acyclic fibration of complete complexes. Then $p$ is part of a deformation retract. Furthermore, the homotopy $h$ can be chosen to satisfy the side conditions $ph=0$, $hi=0$ and $h^2=0$.
Dually, if $i\colon V\hookrightarrow W$ is an acyclic cofibration of complete complexes, then it is part of a similar deformation retract.
\end{lemma}
\begin{proof}
We will only prove assertiona about $p$. By Lemma \ref{lem:cofibrations of complete complexes}, every complete complex is cofibrant. It follows from the model category axioms that $p$ admits a section $i$, which decomposes $W\cong V\simeq C$ with $C$ weakly contractible. It then suffices to provide a contracting homotopy $h$ on $C$ such that $h^2=0$. 
Let $j\colon C\rt \mm{Cone}(C)$ be the inclusion of $C$ into its cone. Since $C$ is cofibrant and acyclic, this is a trivial cofibration between fibrant objects; it therefore admits a retraction. Writing $\mm{Cone}(C)=C\oplus C[1]$, this retraction takes the form $(\mm{id}, h)\colon C\oplus C[1]\rt C$, where $h$ is a contracting homotopy. One can now define $h'=-hdh$ and check that $h'$ is the desired contracting homotopy satisfying $(h')^2=0$.
%
%
% Denoting by 
%$$
%I=C^*\big(\Delta[1]\big)=\big(k\oplus k\rto{+} k\big)
%$$
%the normalized cochains on the $1$-simplex, we can consider the diagram
%$$\begin{tikzcd}
%& I\otimes W\times_{I\otimes V} V\arrow[d, twoheadrightarrow, "\sim"]\\
%W\arrow[r, "{(ip, \id_W)}"{below}]\arrow[ur, dashrightarrow] & W\times_V W
%\end{tikzcd}$$
%The right vertical map is an acyclic fibration, so there exists a diagonal lift. Unraveling the definitions, this lift provides a chain homotopy $h$ between $ip$ and $\id_W$ such that $ph=0$. One can now apply the transformations $h'=[d,h]h[d,h]$ and $h''=-h'dh'$ and check that $p, i$ and $h''$ satisfy the side conditions.\joost{is there a nice reference for this trick?}
\end{proof}
\begin{proposition}\label{prop:qiso is qinv}
	Let $f\colon W \rt V$ be a weak equivalence of complete filtered complexes.
	Then  $f$ is a filtered homotopy equivalence. 
\end{proposition}
\begin{proof}
We can decompose $f$ as $W\rt W\oplus V[0, -1]\rto{p} V$, where $V[0, -1]$ is the (contractible) path space of $V$. The first map is the obvious summand inclusion (hence a homotopy equivalence) and $p$ is given on $W$ by $f$, while on $V[0, -1]$ it is determined uniquely by the fact that on $V[0]$ it is the identity. Note that $p$ is both surjective in every filtration degree and a filtered quasi-isomorphism (since $f$ was). Consequently, $p$ is part of a deformation retract by Lemma \ref{lem:acyclic fib is def retract}, so that the composite $f$ is a homotopy equivalence as well.
\end{proof}

\begin{theorem}[Homotopy Transfer Theorem]\label{thm:HTT}
	Let $\CC$ be a complete conilpotent operad.
	Suppose $W$ is a complete $\Omega \CC$-algebra that homotopy retracts (as a filtered complex) to a complete complex $V$, as above. 
	Then there is a \emph{transferred} $\Omega \CC$-algebra structure on $V$ such that $i$ extends to an $\infty$-morphism of $\Omega \CC$-algebras.
\end{theorem}

\begin{proof}
	The proof is identical to the one in \cite[Section 10.3]{LodayVallette2012}. In loc. cit.\ Loday and Vallette consider the case where $\PP$ is a Koszul operad (and $\CC = \PP^{\antishriek}$) for simplicity, but the proof carries through.
	Indeed, the transferred structure is constructed by establishing a universal map $\Bar\End_W\to \Bar\End_V$, obtained by composing incoming edges with $i$, outgoing edges with $p$ and adding a copy of $h$ to every  internal edge. Since we require our homotopy retract to be made up of filtered maps, the map  $\Bar\End_W\to \Bar \End_V$ is compatible with the filtrations.
		The extension of $i$ to an $\infty$-morphism $i_\infty$ involve similar formulas and is therefore compatible with the filtrations.
\end{proof}

\subsubsection{Minimal models}
\begin{definition}
	A complete complex $V$ is said to be \emph{minimal} if for all $n\in \mathbb{Z}$, $dF^nV \subseteq F^{n+1}V$. In other words, if the differential vanishes on the associated graded.
\end{definition}
\begin{proposition}\label{prop:complex is we to homology}
	Every complete filtered complex (over a field $k$) admits a deformation retract to a minimal complete filtered complex,  with side conditions $ph=0$, $hi=0$ and $h^2=0$. This minimal complete complex is unique up to non-canonical isomorphism.
\end{proposition}
\begin{lemma}\label{lem:exact sequences of equivalent complexes}
Consider a diagram of complete complexes
$$\begin{tikzcd}
0\arrow[r] & M'\arrow[d, "\sim"{swap}, hook]\arrow[r, dotted] & M\arrow[d, dotted, "\sim"] \arrow[r, dotted] & M''\arrow[r]\arrow[d, hook, "\sim"] & 0 \\
0\arrow[r] & V'\arrow[r] & V\arrow[r] & V''\arrow[r] & 0
\end{tikzcd}$$
in which the bottom row is short exact and the vertical maps are acyclic cofibrations. Then there exists a complete module $M$ and a dotted extension of the diagram as indicated, such that the top row is exact and the map $M\rt V$ is an acyclic cofibration as well.
\end{lemma}
\begin{proof}
Note that a map $W'\rt W$ is an acyclic cofibration if and only if it is of the form $W\rt W\oplus C$, where $C$ is acyclic (since it admits a retraction for model-categorical reasons).
Let $N=V\times_{V''} M''$, which fits into a short exact sequence $0\to V'\to N\to M''\to 0$. The map $N\rt V$ is the pullback of an acyclic cofibration along a fibration, and is hence easily seen to be an acyclic cofibration as well (using the above observation).

We can write $V'\cong M'\oplus C$, for some contractible complex. Then the inclusion $C\hookrightarrow V'\hookrightarrow N$ is a cofibration whose domain is contractible. This implies that $N\rt N/C$ is an acyclic fibration and hence a deformation retract by Lemma \ref{lem:acyclic fib is def retract}. In particular, the short exact sequence $0\to C\to N\to N/C\to 0$ splits and we can identify $C\rt N$ with a summand inclusion $C\subseteq N=C\oplus M$. We therefore obtain a commuting diagram
$$\begin{tikzcd}
M\arrow[r, "\sim", hook]\arrow[rd, "\sim"{swap}, hook] & N\arrow[r]\arrow[d, "\sim", hook] & M''\arrow[d, hook, "\sim"]\\
& V\arrow[r] & V''
\end{tikzcd}$$
where all downwards pointing arrows are acyclic cofibrations. Since the projection $N\rt M''$ sent $C\subseteq V'=M'\oplus C$ to zero, we see that the map $M\rt M''$ is surjective in each filtration degree, with kernel given precisely by $M'$. We therefore obtain the desired short exact sequence $M'\to M\to M'' $ mapping to $V'\to V\to V''$ by acyclic cofibrations.
\end{proof}
\begin{proof}[Proof of Proposition \ref{prop:complex is we to homology}]
To see uniqueness, suppose that $V$ and $W$ are weakly equivalent minimal complete complexes. By Proposition \ref{prop:qiso is qinv} every weak equivalence of complete complexes has a homotopy inverse, so we may assume that there exists a weak equivalence $f \colon V \to W$, as opposed to a zig-zag of weak equivalences. The induced map $\Gr(f) \colon \Gr(V) \to \Gr(W)$ is a quasi-isomorphism of complexes with trivial differential, hence an isomorphism. It follows that $f$ is itself an isomorphism.  

We will prove existence of minimal models in two steps, first dealing with the positive part of the filtration and then with the negative part.

\medskip

\noindent \textit{Positive filtration.} Assume that $V=F^0V$. We will inductively construct a tower of acyclic cofibrations $M^q\rto{\sim} V/F^qV$ (where $V/F^qV$ has the induced filtration), where each $M^q$ is a minimal complete complete and each $M^{q+1}\twoheadrightarrow M^q$ is a quotient map.

For $q=0$, one simply sets $M^0=0$. For the inductive step, suppose we have already constructed $i\colon M^q\rto{\sim} V/F^qV$. The complex $\Gr^q(V)$ can then be written as $F^qV/F^{q+1}V\cong H(F^qV/F^{q+1}V) \oplus C$, where $C$ is an acyclic complex and $H(F^qV/F^{q+1}V)$ is the homology. Let us now consider the following diagram, in which the rows are exact
%$$\begin{tikzcd}
%0\arrow[r] & H(F^qV/F^{q+1}V)\wgt{q}\arrow[r] & M^{q+1}\arrow[r] & M^q\arrow[d, equal]\arrow[r] & 0\\
%0\arrow[r] & F^qV/F^{q+1}V\arrow[u, two heads, "\sim"]\arrow[r]\arrow[d, equal] & N\arrow[r]\arrow[u, two heads, "\sim"]\arrow[d] & M^q\arrow[r]\arrow[d, "i"] & 0\\
%0\arrow[r] & F^qV/F^{q+1}V\arrow[r] & V/F^{q+1}V\arrow[r] & V/F^{q}V\arrow[r] & 0.
%\end{tikzcd}$$
$$\begin{tikzcd}
0\arrow[r] & H(F^qV/F^{q+1}V)\arrow[d, hook, "\sim"{swap}]\arrow[r, dotted] & M^{q+1}\arrow[r, dotted]\arrow[d, dotted, "\sim"] & M^q\arrow[r]\arrow[d, "i"] & 0\\
0\arrow[r] & F^qV/F^{q+1}V\arrow[r] & V/F^{q+1}V\arrow[r] & V/F^{q}V\arrow[r] & 0.
\end{tikzcd}$$
By Lemma \ref{lem:exact sequences of equivalent complexes}, there exists a complete complex $M^{q+1}$ making the top row exact, together with an acyclic cofibration $M^{q+1}\rto{\sim} V/F^{q+1}V$. Since the top row is short exact, it induces a short exact sequence on the associated graded (Lemma \ref{lem:graded exact}). Since $\Gr(M^q)$ is concentrated in weight $<q$ and $\Gr(H(F^qV/F^{q+1}V))$ is concentrated in weight $q$, we see that the differential on $\Gr(M^{q+1})$ vanishes, so that $M^{q+1}$ is minimal.

Now, taking the limit of the tower of $M^q$ provides a an acyclic cofibration $M=\lim_q M^q\rt \lim_q V/F^qV=V$ by the completeness of $V$. Since $\Gr(M)\rt \Gr(M^q)$ is an isomorphism in weight $\leq q$, we see that $M$ is minimal.

\medskip

\noindent \textit{Negative filtration.} 
Let us now consider the general case where $V$ need not agree with $F^0V$. We are going to inductively construct a compatible family of acyclic cofibrations $M^{(p)}\rt F^pV$, for $p\leq 0$, where each $M^{(p)}$ is minimal and $M^{(p)}\rt M^{(p-1)}$ is an isomorphism in filtration degree $p$. 

For $F^0V$, we have constructed an acyclic cofibration $M^{(0)}\rto{\sim} F^0V$ in our previous argument. Next, assume we have constructed $ ^{(p)}\rto{\sim}F^pV$. We now apply Lemma \ref{lem:exact sequences of equivalent complexes} to the short exact sequence $0\to F^pV\to F^{p-1}V\to F^{p-1}V/F^{p}V\to 0$ and the minimal models $M^{(p)}\rto{\sim} F^pV$ and $H(F^{p-1}V/F^{p}V)\rto{\sim} F^{p-1}V/F^{p}V$. This produces an acyclic cofibration $M^{(p-1)}\rto{\sim} F^{p-1}V$ whose domain fits into a short exact sequence 
$$
0\rt M^{(p)}\rt M^{(p-1)}\rt H(F^{p-1}V/F^{p}V)\rt 0.
$$ 
Passing to the associated graded, one sees that $M^{(p-1)}$ is minimal. Furthermore $M^{(p-1)}$ agrees with $M^{(p)}$ in filtration degree $p$ because $F^p\big(H(F^{p-1}V/F^{p}V)\big)=0$.

Finally, taking the colimit as $p\to -\infty$ yields an acyclic cofibration $M=\colim M^{(p)}\rt V$, where $M$ is minimal (since it agrees with the minimal complex $M^{(p)}$ in filtration weight $p$). Note that the acyclic cofibration $M\rt V$ admits a retraction, which is then an acyclic fibration. The desired deformation retract with side conditions then follows from Lemma \ref{lem:acyclic fib is def retract}.
\end{proof}

\begin{proposition}
	Let $V$ be a $\Omega \CC$-algebra that deformation retracts into a minimal complex $M$ satisfying the side conditions $ph=0$, $hi=0$ and $h^2=0$. Then, the map $p$ extends to an $\infty$-morphism $p_\infty$ between the transferred $\Omega \CC$-structure on $M$ given by Theorem \ref{thm:HTT} and $V$.
\end{proposition}

\begin{proof}
	The proof is the same as \cite[Proposition 10.3.14]{LodayVallette2012}, adapted to the complete filtered case.
\end{proof}

\begin{corollary}
	Restricting to the subcategory of complete filtered algebras,
	$\infty$-weak equivalences are $\infty$-quasi-invertible. 
%	It follows that the following homotopy categories are equivalent 
%	$$\mathrm{Ho}(\text{complete }\PP\text{-alg}, \PP\text{-morph}) \cong \mathrm{Ho}(\text{complete }\PP\text{-alg}, \infty\text{-morph})\cong \mathrm{Ho}(\text{complete }\PP_\infty\text{-alg}, \infty\text{-morph}).$$
\end{corollary}

\begin{proof}
	Given an $\infty$-weak equivalence $f\colon V\rightsquigarrow W$ one constructs an $\infty$-quasi-inverse by taking the composite 
	$$\begin{tikzcd}
	W\arrow[r,rightsquigarrow, "p_\infty"] & H(W) \arrow[r, rightsquigarrow, "{[f]^{-1}}"] & H(V)\arrow[r, rightsquigarrow, "i_\infty"]& V
	\end{tikzcd}$$
where $[f]^{-1}$ is the inverse $\infty$-morphism as in \cite[Theorem 10.4.2]{LodayVallette2012}.
%	
%	To conclude that the homotopy categories are equivalent one uses the \emph{rectification} property by splitting the $\infty$-weak equivalence $f\colon V\to W$ into a strict morphism $\Omega \Bar V \to W$.
%	
%See \cite[Theorem 11.4.12]{LodayVallette2012} for more details.
\end{proof}

\subsubsection{The $\infty$-category of algebras}
By definition, the $\infty$-category of complete algebras over a complete operad $\PP$ is the $\infty$-categorical localization of the category of $\PP$-algebras at the filtered weak equivalences. As an application of the Homotopy Transfer Theorem, we will describe this $\infty$-category more explicitly in terms of $\infty$-morphisms.

As should be expected, the bar-cobar construction provides a cofibrant replacement functor on $\PP$-algebras. In fact, we will need a slightly stronger version of this.
\begin{proposition}\label{prop:bar cobar resolutions are cofibrant}
Let $\phi\colon \CC\rt \PP$ be a twisting morphism and $A$ a $\PP$-algebra. Then the natural map of $\PP$-algebras $\PP\circ A\rt \Omega_\phi\Bar_\phi(A)$ is a cofibration. In particular, the bar-cobar construction $\Omega_\phi\Bar_\phi(A)$ is a cofibrant $\PP$-algebra.
\end{proposition}
%\joostline{check everything below until the end of the proof of the proposition}
The standard method for checking cofibrancy of the bar-cobar construction is to endow it with a filtration, coming from the coradical filtration on the cooperad $\CC$. Since we do not have access to the coradical filtration in the complete setting (Warning \ref{war:complete conilpotence}), we will give a slightly different argument, using the homotopy transfer theorem as follows:
\begin{lemma}[$\infty$-sections from homotopy transfer]\label{lem:sections from htt}
Let $p\colon B\twoheadrightarrow A$ be an acyclic fibration between $\Omega\CC$-algebras and $i\colon A\rt B$ a linear section. Then $i$ extends to an $\infty$-morphism $i_\infty$ such that $pi_\infty=\id_A\colon A \rightsquigarrow A$.
\end{lemma}
\begin{proof}
By (the proof of) Lemma \ref{lem:acyclic fib is def retract}, $p$ and $i$ are part of a deformation retract with homotopy $h$ satisfying the side conditions. We can therefore apply the Homotopy Transfer Theorem \ref{thm:htt-B} to obtain \emph{another} $\Omega\CC$-algebra structure on $A$ for which $i$ can be upgraded to an $\infty$-morphism.

Since $p$ was already a (strict) map of $\Omega\CC$-algebras from the start, the formula for the transferred structure on $A$ in terms of trees with roots labeled by $p$ \cite[\S 10.3.3]{LodayVallette2012} shows that this transferred structure coincides with the original $\Omega\CC$-algebra structure on $A$. Furthermore, the formula for $i_\infty$ in terms of trees with roots labeled by $h$ \cite[\S 10.3.10]{LodayVallette2012} shows that $pi_\infty=\id_A$. 
\end{proof}
\begin{proof}[Proof of Proposition \ref{prop:bar cobar resolutions are cofibrant}]
It suffices to verify this when $\PP=\Omega\CC$ and $\phi=\iota$ is the universal twisting morphism. Indeed, $\PP(A)\rt \Omega_\phi\Bar_\phi(A)$ is the image of $\Omega\CC\circ A\rt \Omega_\iota\Bar_\iota(A)$ (where $A$ is viewed as a $\Omega\CC$-algebra) under the left Quillen functor $\PP\circ_{\Omega\CC}(-)\colon \Alg_{\Omega\CC}\rt \Alg_{\PP}$. 

In the case where $\PP=\Omega\CC$, since trivial fibrations are preserved under pullbacks, it suffices to verify that every acyclic fibration of $\Omega\CC$-algebras
$$\begin{tikzcd}
	\Omega\CC\circ A\arrow[r, "s"]\arrow[d] & B\arrow[d, twoheadrightarrow, "\sim"{swap}, "p"] \\
\Omega_\iota\Bar_\iota(A)\arrow[ru, dotted]\arrow[r, equal] & \Omega_\iota\Bar_\iota(A)
\end{tikzcd}$$
there exists a dotted section as indicated. Since maps $\Omega_\iota\Bar_\iota(A)\rt B$ correspond bijectively to $\infty$-morphisms $A\rightsquigarrow B$, this is equivalent to finding an $\infty$-morphism $s_\infty\colon A\rightsquigarrow B$ whose linear part agrees with the map $s$, such that the composition $A\rightsquigarrow B\rt \Omega_\iota\Bar_\iota(A)$ agrees with the universal $\infty$-morphism $v_\infty$ (adjoint to the identity on $\Omega_\iota\Bar_\iota(A)$). 

Let us first observe that the linear map $A\rt \Omega_\iota\Bar_\iota(A)$ underlying $v_\infty$ is the inclusion of a summand (induced by the inclusion $k\rt \Omega\CC\circ_\iota \CC_+$). In particular, it is a cofibration, so that we can find a linear section $i\colon \Omega_\iota\Bar_\iota(A)\rt B$ extending $s$. Lemma \ref{lem:sections from htt} shows that $i$ can be extended to an $\infty$-morphism $i_\infty$ such that $pi_\infty=\id_A$. Now the composite $s_\infty=i_\infty v_\infty$ provides the desired extension of $s$.
\end{proof}
\begin{definition}\label{def:oocat of algebras}
Let $\phi\colon \CC\rt \PP$ be a Koszul twisting morphism. We denote by $\oocat{Alg}_{\PP}^\mm{cpl}$ the simplicially enriched category where:
\begin{enumerate}[start=0]
\item objects are complete $\PP$-algebras.
\item for any two objects $A_0$ and $A_1$, the simplicial set of morphisms between them is given in simplicial degree $n$ by the set of $\infty$-morphisms
$$
\Map_{\PP}(A_0, A_1)_n = \big\{A_0\rightsquigarrow A_1\botimes \Omega[\Delta^n]\big\}.
$$
The composition of maps is given by
$$\begin{tikzcd}
A_0\arrow[r, rightsquigarrow, "\phi"] & A_1\botimes \Omega[\Delta^n]\arrow[r, rightsquigarrow, "{\psi\otimes \id}"] & A_2\botimes \Omega[\Delta^n]\botimes \Omega[\Delta^n] \arrow[r, "{\id\otimes \mu}"] & A_2\botimes \Omega[\Delta^n]
\end{tikzcd}$$
where the last map arises from the multiplication on $\Omega[\Delta^n]$.
\end{enumerate}
Note that $\oocat{Alg}_{\PP}^\mm{cpl}$ depends (implicitly) on a choice of Koszul twisting morphism.
\end{definition}
\begin{lemma}\label{lem:kan enriched}
For any two objects $A_0, A_1$, the simplicial set $\Map_{\PP}(A_0, A_1)$ of $\infty$-morphisms is a Kan complex. Furthermore, every (strict) weak equivalence $f\colon A_1\rt A_2$ induces a homotopy equivalence $f_*\colon \Map_{\PP}(A_0, A_1)\rt \Map_{\PP}(A_0,A_2)$.
\end{lemma}
\begin{proof}
One can identify $\Map_{\PP}(A_0, A_1)$ with the simplicial set of strict maps of $\PP$-algebras $\Omega_\phi\Bar_\phi(A_0)\rt A_1\botimes \Omega[\Delta^\bullet]$. The result then follows formally from $\Omega_\phi\Bar_\phi(A_0)$ being cofibrant (Proposition \ref{prop:bar cobar resolutions are cofibrant}) and $A_1\botimes \Omega[\Delta^\bullet]$ being a fibrant simplicial resolution of $A_1$ \cite{hinich1997homological}.
\end{proof}
\begin{proposition}\label{prop:oocat of algebras over k}
Let $\phi\colon \CC\rt \PP$ be a Koszul twisting morphism and consider the functor from the model category of $\PP$-algebras to the simplicial category of $\PP$-algebras
$$
j\colon \Alg_{\PP}^\mm{cpl}\rt \oocat{Alg}_{\PP}^\mm{cpl}.
$$
This exhibits $\oocat{Alg}_{\PP}^\mm{cpl}$ as the $\infty$-categorical localization of the category $\Alg_{\PP}^\mm{cpl}$ at the filtered quasi-isomorphisms.
\end{proposition}
\begin{proof}
By Lemma \ref{lem:kan enriched}, $j$ sends filtered quasi-isomorphisms to homotopy equivalences, so it induces an essentially surjective functor $j\colon \Alg_{\Omega\CC}^\mm{cpl}[\text{w.e.}^{-1}]\rt \oocat{Alg}_{\Omega\CC}^\mm{cpl}$. This functor is fully faithful because the mapping spaces in $\oocat{Alg}_{\Omega\CC}^\mm{cpl}$ compute the derived mapping spaces in the model category of complete $\Omega\CC$-algebras (by the proof of Lemma \ref{lem:kan enriched}). 
\end{proof}

\subsection{Main example: curved $L_\infty$-algebras}\label{sec:curved algebras}
Recall \cite[Section 3.2.2]{HirshMilles2012} that a \emph{curved Lie algebra} is a complete (non-differential) graded vector space $\mf{g}$, together with a Lie bracket $[-, -]$, a Lie algebra derivation $\nabla$ of cohomological degree $1$ and a degree $2$ element $\omega\in F^1(\mf{g})$ such that
$$
\nabla(\omega)=0 \qquad \qquad \text{and} \qquad \qquad \nabla^2(x) + [\omega, x]=0.
$$
(This notion is also frequently appears with the different convention $\nabla^2(x)=[\omega, x]$, e.g.\ in \cite{maunder2017Koszul}.) This is a particular example of a curved $L_\infty$-algebra \cite{Fukaya2003, HiroshigeStasheff2006, Getzler2018}, where all operations in arity $\geq 3$ vanish:
\begin{definition}\label{def:cLie classical}
A \emph{(classical) curved $L_\infty$-algebra} (over the field $k$) is a complete graded vector space $\mf{g}$, endowed with operations 
$$
\ell_i\colon \mm{Sym}^i_k(\mf{g}[1])\rt \mf{g}[1]
$$
such that $\ell_0\in F^1\mf{g}$ and the following equations hold:
$$
\sum_{\substack{p+q=n+1\\ q\geq 0, \ p\geq 1}}\sum_{\sigma\in \mathrm{Sh}_{p-1,q}^{-1}}\sgn(\sigma)(-1)^{(p-1)q}(\ell_p\circ_1\ell_q)^\sigma = 0.
$$
Equivalently, this is the data of a codifferential on the cofree complete coalgebra $\mm{Sym}^c_k(\mf{g}[1])$, sending the element $1$ to an element of filtration weight $1$.
\end{definition}
The purpose of this section is to describe the homotopy theory of curved $L_\infty$-algebras from the perspective of the above operadic framework.

\subsubsection{The Koszul morphism $\ucocom\{1\}\to \cLie$}\label{sec:Koszul morphism}
Let $\ucocom$ denote the linear dual of the unital commutative operad, $\ucocom(n)= k\mu_n,$ for $n\geq 0$. This comes equipped with partial cocomposition maps (dual to the partial composition of the unital commutative operad). Note that the partial composition maps do \emph{not} determine a total cocomposition map $\ucocom\rt \ucocom\circ \ucocom$.
We can endow $\ucocom$ with two filtrations:
\begin{enumerate}[label=(\alph*)]
\item the (`classical') filtration $\ucocom^\cl$, where $F^0\ucocom^\cl=\ucocom$, $F^1\ucocom^\cl=k\mu_0$ and $F^2\ucocom^\cl=0$.
\item the (`mixed') filtration $\ucocom^\mix$ where  $F^0\ucocom^\mix=\ucocom$, $F^1\ucocom^\mix=k\mu_0 \oplus k\mu_1$ and $F^2\ucocom^\mix=0$.
\end{enumerate}
With these filtrations, $\ucocom^\cl$ becomes a complete cooperad and $\ucocom^\mix$ becomes a conilpotent cooperad in the complete sense, as in Definition \ref{def:complete cooperad}.

\begin{remark}
	Note that $\ucocom^\cl$ is a counital cooperad, but (despite the name) $\ucocom^\mix$ is not since the ``counit'' $\ucocom(1) \to k$ is not compatible with the filtration.  
\end{remark}

\begin{definition}\label{def:mixed-curved lie over k}
	The \emph{mixed-curved $L_\infty$-operad} is the complete operad
	$$\cLie_\infty\coloneqq \Omega (\ucocom^\mix\{1\}).$$
	We will refer to algebras over $\cLie_\infty$ as \emph{mixed-curved $L_\infty$-algebras} (in contrast to the (classical) curved $L_\infty$-algebras from Definition \ref{def:cLie classical}) and write $\cat{cLie}^\mm{mix}$ for the category of mixed-curved $L_\infty$-algebras.
\end{definition}

Unraveling the definition, as a graded operad, $\cLie_\infty$ is freely generated by an infinite collection of 
operations $(\ell_n)_{n\geq0}$, where $\ell_n$ has arity $n$ and degree $2-n$. The filtration is given by 
$$
F^p\cLie_\infty = \text{span}\big\{\text{words containing at least }p \text{ times } l_0 \text{ or } l_1\big\}
$$
and the differential reads 
\begin{equation}\label{eq:cLoo structure equation}
-\partial (\ell_n)=\sum_{\substack{p+q=n+1\\ q\geq 0, \ p\geq 1}}\sum_{\sigma\in \mathrm{Sh}_{p-1,q}^{-1}}\sgn(\sigma)(-1)^{(p-1)q}(\ell_p\circ_1\ell_q)^\sigma
\end{equation}
where we use the convention $\partial(f)=d\circ f - (-1)^{|f|} f\circ d$. 
Let us make the differential more explicit in low arity: 
\begin{flalign*}
(n=0) & \quad -\partial(\ell_0)=\ell_1\circ_1 \ell_0 &&\\
(n=1) & \quad -\partial(\ell_1)=\ell_1\circ_1\ell_1+\ell_2\circ_1\ell_0 &&\\
(n=2) & \quad -\partial(\ell_2)=\ell_1\circ_1\ell_2-\ell_2\circ_1\ell_1+(\ell_2\circ_1\ell_1)^{(12)}+\ell_3\circ_1\ell_0.
\end{flalign*}
\begin{propdef}\label{propdef:mixed-curved lie}
	The filtered operad of  \emph{mixed-curved Lie algebras} is the operad $\cLie$ obtained from $\cLie_\infty$ by quotienting out the operadic dg-ideal generated by $\ell_3,\ell_4,\dots$. The quotient map $\cLie_\infty\to \cLie$ is a weak equivalence.
\end{propdef}
\begin{proof}
	On the associated graded of $\cLie_\infty$ let us consider a filtration by the number of times an element $\ell_0$ appears.
	Taking a spectral sequence with respect to this filtration, on the $0$-th page we observe that	the differential acts essentially like the differential in the ordinary operad $\Lie_\infty$. 
	Indeed, one can declare two $\ell_n$-labelled trees in $\cLie_\infty$ to have the same skeleton if they have the same number of $\ell_0$'s and $\ell_1$'s, appearing in the same position, see Figure \ref{fig:skeleton}.

\begin{figure}[h]
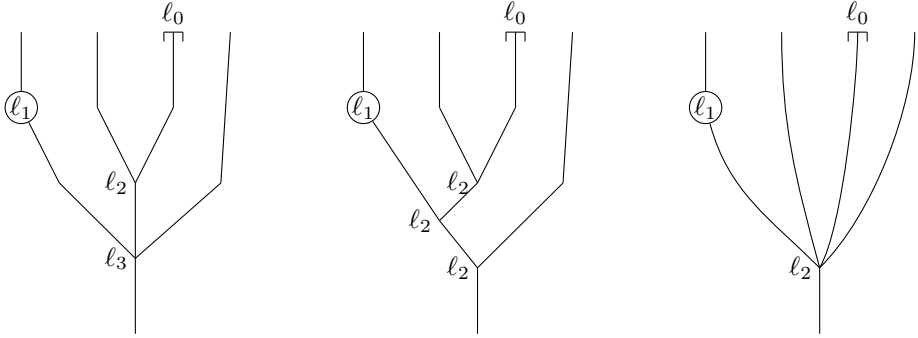

	\ctikzfig{skeleton}
	\caption{Three trees in $\cLie_\infty$ with the same skeleton.}\label{fig:skeleton}
\end{figure}
	It follows that map $\cLie_\infty\to \cLie$ on the $0$-th page decomposes as a sum of maps 
	\[
	\bigoplus_\text{skeleton types}\Lie_\infty\to \bigoplus_\text{skeleton types}\Lie,
	\]
	which is a quasi-isomorphism. 
	Since the (second) filtration is bounded above and exhaustive, it follows that the map $\Gr \cLie_\infty \to \Gr \cLie$ is also a quasi-isomorphism, i.e. $\cLie_\infty\to \cLie$ is a weak equivalence.
\end{proof}

\begin{corollary}
The map $\ucocom^\mix\{1\}\to \cLie$ mapping $\mu_n$ to $\ell_n$ for $n=0,1,2$ is a Koszul twisting morphism.
\end{corollary}
\begin{corollary}
Every mixed-curved $L_\infty$-algebra is filtered quasi-isomorphic to a mixed-curved Lie algebra.
\end{corollary}

Notice that unlike $\Lie$, the operad $\cLie$ has a differential and the complete complex underlying $\cLie$ is not minimal, since on the associated graded $\partial(\ell_1) = - \ell_2 \circ_1 \ell_0\ne 0$. 
A $\cLie$ algebra just a complete cochain complex $(V,d)$ equipped with operations $\ell_0,\ell_1$ increasing the filtration by $1$ and $\ell_2$ 
such that $(V,\omega=\ell_0, \nabla=\ell_1+d,\ell_2)$ is a curved Lie algebra in the usual sense, as described above Definition \ref{def:cLie classical}.
%\ricardo{Previously we had $\ell_1-d$ and now it's both $\ell_1+d$ and $\phi'_1=\phi_1 + \phi_\mm{lin}$. Keep this in mind in the revision}

\begin{remark}\label{rem:classical curved lie operadically}
One can apply a similar analysis starting instead from the cooperad $\ucocom^\cl$. In this case one obtains an operad $\Omega(\ucocom^\cl)$ with the same generating operations and differential \eqref{eq:cLoo structure equation}, but where only $\ell_0$ is of filtration weight $1$ and $\ell_1$ is of weight $0$. In particular, a curved $L_\infty$-algebra in the classical sense of Definition \ref{def:cLie classical} is simply a $\Omega(\ucocom^\cl)$-algebra structure on a complete graded vector space (with no differential). Equivalently, this is the data of a codifferential on the cosymmetric coalgebra $\ucocom^\cl\circ (\mf{g}[1])$ (where we consider $\ucocom^\cl$ as a counital cooperad).

The complete operad $\Omega(\ucocom^\cl)$ is filtered acyclic, since $\ucocom^\cl$ is counital. Likewise, the underlying unfiltered versions of $\cLie$ and $\cLie_\infty$ (either before or after completing with respect to the filtrations above) can be shown to be quasi-isomorphic to the unit operad. This would make the na\"ive homotopy theory of their algebras quite trivial.                  
\end{remark}

\subsubsection{Morphisms of mixed-curved $L_\infty$-algebras.}\label{sec:oo-morphism of clie}
Let us fix $\ucocom^\mix\{1\}\rt \cLie_\infty$, $\mu_n \mapsto \ell_n$ the universal twisting morphism. 
From Definition \ref{propdef:bar algebras} we have that an $\infty$-morphism $\phi\colon \mathfrak g \rightsquigarrow \mathfrak h$ between $\cLie_\infty$-algebras is determined by a map $\ucocom^\mix\{1\}_+ \circ \mathfrak g \to \mathfrak h$. Notice that at the level of the underlying vector spaces we have $\ucocom^\mix\{1\}_+ \circ \mathfrak g =  \mm{Sym}^c\left(\mathfrak g [1]\right)[-1] \oplus \mathfrak g$.
Given this, $\phi$ is determined by maps
$$
\phi_\mm{lin}\colon \mf{g}\rt \mf{h} \qquad\qquad \phi_n\colon \mm{Sym}^n(\mf{g}[1])[-1]\rt \mf{h} \qquad n\geq 0
$$
with $\phi_n$ maps of cohomological degree $1-n$ and where $\phi_0$ and $\phi_1$ increase the filtration weight by $1$. The map $\phi_\mm{lin}$ is required to be a chain map and
\begin{multline}\label{eq:mixed cLoo map}
\partial(\phi'_n)=\sum_{\substack{p+q=n+1\\ q\geq 0, \ p\geq 1}}\sum_{\sigma\in \mathrm{Sh}_{p-1,q}^{-1}}\sgn(\sigma)(-1)^{(p-1)q}(\phi'_p\circ_1\ell^{\mathfrak g}_q)^\sigma-\\ 
-\sum_{\substack{k\geq 0\\ i_1+\dots +i_k=n}} \sum_{\sigma\in \mathrm{Sh}_{(i_1,\dots,i_k)}^{-1}} \sgn(\sigma) \frac{(-1)^\epsilon}{k!} \ell^\mathfrak{h}_k(\phi'_{i_1}, \dots, \phi'_{i_k})^\sigma,
\end{multline}
where $\phi'_n=\phi_n$ if $n\neq 1$, $\phi'_1=\phi_1 + \phi_\mm{lin}$ and  $\epsilon = \prod_{j=1}^{k-1} (k-j)(i_k-1)$. Notice that the first sum is finite (like the infinitesimal cocomposition of $\ucocom$) whereas the second sum is infinite (since it corresponds to the total cocomposition of $\ucocom$).
For example, in the first term of an $\infty$-morphism $\phi \colon \mathfrak{g} \rightsquigarrow \mathfrak h$ we have
$$
d_{\mf{h}}(\phi_0) = \phi_{\mm{lin}}(\ell_0^{\mathfrak g}) +\phi_1(\ell_0^{\mathfrak g}) -\left(\ell_0^{\mathfrak h}+\ell_1^{\mathfrak h}(\phi_0) + \frac{1}{2!} \ell_2^{\mathfrak h}(\phi_0,\phi_0) + \frac{1}{3!}\ell_3^{\mathfrak h}(\phi_0,\phi_0,\phi_0) \dots\right).
$$

In particular, an $\infty$-morphism $0\rightsquigarrow \mathfrak h$ into an uncurved $\Lie_\infty$-algebra is equivalent to a choice of a Maurer--Cartan element in $F^1\mathfrak h$.
% \ricardoline{I am a bit puzzled, it feels that $ l_0^{\mathfrak h}$ should be in the previous formula, but that does not seem to be the case. Did I make a mistake somewhere? If this is correct, it means it means that an infty morphism from $0$ into a curved Lie algebra ignores its curvature.}
%
%\ricardoline{I think I was wrong. There is 0th bracket coming from the component of the counit $(\ucocom(0) \to \ucocom \circ I) (0) = \ucocom(0)$ sending $\mu_0\mapsto \mu_0$. This means that in the formula above we should have $k\geq 0$.}

\begin{remark}\label{rem:oo-morphisms of classical clie}
Likewise, applying Definition \ref{propdef:bar algebras} in the setting of $\Omega(\ucocom^\cl)$, one sees that an $\infty$-morphism between $\Omega(\ucocom^\cl)$-algebras is given by maps $\phi_\mm{lin}, \phi_i$ satisfying Equation \ref{eq:cLoo map}, except that $\phi_1$ is of filtration weight $0$ instead of $1$.

Now let $\mf{g}$ and $\mf{h}$ be two (classical) curved $L_\infty$-algebras in the sense of Definition \ref{def:cLie classical}, corresponding to $\Omega(\ucocom^\cl)$-algebras with zero differential. Then an \emph{$\infty$-morphism of curved $L_\infty$-algebras} $\mf{g}\rightsquigarrow \mf{h}$ is defined to be an $\infty$-morphism between the corresponding $\Omega(\ucocom^\cl)$-algebras such that $\phi_\mm{lin}=0$. In other words, it is given by a collection of maps $\phi'_n\colon \mm{Sym}^n(\mf{g}[1])[-1]\rt \mf{h}$
satisfying Equation \eqref{eq:mixed cLoo map}, where the left hand side is zero, because there is no differential. It is not hard to verify that this is equivalent to the data of a map of dg-coalgebras $\mm{Sym}^c(\mf{g}[1])\rt \mm{Sym}^c(\mf{h}[1])$ (see Definition \ref{def:cLie classical}).
\end{remark}
\begin{example}
A \emph{curved} map between two $\cLie$-algebras is an $\infty$-morphism $\phi \colon (\mathfrak g,\omega_{\mathfrak g}) \rightsquigarrow (\mathfrak h,\omega_{\mathfrak h})$ ($\omega$ denotes the curvature) such that $\phi_{\geq 2} =0$. In this case, Equation \ref{eq:cLoo map} reduces to the more familiar notion of a map of curved Lie algebras that we find for example in \cite[Def 4.3]{maunder2017Koszul} (up to signs): concretely, taking the linear term $\phi'_1=\phi_\mm{lin}+\phi_1$ as before and the curved differential $d' = d+\ell_1$, one finds

\begin{align*}
\phi'_1([X,Y])&= [\phi'_1(X),\phi'_1(Y)],\\
\phi'_1('X)  &=  d'\phi'_1 (X)+ [\phi_0,  \phi'_1(X)],\\
\omega_{\mathfrak h} &= \phi'_1(\omega_{\mathfrak g}) + d\phi_0 +\frac{1}{2}[\phi_0,\phi_0].
\end{align*}
\end{example}
\begin{remark}[General remark concerning signs]\label{rem:signs}
Despite the large quantity of signs, most of them come from the degree shift in the space of generators, $\ucocom\{1\}$. 
For instance, for the shifted operad $\cLie_{\infty}\{-1\} = \Omega(\cocom)$, the equation for $\infty$-morphisms \eqref{eq:mixed cLoo map} takes the form

\[	\partial(\phi_n)=\sum_{\substack{p+q=n+1\\ q\geq 0, \ p\geq 1}}(\phi'_p\circ_1\ell^{\mathfrak g}_q)^\sigma
	-\sum_{\substack{k\geq 0\\ i_1+\dots +i_k=n}} \sum_{\sigma\in \mathrm{Sh}_{(i_1,\dots,i_k)}^{-1}} \frac{1}{k!} \ell^\mathfrak{h}_k(\phi'_{i_1}, \dots, \phi'_{i_k})^\sigma.
\]
In practice, this allows us to abuse the notation later on, by replacing signs with $\pm$, since most of the signs are encompassed by this degree shift.
\end{remark}

\subsubsection{From mixed-curved $L_\infty$-algebras to curved $L_\infty$-algebras}
One can think of mixed-curved $L_\infty$-algebras as overdetermined versions of curved $L_\infty$-algebras in the sense of Definition \ref{def:cLie classical}: indeed, they come with two derivations $d, \ell_1$ (the second of filtration weight $1$) instead of a single $\ell_1$ of weight $0$. Likewise, $\infty$-morphisms between $\cLie_\infty$-algebras come equipped with two linear components, one being of filtration weight $1$. The classical homotopy theory of curved $L_\infty$-algebras is usually formulated in terms of $\infty$-morphisms with a single linear part (cf.\ Remark \ref{rem:oo-morphisms of classical clie}).
\begin{definition}\label{def:cat of classical curved Lie}
The \emph{$\infty$-category $\oocat{cLie}$ of curved $L_\infty$-algebras} is defined to be following simplicially enriched category:
\begin{enumerate}[start=0]
\item objects are curved $L_\infty$-algebras in the sense of Definition \ref{def:cLie classical}; equivalently, complete algebras over $\Omega(\ucocom^\cl)$ with zero differential (Remark \ref{rem:classical curved lie operadically}).
\item the simplicial sets of maps $\Map(\mf{g}, \mf{h})$ are given in degree $n$ by the set of $\infty$-morphisms $\mf{g}\rightsquigarrow \mf{h}\botimes \Omega[\Delta^n]$ in the sense of Remark \ref{rem:oo-morphisms of classical clie}.
\end{enumerate}
\end{definition}
To study the properties of the $\infty$-category of curved $L_\infty$-algebras, we will relate it to the $\infty$-category $\oocat{cLie}^\mm{mix}\coloneqq \oocat{Alg}^\mm{cpl}_{\cLie_\infty}$ of mixed-curved $L_\infty$-algebras, i.e.\ algebras over the operad $\cLie_\infty$ (Definition \ref{def:mixed-curved lie over k}). Note that this latter $\infty$-category has very good abstract properties, since it arises from a combinatorial model category: for example, it has all limits and colimits. More precisely, observe that there is a functor of simplicially enriched categories
\begin{equation}\label{eq:functor summing differential}\begin{tikzcd}
\blend\colon \oocat{cLie}^\mm{mix}= \oocat{Alg}^\mm{cpl}_{\cLie_\infty}\arrow[r] & \oocat{cLie}
\end{tikzcd}\end{equation}
sending a mixed-curved $L_\infty$-algebra $(\mf{g}, d, \ell_i)$ to the (classical) curved $L_\infty$-algebra $(\mf{g}, \ell'_i)$ where $\ell'_i=\ell_i$ for $i\neq 1$ and $\ell'_1=d+\ell_1$. On $\infty$-morphisms, the functor sends $(\phi_\mm{lin}, \phi_i)\colon \mf{g}\rightsquigarrow \mf{h}$ to $(\phi'_n)\colon \blend(\mf{g})\rt \blend(\mf{h})$ where $\phi'_n=\phi_n$ for $n\neq 1$ and $\phi'_1=\phi_\mm{lin}+\phi_1$, as in Equation \ref{eq:cLoo map}.

This forgets the redundancies is the definition of a curved $L_\infty$-algebra by combining the differential $d$ and its perturbation $\ell_1$, resp.\ the linear map $\phi_\mm{lin}$ and its perturbation $\phi_1$. Alternatively, one can also get rid of redundancies by imposing further restrictions on $d$:
\begin{definition}\label{def:graded mixed-curved Loo}
A \emph{graded mixed-curved $L_\infty$-algebra} is a graded complex $\mf{g}$, together with the structure of a mixed-curved $L_\infty$-algebra on $\Tot(\mf{g})$.
\end{definition}
There is a weight-graded operad $\cLie_\infty^\tot$ whose algebras are precisely the graded mixed-curved $L_\infty$-algebras. Indeed, let $\ucocom^\tot$ be the weight-graded cooperad spanned by operations $n$-ary operations $\mu_n^r$ of weight $r$, with $r\geq 0$ for $n\geq 2$ and $r\geq 1$ for $n=0, 1$. The cocomposition is that of the cocommutative operad
$$
\Delta(\mu_n^r) = \sum \mu_{m}^p\circ \big(\mu_{n_1}^{q_1}, \dots, \mu_{n_m}^{q_m}\big)
$$
where the sum runs over all indices such that $n_1+\dots+n_m=n$ and $p+q_1+\dots+q_m=r$. Note that this sum is finite because $n_i\in \{0, 1\}$ implies that $q_i\geq 1$, and that this defines a conilpotent cooperad. The desired operad is then given by
$$
\cLie_\infty^\tot = \Omega(\ucocom^\tot).
$$
Here we use that the bar-cobar formalism (as in Section \ref{sec:bar-cobar}) works equally well in the weight-graded setting. Unraveling the definitions, an algebra over this operad comes equipped with operations
$$
\ell_i^p\colon \mm{Sym}^i\big(\mf{g}[1]\big)\rt \mf{g}\wgt{p}[1]
$$
with $i\geq 2$ and $p\geq 0$ or $i=0, 1$ and $p\geq 1$, such that $\sum_p \ell_i^p$ makes $\Tot(\mf{g})$ a $\cLie_\infty$-algebra. 

It follows that the category of graded mixed-curved $L_\infty$-algebras carries a model structure and there is a notion of $\infty$-morphism between graded mixed-curved $L_\infty$-algebras. Explicitly, an $\infty$-morphism $\mf{g}\rightsquigarrow \mf{h}$ is given by maps $\phi_\mm{lin}\colon \mf{g}\rt \mf{h}$ and 
$$
\phi_0^p\colon k\rt \mf{h}\wgt{p} \qquad \qquad \phi_1^p\colon \mf{g}\rt \mf{h}\wgt{p} \qquad\qquad \phi_n^q\colon \mm{Sym}^n\big(\mf{g}[1]\big)[-1]\rt \mf{h}\wgt{q}
$$
for $p\geq 1$, $n\geq 2$ and $q\geq 0$. These maps have the property that $\phi_\mm{lin}$ is a chain map and that the maps between total complexes $\phi_n=\sum_p \phi_n^p$ satisfy Equation \ref{eq:mixed cLoo map}. In other words, an $\infty$-morphism between graded mixed-curved $L_\infty$-algebras $\mf{g}\rightsquigarrow \mf{h}$ is simply an $\infty$-morphism of mixed-curved $L_\infty$-algebras $\Tot(\mf{g})\rightsquigarrow \Tot(\mf{h})$ whose linear part respects the grading.
\begin{definition}
Let us denote by $\oocat{cLie}^\mm{mix}$ and $\oocat{cLie}^\mm{gr-mix}$ (the $\infty$-categories associated to) the simplicially enriched categories of (graded-)mixed-curved $L_\infty$-algebras and $\infty$-morphisms, as in Definition \ref{def:oocat of algebras}.
\end{definition}
The above discussion shows that there is a sequence of $\infty$-categories
$$
\begin{tikzcd}
\oocat{cLie}^\mm{gr-mix}\arrow[r, "\Tot"] & \oocat{cLie}^\mm{mix}\arrow[r, "\blend"] & \oocat{cLie}.
\end{tikzcd}
$$
\begin{proposition}\label{prop:graded mixed = classical}
The composite functor of $\infty$-categories is an equivalence.
\end{proposition}
\begin{proof}
To see that $\blend\circ \Tot$ is essentially surjective, pick a curved $L_\infty$-algebra $\big(\mf{g}', \ell'_i\big)$. We can split the filtration on the complete graded vector space underlying $\mf{g}$, i.e.\ we can write $\mf{g}'=\Tot(\mf{g})$ for some weight-graded (non-differential) graded vector space $\mf{g}$. Using this splitting, the operations $\ell'_i$ can be decomposed by pure weight: we can write
$$
\ell'_0 = \ell_0^1+\ell_0^2+\dots \qquad \ell'_1 = d+\ell_1^1+\dots \qquad\qquad \ell'_n= \ell_n^0+\ell_n^1+\dots \qquad (n\geq 2)
$$
where each $\ell_n^r\colon \mf{g}^{\otimes n}\rt \mf{g}\wgt{r}$ increases the weight by (exactly) $r$. We let $d$ denote the part of $\ell_1$ of pure weight $0$, and since $\ell_0$ was required to be of filtration weight $1$, there is no $\ell_0^0$.

Note that $\ell'_1\circ \ell'_1+\ell'_2\circ_1 \ell_0$ implies that $\ell'_1\circ \ell'_1$ is of filtration degree $1$; this implies that $d^2=0$. This means that $\big(\mf{g}', d, \ell'_0, \ell'_1-d, \ell'_2, \dots\big)$ is a mixed-curved $L_\infty$-algebra. By definition, this implies that $\big(\mf{g}, d, \ell_n^i\big)$ is a mixed-curved $L_\infty$-algebra, which maps to $\mf{g}'$ under the functor $\blend\circ \Tot$.

Concerning full faithfulness, pick two graded mixed-curved $L_\infty$-algebras $\mf{g}, \mf{h}$. The set of \emph{graded mixed} $\infty$-morphisms $\mf{g}\rightsquigarrow \mf{h}$ coincides with the set of those \emph{mixed} $\infty$-morphisms $\phi\colon \Tot(\mf{g})\rightsquigarrow \Tot(\mf{h})$ such that $\phi_\mm{lin}$ is homogeneous of weight $0$. The functor $\blend$ sends such an $\infty$-morphism to $\blend(\phi)$ with $\blend(\phi)_n=\phi_n$ for $n\neq 1$ and $\blend(\phi)_1=\phi_\mm{lin}+\phi_1$. This is a bijection: indeed, for any $\infty$-morphism $\phi'$ between curved $L_\infty$-algebras, we can always decompose $\phi'_1$ \emph{uniquely} into a part $\phi_\mm{lin}$ which is \emph{homogeneous} of weight $0$ and a part $\phi_1$ of filtration weight $\geq 1$. The resulting map $\phi_\mm{lin}$ is then a chain map for weight reasons. 

Replacing $\mf{h}$ by $\mf{h}\botimes\Omega[\Delta^n]$ shows that the functor $\blend\circ \Tot$ is fully faithful. In fact, it is even (strictly) fully faithful as a functor of simplicial categories.
\end{proof}
\begin{proposition}\label{prop:graded mixed as pullback}
The functor $\Tot\colon \oocat{cLie}^\mm{gr-mix}\rt \oocat{cLie}^\mm{mix}$ is a right adjoint functor between presentable $\infty$-categories. Furthermore, it fits into a homotopy pullback square of $\infty$-categories
$$\begin{tikzcd}
\oocat{cLie}^\mm{gr-mix}\arrow[r, "\Tot"]\arrow[d, "\mm{forget}"{left}] & \oocat{cLie}^\mm{mix}\arrow[d, "\mm{forget}"]\\
\oocat{Mod}^\gr_k\arrow[r, "\Tot"{below}] & \oocat{Mod}_k^\mm{cpl}.
\end{tikzcd}$$
\end{proposition}
\begin{proof}
Note that taking total complexes defines a right Quillen functor from graded mixed-curved $L_\infty$-algebras and mixed-curved $L_\infty$-algebras. In light of Proposition \ref{prop:oocat of algebras over k}, the functor $\Tot\colon \oocat{cLie}^\mm{gr-mix}\rt \oocat{cLie}^\mm{mix}$ is the associated derived functor between $\infty$-categorical localizations, hence a right adjoint between presentable $\infty$-categories \cite{hinichlocalization}. 

For the second statement, note that by Remark \ref{rem:model for derived complete}, the $\infty$-category of (derived) complete complexes can be modeled by the model category on complete complexes. In particular, we can describe the bottom row of the square in terms of simplicial categories as well: we simply take the simplicially enriched categories of complete (resp.\ graded) complexes with maps given in simplicial degree $n$ by chain maps $V\rt W\botimes \Omega[\Delta^n]$.

Using this and our description of $\infty$-morphisms between graded mixed $L_\infty$-algebras, one then sees that the above square arises from a (strict) pullback square of categories enriched in Kan complexes. To conclude, it suffices to verify that the right vertical functor is a \emph{fibration} of simplicially enriched categories.

To see this, we have to verify two conditions: first, given $\mf{g}, \mf{h}\in \cat{cLie}^\mm{mix}$, we have to verify that sending an $\infty$-morphism $\phi\colon \mf{g}\leadsto \mf{h}\otimes\Omega[\Delta^\bullet]$ to its base map $\phi_\mm{lin}\colon \mf{g}\rt \mf{h}\otimes\Omega[\Delta^n]$ gives a Kan fibration of simplicial sets. Unraveling the definitions, the Kan condition translates into the lifting condition
$$\begin{tikzcd}
\mm{Free}(\mf{g})\arrow[d]\arrow[r] & \mf{h}\otimes \Omega[\Delta^n]\arrow[d]\\
\Omega_\phi\Bar_\phi(\mf{g})\arrow[r]\arrow[ru, dotted] & \mf{h}\otimes \Omega[\Lambda^n_i]
\end{tikzcd}$$
where $\Omega_\phi\Bar_\phi(\mf{g})$ is the bar-cobar resolution for the canonical twisting morphism. The left vertical map is a cofibration by Proposition \ref{prop:bar cobar resolutions are cofibrant} and the right vertical map is an acyclic fibration, so the desired lift exists.

Second, we have to verify that the induced map on homotopy categories is an isofibration: if $i\colon V\rt \mf{g}$ is a homotopy equivalence between complete complexes and $\mf{g}$ carries a curved $\Lie_\infty$-structure, then there exists an $\infty$-morphism of curved $\Lie_\infty$-algebras $i_\infty$ whose base map is (homotopic to) $i$. To do this, we can extend $i$ to a homotopy retract by Proposition \ref{prop:qiso is qinv} and then apply the Homotopy Transfer Theorem \ref{thm:HTT}.
\end{proof}
\begin{corollary}\label{cor:main result}
The $\infty$-category $\oocat{cLie}$ of curved $L_\infty$-algebras is a presentable $\infty$-category, which arises as the pullback of $\infty$-categories 
$$
\oocat{cLie}\simeq \oocat{cLie}^\mm{mix}\times_{\oocat{Mod}_k^\mm{cpl}} \oocat{Mod}_k^\gr.
$$
\end{corollary}

%\ricardoline{cofibrancy of $\cLie_\infty$?}

%\subsubsection{Maurer--Cartan, gauge group and whatnot}
%
%Maybe some discussions to be had about filtration degrees?
%
%Assuming that Maurer--Cartan elements of curved Lie algebras live in $F^1$, are the MC spaces invariant under our weak equivalences?
%
%\subsubsection{$\infty$-morphisms as MC elements of the (curved) convolution Lie algebra}

\subsubsection{Other types of curved algebras}\label{sec:Other types of curved algebras}

The above results generalize from curved $L_\infty$-algebras to other types of curved algebras. More precisely, let us fix the following data (in unfiltered complexes). First, let $\CC$ be a graded cooperad concentrated in arity $\geq 2$ and let us think of the symbol $\uCC$ be a ``counital extension'', in the sense that for $n\geq 2$ we have 
$$
\uCC(n)= \CC(n), \qquad \uCC(0)= k\cdot \mu_0, \qquad \uCC(1)= k\cdot \mu_1\text{ in degree }0.
$$
Furthermore, $\uCC$ comes equipped with partial cocomposition maps $\uCC(r)\rt \uCC(n+1)\circ \uCC(r-n)$ which are coassociative and extend the partial cocomposition of $\CC$. 
We can endow this with two filtrations:
\begin{definition} 
Let $\uCC^\cl$ be the complete cooperad where $\mu_0$ has weight $1$ and all other operations have weight $0$. Then $\uCC^\cl$ is a counital, non-conilpotent complete cooperad, i.e.\ a counital coalgebra for $\circ$ on the category of symmetric complete complexes.

Similarly, let us call the ``mixed variant'' $\uCC^\mix$ the complete cooperad where $\mu_0$ and $\mu_1$ have weight $1$ and all other operations have weight $0$. Then $\uCC^\mix$ is a conilpotent non-counital cooperad.
\end{definition}
Furthermore, let $\PP$ be a graded operad and $\phi\colon \CC\rt \PP$ be a Koszul twisting morphism. 
 In particular the map $\Omega\CC\twoheadrightarrow \PP$ is surjective.
We then have the following three versions of `curved $\PP$-algebras':
\begin{definition}\label{def:classical general curved algebra}
A \emph{(classical) curved $\PP$-algebra} is a graded vector space $\mf{g}$, together with a codifferential on the cofree $\uCC^\cl$-coalgebra $\uCC^\cl(\mf{g})$, such that the corresponding map onto cogenerators 
$$
\delta\colon \uCC^\cl(\mf{g})\rt \mf{g}[1]
$$
induces a map $\Omega\CC^{\cl}\circ\mf{g}\rt \mf{g}$ vanishing on the kernel of $\Omega\CC^{\cl}\rt \PP$. An \emph{$\infty$-morphism} of curved $\PP$-algebras is a map of differential-graded $\uCC^\cl$-coalgebras.
\end{definition}
\begin{lemma}
Let us denote by $\cPP$ the quotient of the complete operad $\Omega(\CC^\mix)$ by the operadic ideal generated by $\ker(\phi)\subseteq \CC\subseteq \uCC^\mix$. This inherits a differential and the quotient map $\Omega(\CC^\mix)\rt \mm{u}\PP$ is a filtered quasi-isomorphism.
\end{lemma}
\begin{proof}
Exactly the same as Definition/proposition \ref{propdef:mixed-curved lie}.
\end{proof}
\begin{definition}
A \emph{mixed-curved $\PP$-algebra} is an (filtered complete) algebra over the complete operad $\cPP$. 
A \emph{graded mixed-curved $\PP$-algebra} is a weight-graded complex $\mf{g}$ equipped with a mixed-curved $\PP$-algebra structure on $\Tot(\mf{g})$.
\end{definition}
\begin{remark}\label{rem:Ctot}
The exact same discussion as below Definition \ref{def:graded mixed-curved Loo} shows that there is a weight-graded operad $\cPP^{\tot}$ whose algebras are graded mixed-curved $\PP$-algebras. Indeed, let $\uCC^{\tot}$ denote the weight-graded (conilpotent, noncounital) cooperad given by $\CC=\uCC\big/\langle \mu_0, \mu_1\rangle$ in weight $0$ and $\uCC$ in weight $\geq 1$. The comultiplication is inherited from $\uCC$. Then $\cPP^{\tot}$ is the quotient of $\Omega(\CC^{\tot})$ by the operadic ideal generated by $\ker(\phi)\subseteq \CC$ (in all weights $\geq 0$). An argument analogous to Proposition/definition \ref{propdef:mixed-curved lie} shows that $\uCC^{\tot}\rt \cPP^{\tot}$ is a Koszul twisting morphism.
\end{remark}
In particular, there is a natural notion of $\infty$-morphism between mixed (resp.\ graded mixed) curved $L_\infty$-algebras, given by maps between their bar constructions (which are conilpotent coalgebras over $\uCC^\mix$, resp.\ $\uCC^{\tot}$). We define the $\infty$-category of (graded mixed, mixed) curved $\PP$-algebras as in Definition \ref{def:oocat of algebras}, using $\infty$-morphisms into $\mf{h}\otimes \Omega[\Delta^\bullet]$ as morphisms.
\begin{proposition}\label{prop:curved algebras}
There are functors of $\infty$-categories
$$\begin{tikzcd}
\oocat{cAlg}^{\mm{gr-mix}}_{\PP}\coloneqq \oocat{Alg}_{\cPP^\tot}^\gr\arrow[r, "\Tot"] & \oocat{cAlg}^{\mm{mix}}\coloneqq \oocat{Alg}^\mm{cpl}_{\cPP}\arrow[r, "\blend"] & \oocat{cAlg}_{\PP}.
\end{tikzcd}$$
The total composite is an equivalence and the left functor $\Tot$ fits into a homotopy pullback square of the form
$$\begin{tikzcd}
\oocat{cAlg}^{\mm{gr-mix}}_{\PP}\arrow[r, "\Tot"]\arrow[d, "\mm{forget}"{left}] & \oocat{cAlg}^{\mm{mix}}_{\PP}\arrow[d, "\mm{forget}"]\\
\oocat{Mod}_k^\gr\arrow[r, "\Tot"{below}] & \oocat{Mod}_k^\mm{cpl}.
\end{tikzcd}$$
\end{proposition}
In particular, all the above $\infty$-categories are presentable.
\begin{proof}
The same proofs as for Proposition \ref{prop:graded mixed = classical} and Proposition \ref{prop:graded mixed as pullback}.
\end{proof}
\begin{example}\label{ex:graded mixed}
We can endow the unit cooperad $\mb{1}$ with two filtrations: let $\mb{1}^\mm{cl}$ simply be $\mb{1}$ in filtration weight $0$ and let $\mb{1}^\mm{mix}$ be $\mb{1}$ in filtration weight $1$. This corresponds to the case where $\CC=0$ is the zero cooperad, except that we omit the nullary operation $\mu_0$ (everything above holds in this setting as well). The cobar construction $\Omega(\mb{1}^\mm{cl})$ is the graded-commutative algebra $k[\delta]$ with $\delta$ of degree $1$ and filtration weight $0$, while $\Omega(\mb{1}^\mm{mix})$ has $\delta$ of filtration weight $1$. In particular, a `classical $\Omega(\mb{1}^\mm{cl})$-algebra' is nothing but a complete filtered complex and a graded mixed $\Omega(\mb{1}^{\mm{mix}})$-algebra is precisely a graded mixed complex (Definition \ref{def:graded mixed complex}). Proposition \ref{prop:curved algebras} then reproduces the equivalence of $\infty$-categories
$$\begin{tikzcd}
\Tot\colon \oocat{Mod}_k^{\mm{gr-mix}}=\oocat{Alg}_{\Omega(1)}^{\mm{gr-mix}}\arrow[r] & \oocat{Alg}_{\Omega(1)}=\oocat{Mod}_k^\mm{cpl}
\end{tikzcd}$$
announced in Proposition \ref{prop:graded mixed complexes}.
\end{example}
\begin{example}\label{ex:other curved algebras}
A way to encounter operads that fit the framework above is to consider Koszul operads $\PP$ such that their Koszul dual $\PP^!$ is extendable in the sense of \cite[Section 4.6]{DotsenkoShadrinVallette2018}: there is a \emph{unital extension} $\mathrm u\PP^!$ operad with a monomorphism $\PP\to \mathrm u\PP^!$.

It follows that taking $\PP$ to be the operad of associative, permutative, gravity or $\Pois_n$, there is a meaningful notion of curved $\PP$-algebras  \cite[Proposition 4.6.1]{DotsenkoShadrinVallette2018}. Let us treat the associative case in a more detail:

The associative operad $\As$ is Koszul self dual and extendable to the operad governing unital associative algebras $\mathrm{u}\As$.
The cooperad $\ucoas\{1\}$ has dimension $n!$ in every arity $n\geq 0$. 

Following Definition \ref{def:classical general curved algebra}, a curved $\As$ algebra is therefore a graded vector space $(A,\cdot,\mu_1,\theta)$, equipped with a product $\cdot\colon A\otimes A\to A$ satisfying the associativity relation, a degree $1$ endomorphism $d\colon A\to A$ which is an algebra derivation and a ``curvature'' element $\theta\in A_2$ such that $d^2x = \theta x - x  \theta$ for all $x\in A$. 
This corresponds precisely to the classical notion of a curved associative algebra, as in \cite[Section 3.1]{positselski2011Twokinds}.

A mixed-curved associative algebra on the other hand, is an algebra over the complete operad $\mathrm c\As$, i.e., a filtered dg associative algebra $(A,d,\cdot)$, complete with respect to the filtration and equipped furthermore with a degree $1$ and filtration increasing algebra derivation $\mu_1\colon F^\bullet A^\bullet \to F^{\bullet +1} A^{\bullet +1}$ and a curvature element $\theta \in F^1 A^2$ satisfying 
\[
(d+\mu_1)^2x = \theta x - x  \theta,\quad \quad \forall x\in A.
\]

\end{example}

\begin{example}[Example of example \ref{ex:other curved algebras}]
A classical example of a curved associative algebra is as follows: Given a manifold $M$ and a vector bundle $E\to M$, one can consider the set of $\End(E)$-valued differential forms $\Omega(M,\End(E))= \bigoplus_{p=0}^{\dim M} \Omega^p(M,\End(E))$ which is a graded associative algebra.
The choice of a connection $\nabla$ on $E$ gives rise to the covariant exterior derivative $d_\nabla\colon \Omega^\bullet(M,\End(E))\to \Omega^{\bullet +1}(M,\End(E))$. 
Interpreting the curvature of $\nabla$ as an element $F_\nabla \in \Omega^2(M,\End(E))$, the tuple $(\Omega(M,\End(E)),\wedge,d_\nabla,F_\nabla)$ is a curved associative algebra, see for instance \cite[Appendix B.1]{positselski2011Twokinds}.
\end{example}

\section{Curved Lie algebras over filtered algebras}\label{sec: filtered algebras}
In Section \ref{sec:complete operads} we have studied the homotopy theory of curved Lie algebras (or curved $L_\infty$-algebras) over a field $k$ from the perspective of filtered operadic homological algebra. In particular, we have seen that there is a simple operadic way to describe `overdetermined variants' of curved Lie algebras, i.e.\ mixed-curved Lie algebras; it then remains to remove some redundancies, for example by choosing a splitting of the filtration.

The purpose of this section is to give a similar analysis of the homotopy theory of curved Lie algebras over a \emph{filtered} commutative dg-algebra $B$:
\begin{definition}\label{def:classical clie over B}
Let $B$ denote a complete filtered commutative dg-algebra. A \emph{(classical) curved $L_\infty$-algebra} over $B$ is a complete $B$-module $\mf{g}$, endowed with operations 
$$
\ell_i\colon \mm{Sym}^i_B(\mf{g}[1])\rt \mf{g}[1]
$$
which are $B$-linear for $i\neq 1$ and a $B$-module derivation for $i=1$, such that $\ell_0\in F^1\mf{g}$ and the following equations hold:
$$
\sum_{\substack{p+q=n+1\\ q\geq 0, \ p\geq }}\sum_{\sigma\in \mathrm{Sh}_{p-1,q}^{-1}}\sgn(\sigma)(-1)^{(p-1)q}(\ell_p\circ_1\ell_q)^\sigma = 0.
$$
\end{definition}
This situation is more complicated because the operation $\ell_1$ on $\mf{g}$ (which does not square to zero and hence should be considered as algebraic data) is required to interact with the differential on $B$. To deal with this, we will assume throughout that the filtration on $B$ splits multiplicatively, so that $B$ arises as the totalization of the weight-graded algebra $\Gr(B)$. In this case, the differential on $B=\Tot(\Gr(B))$ decomposes as $d+\delta$, where $d$ is homogeneous of weight $0$ and $\delta$ is of filtration degree $1$, endowing the weight-graded $\Gr(B)$ with the structure of a `graded mixed cdga' in the following sense:
\begin{definition}\label{def:graded mixed algebra}
We define a \emph{graded mixed cdga} $B$ to be the datum of a graded cdga denoted $B_{\gr}$, together with a derivation $\delta\colon B\coloneqq \Tot(B_{\gr})\rt \Tot(B_{\gr})[1]=B[1]$ of filtration weight $1$ such that $\delta^2+[d, \delta]=0$.

By a \emph{complete module} over a graded mixed cdga $B$ we will mean a complete module of the associated complete filtered cdga $(B, d)$, while a \emph{graded module} is a weight-graded module over $B_{\gr}$.
%Likewise, a \emph{graded mixed module} $M$ over a graded mixed cdga $B$ is the datum of a graded $B_{\gr}$-module $M_{\gr}$, together with a map $\delta\colon M\coloneqq \Tot(M_{\gr})\rt \Tot(M_{\gr})[1]=M[1]$ of filtration weight $1$ such that $\delta(b\cdot m)=\delta(b)\cdot m + (-1)^{|b|} b\cdot \delta(m)$ and $\delta^2+[d, \delta]=0$.
%
%Maps between graded mixed cdgas (and graded mixed modules) are given by maps of graded cdgas (modules) whose induced map on total complexes respects the derivation $\delta$.
\end{definition}

\begin{remark}\label{def:de Rham algebra}
We will typically denote a graded mixed cdga $B$ and its associated complete cdga (with the differential $d$, not the total differential $d+\delta$) by the same symbol $B$. One can then identify $B_{\gr}=\Gr(B)$ with the associated graded of $B$. The data of a graded mixed cdga can also be described explicitly at the graded level: the graded cdga $B_{\gr}$ comes equipped with derivations $\delta_{p}\colon B\rt B\wgt{p}[1]$ for all $p\geq 1$, such that
$$
[d, \delta_r]+\sum_{p+q=r} \delta_p\circ \delta_q=0.
$$
This corresponds to a shifted version of a multicomplex \cite[10.3.7]{LodayVallette2012}.
\end{remark}

\begin{example}[De Rham algebra]\label{ex:de Rham}
Given a cofibrant cdga $A$, we denote by $\dR(A)$ the graded mixed cdga given by its completed algebra of differential forms, with respect to the grading given by the form degree $\dR(A)\langle p\rangle = \Omega_{\dR}^p(A)$ and derivation $\delta$ given by the de Rham differential $\delta = d_\dR$.
	
In particular, if $A$ is of the form $A=(\Sym(V), d)$, we have	 $\dR(A) = A\otimes \widehat\Sym(V[-1])$.
\end{example}

We will give an operadic description of mixed-curved $L_\infty$-algebras over such graded mixed algebras $B$. Contrary to the case where $B=k$ is a field, the operad controlling such algebras is not augmented, and hence does not quite fit into the framework of Section \ref{sec:complete operads}. Nonetheless, there are analogues of $\infty$-morphisms and the homotopy transfer theorem over $B$ as well.

\subsection{Mixed-curved Lie algebras}
Let $(B, \delta)$ be a graded mixed cdga. Throughout, we will consider $B$ with the (internal) differential $d$ and view $\delta$ as some additional algebraic structure, so that e.g.\ a $B$-module is simply a dg-module over $(B, d)$. In this case, there are obvious mixed and graded mixed variants of curved $L_\infty$-algebras over $B$:
\begin{definition}\label{def:mixed-curved lie over B}
A \emph{mixed-curved $L_\infty$-algebra} over a graded mixed cdga $(B, \delta)$ is a complete module $\mf{g}$ over the complete filtered cdga $B$, equipped with the structure of a $k$-linear mixed-curved $L_\infty$-algebra (Definition \ref{def:mixed-curved lie over k}) such that:
\begin{itemize}
\item for each $n\geq 2$, the map $\ell_n\colon \mm{Sym}_k(\mf{g}[1])\rt \mf{g}[2]$ is $B$-multilinear.
\item the map $\ell_1\colon \mf{g}\rt \mf{g}[-1]$ is a derivation over $\delta$, in the sense that
$$
\ell_1(b\cdot x) = \delta(b)x + (-1)^{|b|} b\cdot \ell_1(x), \hspace{1cm} \forall x\in \mf g, b\in B.
$$
\end{itemize}
A mixed-curved Lie algebra over $B$ is a curved $L_\infty$-algebra over $B$ for which the $\ell_n$ vanish for $n\geq 3$.
\end{definition}
\begin{definition}\label{def:graded mixed-curved lie over B}
A \emph{graded mixed-curved $L_\infty$-algebra} over a graded mixed cdga $(B, \delta)$ is a weight-graded module $\mf{g}_{\gr}$ over $B_{\gr}$, together with the structure of a mixed-curved $L_\infty$-algebra over $B$ on $\Tot(\mf{g}_{\gr})$.
\end{definition}
It is not difficult to see that there exists a complete $k$-linear operad $\cLie_{\infty, B}$ whose algebras are mixed-curved $L_\infty$-algebras over $B$. We can describe this operad more concretely by expressing it as a distributive law \cite[Section 8.6]{LodayVallette2012}:
\begin{lemma}\label{lem:operad for mixed-curved Loo over B}
Let $\cLie_{\infty}=\Omega(\ucocom^\mix\{1\})$ be the complete curved $L_\infty$-operad and consider
$$\begin{tikzcd}
\phi\colon \cLie_\infty\circ B\arrow[r] & B\circ \cLie_\infty
\end{tikzcd}$$
sending $\ell_n\circ (b_1, \dots, b_n) \mapsto (-1)^{(n-2)(|b_1|+\dots+|b_n|)} b_1\dots b_n\circ \ell_n$ for $n\geq 2$ and $\ell_1\circ b\mapsto (-1)^{|b|} b\circ \ell_1+ \delta(b)\circ 1$ (and extended in the obvious way to compositions). This defines a distributive law, with associated $k$-linear operad $\cLie_{\infty, B}\cong B\circ \cLie_\infty$.
\end{lemma}
Similarly, one obtains a graded operad $\cLie_{\infty, B}^\gr$ whose algebras are graded mixed-curved $L_\infty$-algebras.
\begin{proof}
We have to verify that $\phi$ is compatible with composition in $\cLie_\infty$ and $B$, i.e.\ that the two squares
$$\begin{tikzcd}
\cLie_\infty\circ \cLie_\infty\circ B\arrow[r, "\phi"]\arrow[d, "\text{compose}"{left}] & \cLie_\infty\circ B\circ \cLie_\infty\arrow[d, "\phi"] & \cLie_\infty\circ B\circ B\arrow[r, "\phi"]\arrow[d, "\text{compose}"{left}] & B\circ \cLie_\infty\circ B\arrow[d, "\phi"]\\
\cLie_\infty\circ B\arrow[r, "\phi"{below}] & B\circ \cLie_\infty & \cLie_\infty\circ B\arrow[r, "\phi"{below}] & B\circ \cLie_\infty
\end{tikzcd}$$
commute. The left square commutes since $\phi$ is defined on generators and extended to be composition-preserving in $\cLie_\infty$. For the right square, the only nontrivial thing to observe is that the bottom-left composite sends $\ell_1\circ b_1\circ b_2$ to $(-1)^{|b_1|+|b_2|}(b_1b_2)\circ \ell_1+\delta(b_1b_2)\circ 1$, while the top-right composite sends it to $(-1)^{|b_1|+|b_2|}(b_1b_2)\circ \ell_1+(-1)^{|b_1|}b_1\delta(b_2)\circ 1+\delta(b_1)b_2\circ 1$. These terms coincide since $\delta$ is a derivation.

Finally, one has to check that $\phi$ preserves the differentials. To this end, note that for $n\neq 1$, we can write $-[d, \ell_n]=[\ell_1, \ell_n]+T_n$, where $T_n$ is a composition of $\ell_i$ with $i\neq 1$ (see Equation \ref{eq:cLoo structure equation}). Ignoring the Koszul signs appearing in the definition of $\phi$, one can verify that 
$$
\phi\Big([\ell_1, \ell_n]\circ (b_1, \dots, b_n)\Big)=(b_1\dots b_n)\circ [\ell_1, \ell_n] \qquad\qquad \phi\Big(T_n\circ (b_1, \dots, b_n)\Big)=(b_1\dots b_n)\circ T_n\Big).
$$
Using this, one sees that $\phi$ sends $d(\ell_n\circ (b_1,\dots, b_n))$ to $d(b_1\cdot \dots\cdot b_n\circ \ell_n)$ for $n\neq 2$. For $n=1$, we have (in the case $b$ is even, the odd case is similar)
\begin{align*}
\phi\big(d(\ell_1\circ b)\big) &= \phi\Big(-\ell_1^2\circ b- (\ell_2\circ_1 \ell_0)\circ b -\ell_1\circ d(b)\Big)\\
&= -\Big(b\circ \ell_1^2 + \delta^2(b)\circ 1\Big) - \Big(b\circ (\ell_2\circ_1 \ell_0)\Big)-\Big(-d(b)\circ \ell_1 + \delta(d(b))\circ 1\Big)\\
&= b\circ d(\ell_1) + d(b)\circ \ell_1 + d(\delta(b))\circ 1=d\big(\phi(\ell_1\circ b)\big)
\end{align*}
We thus obtain an operad $B\circ \cLie_\infty$ from the distributive law $\phi$. By construction, an algebra over this operad is a $B$-module with the structure of an algebra over $\cLie_\infty$, such that $\ell_n$ is $B$-multilinear and $\ell_1$ is a derivation over $\ell_1$. This is precisely a mixed-curved $L_\infty$-algebra.
\end{proof}
Being the category of algebras over a complete operad, the category of mixed-curved $L_\infty$-algebras over $B$ admits a model structure whose weak equivalences are the filtered quasi-isomorphisms and whose fibrations are surjections in every filtration degree (Theorem \ref{thm:Model str}). Furthermore, the natural map of operads $B\rt \cLie_{\infty, B}$ induces a Quillen adjunction between $B$-modules and mixed-curved $L_\infty$-algebras
$$\begin{tikzcd}
\mm{Free}\colon \Mod_B \arrow[r, shift left=1ex] & \cat{cLie}_B\colon \mm{forget}\arrow[l, shift left=1ex].
\end{tikzcd}$$
To relate mixed-curved $L_\infty$-algebras over $B$ to curved $L_\infty$-algebras in the sense of Definition \ref{def:classical clie over B}, we will need a few more details on the homotopy theory of algebras over operads like $\cLie_{\infty, B}=B\circ \cLie_{\infty}$.

\subsection{Some operadic results}
Let $\CC$ be a nonunital complete (or graded) cooperad over $k$, let $\PP=\Omega(\CC)$ and let $(B, \delta)$ be a graded mixed cdga. Suppose that
$$
\phi\colon \PP\circ B\rt B\circ \PP
$$
is a well-defined distributive law, given  on generating elements $c\in \CC[-1]$ by
\begin{align}
c\circ (b_1, \dots, b_n)&\mapsto(-1)^{(|b_1|+\dots |b_n|)|c|}(b_1 \dots b_n)\circ c & \text{arity }\neq 1\nonumber\\
c\circ b&\mapsto (-1)^{|b|} b\circ c+\lambda_c\cdot \delta(b)\circ 1 & \text{arity }=1\label{eq:distributive law}
\end{align}
for certain $\lambda_c\in k$. This endows $\PP_B\coloneqq B\circ \PP$ with the structure of a (unital) operad, which need not be augmented (it is augmented iff all $\lambda_c=0$). There are natural maps of operads $B\rt \PP_B\lt \PP$, so that every $\PP_B$-algebra has an underlying $B$-module and an underlying ($k$-linear) $\PP$-algebra. Explicitly, a $\PP_B$-algebra is precisely a $B$-module $M$, equipped with the structure of an ($k$-linear) algebra over $\PP=\Omega(\CC)$, such that the generating operations $c\in \CC(p)$ interact with the $B$-module structure via
\begin{align*}
c(b_1\cdot m_1, \dots, b\cdot m_p) &=  (-1)^{(|b_1|+\dots |b_n|)|c|}(b_1\cdot \dots \cdot b_p)\cdot c(m_1, \dots, m_p) & \text{for }p\neq 2, \\
c(b\cdot m)& =(-1)^{|b|}b\cdot  c(m) + \lambda_c\cdot \delta(b)\cdot m.&
\end{align*}
By Theorem \ref{thm:Model str}, the category of $\PP_B$-modules carries a model structure. The purpose of this section is to gather some general results on this homotopy theory of $\PP_B$-algebras.
\begin{remark}\label{rem:graded algebras over B}
Recall from Remark \ref{rem:Ctot} that for any complete cooperad $\CC$, there is a cooperad $\CC^{\tot}$ given in weight $i$ by $\CC^\tot\langle i \rangle=F^i\CC$. This weight-graded cooperad has the property that a weight-graded $\Omega(\CC^\tot)$-algebra is precisely a weight-graded complex $V$ together with an $\Omega(\CC)$-algebra structure on the corresponding filtered complex $\prod_i V\langle i \rangle$. The distributive law \eqref{eq:distributive law} induces a distributive law in the weight-graded setting
$$
\Omega(\CC^\tot)\circ B_{\gr}\rt B_{\gr}\circ \Omega(\CC^\tot).
$$
Let $\PP_B^\gr=B_{\gr}\circ \Omega(\CC^\tot)$ denote the corresponding weight-graded operad, whose algebras are weight-graded $B_{\gr}$-modules $\mf{g}_{\gr}$ such that the complete $B$-module $\Tot(\mf{g}_{\gr})$ carries a compatible $\PP_B$-structure. All of the results from this section apply to this weight-graded operad as well (with easier proofs).
\end{remark}

\subsubsection{Homotopy transfer theorem}
The notion of an $\infty$-morphism for algebras over $\PP=\Omega(\CC)$ has an analogue for $\PP_B$-algebras:
\begin{definition}\label{def:oo-morphism B-linear}
An \emph{$\infty$-morphism} of $\PP_B$-algebras $\phi\colon \mf{g}\rightsquigarrow \mf{h}$ is an $\infty$-morphism between the underlying $k$-linear $\PP$-algebras, i.e.\ a map of $\CC$-coalgebras $\CC_+\circ \mf{g}\rt \CC_+\circ \mf{h}$, satisfying the following condition: forgetting differentials, each element $c\in \CC_+(p)$ induces a $B$-multilinear map
$$
\phi_c\colon \mf{g}^{\otimes p} \rt \mf{h}.
$$
\end{definition}
\begin{remark}\label{rem:composing oo-morphisms}
Given two $\infty$-morphisms $\phi\colon \mf{g}\rt \mf{h}$ and $\psi\colon \mf{h}\rt \mf{k}$, recall that the composition of their underlying $\infty$-morphisms of $\PP$-algebras is given by
$$
\big(\psi\phi\big)_c = \sum \psi_{c^{(1)}}\circ \big(\phi_{c^{(i_1)}}, \dots, \phi_{c^{(i_p)}}\big)
$$
where $\Delta(c)=\sum c^{(1)}\circ \big(c^{(i_1)}, \dots, c^{(i_p)}\big)$ is the (total) cocomposition. If each $\phi_c, \psi_c$ is $B$-multilinear, certainly $(\psi\phi)_c$ is $B$-multilinear as well. It follows that $\infty$-morphisms of $\PP_B$-algebras can be composed.
\end{remark}
Using this notion of $\infty$-morphism, we have the following version of the homotopy transfer theorem:
\begin{theorem}[Homotopy Transfer Theorem]\label{thm:htt-B}
Let $(B, \delta)$ be a graded mixed cdga, $\CC$ a complete nonunital cooperad over $k$ and suppose that $\PP_B=B\circ \Omega(\CC)$ arises from the distributive law \eqref{eq:distributive law}. Let $W$ be a $\PP_B$ algebra and consider a deformation retract of $B$-modules
$$\begin{tikzcd}[cells={nodes={}}]
V \arrow[yshift=-.5ex,swap]{r}{i}  
& W \arrow[yshift=.5ex,swap]{l}{p} 
\arrow[loop right, distance=2em, 
%start anchor={[yshift=1ex]east}, end anchor={[yshift=-1ex]east}
]{}{h} 
\end{tikzcd}$$
satisfying the side conditions $ph=0$, $hi=0$ and $h^2=0$. Then the $B$-module structure on $V$ extends to a \emph{transferred} $\PP_B$-algebra structure, and $i$ extends to an $\infty$-morphism $i_\infty$ of $\PP_B$-algebras.
\end{theorem}
\begin{proof}
Apply the $k$-linear Homotopy Transfer Theorem \ref{thm:HTT} to endow $V$ with a transferred $\Omega\CC$-algebra structure, together with a $k$-linear $\infty$-morphism $i_\infty\colon V\rightsquigarrow W$. We claim that the resulting $\PP$-algebra structure on $V$ is compatible with its $B$-module structure and that $i_\infty$ is $B$-multilinear.

To see this, recall that for any generating operation $c\in \Omega(\CC)$, the resulting operation on $V$ is given by a sum of trees, with vertices labeled by elements from $\CC$, leaves labeled by $i$, internal edges labeled by $h$ and the root labeled by $p$ \cite[\S 10.3.3]{LodayVallette2012}. Almost all operations appearing in this tree are $B$-multilinear, the only exception coming from vertices $c_1\in \CC(1)$, which can appear locally in the tree as
$$
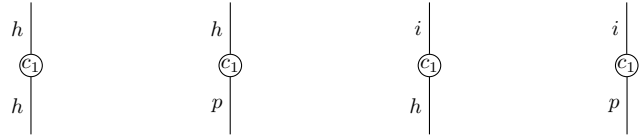
\begin{tikzpicture}[grow'=up]
\tikzset{level 1+/.style={level distance=1cm}}
\Tree[.\node {}; \edge node[auto=left, scale=0.8]{$h$}; [.\node[draw, circle, scale=0.8]{$c_1$};\edge node[auto=left, scale=0.8]{$h$}; [.\node {};]]]
\end{tikzpicture}
\qquad\qquad\qquad
\begin{tikzpicture}[grow'=up]
\tikzset{level 1+/.style={level distance=1cm}}
\Tree[.\node {}; \edge node[auto=left, scale=0.8]{$p$}; [.\node[draw, circle,, scale=0.8]{$c_1$};\edge node[auto=left, scale=0.8]{$h$}; [.\node {};]]]
\end{tikzpicture}
\qquad\qquad\qquad
\begin{tikzpicture}[grow'=up]
\tikzset{level 1+/.style={level distance=1cm}}
\Tree[.\node {}; \edge node[auto=left, scale=0.8]{$h$}; [.\node[draw, circle,, scale=0.8]{$c_1$};\edge node[auto=left, scale=0.8]{$i$}; [.\node {};]]]
\end{tikzpicture}
\qquad\qquad\qquad
\begin{tikzpicture}[grow'=up]
\tikzset{level 1+/.style={level distance=1cm}}
\Tree[.\node {}; \edge node[auto=left, scale=0.8]{$p$}; [.\node[draw, circle,, scale=0.8]{$c_1$};\edge node[auto=left, scale=0.8]{$i$}; [.\node {};]]]
\end{tikzpicture}
$$
Even though $c_1$ does not define a $B$-linear map on $W$, the first 3 composites are all $B$-linear by the side conditions: for example, 
$$
(hc_1h)(b\cdot m) = b\cdot (hc_1h)(m) + \lambda_{c_1}\cdot \delta(b)\cdot  h^2(m) = b\cdot (hc_1h)(m)
$$ by the side condition $h^2=0$. The fourth operation appears exactly once, in the formula for the transferred operation $c^V_1\colon V\rt V$. Consequently, all transferred operations of arity $\neq 1$ are $B$-multilinear, while $c_1^V=p\circ c_1\circ i + f$ with $f$ $B$-linear. In particular, $c_1^V$ satisfies
$$
c_1^V(b\cdot m) = \big(pc_1i\big)(b\cdot m) + f(b\cdot m) = b\cdot c_1^V(m) + \lambda_{c_1}\cdot \delta(b)\cdot m
$$
since $pi=1$.

A very similar argument shows that all components of the map $i_\infty$ are $B$-linear, since they are given by trees with vertices labeled by $\CC$, leaves labeled by $i$ and the root and internal edges labeled by $h$ \cite[\S 10.3.10]{LodayVallette2012}. The operations from $\CC(1)$ now only appear in the form of the left two pictures, and hence contribute $B$-linear terms to the formula for $i_\infty$.
\end{proof}

\subsubsection{Cofibrant resolutions}\label{sec:cofibrant resolutions over B}
Cofibrant $\PP_B$-algebras can be studied efficiently using bar-cobar methods.
\begin{lemma}\label{lem:cofibrant => B-cofibrant}
Every cofibrant $\PP_B$-algebra is cofibrant as a $B$-module.
\end{lemma}
\begin{proof}
Every cofibrant $\PP_B$-algebra is a retract of a quasi-free $\PP_B$-algebra with an increasing filtration on its generators. Since $\PP_B=B\circ \PP$ is cofibrant as a left $B$-module, such quasi-free algebras are cofibrant as $B$-modules as well (see also \cite[Corollary 5.5]{berger2003axiomatic}).
\end{proof}
Conversely, a $\PP_B$-algebra which is cofibrant as a $B$-module admits an explicit cofibrant replacement by means of a $B$-linear extension of the `bar-cobar construction' for the operad $\PP=\Omega(\CC)$. To explain this, it will be convenient to phrase things in terms of symmetric $B$-bimodules, i.e.\ symmetric sequences $X$ with actions 
$$
B\curvearrowright X(p)\curvearrowleft \Sigma_p\ltimes B^{\otimes p}.
$$
The category of such symmetric $B$-bimodules has a monoidal structure given by the relative composition product $\circ_B$, such that unital associative algebras in symmetric $B$-bimodules are simply ($k$-linear, complete) operads equipped with an operad map from $B$. In particular, $\PP_B$ is an algebra in this monoidal category. 

Now recall that the differential on the cobar construction $\PP=\Omega\CC$ decomposes as a sum of two differentials $d_{\Omega\CC}=d+\delta$, where $d$ arises from the differential on $\CC$ and $\delta$ is the cobar differential. For a distributive law as in \eqref{eq:distributive law}, the differential on $\PP_B=B\circ \Omega\CC$ can be decomposed similarly as $d+\delta$, where both terms are (square zero) derivations for the operadic composition.

Let us now consider $\CC_{B+}\coloneqq B\circ \CC_+$, equipped with the structure of a symmetric $B$-bimodule where $c\circ (b_1, \dots, b_n)=(-1)^{|c|(|b_1|+\dots |b_n|)}(b_1 \dots b_n)\circ c$. Without differentials, let
$$
\mc{M}_B\coloneqq \PP_B\circ_B \CC_{B+}\circ_B \PP_B
$$
be the free $\PP_B$-bimodule generated by $\CC_{B+}$ (relative to $B$). This inherits an (internal) differential $d$ from $\PP_B$ and $\CC_B$. Furthermore, it comes with an additional derivation $\delta$ (which is a derivation over the cobar differentials $\delta$ on $\PP_B$) given on generators $b\circ c\in \CC_{B+}=B\circ\CC$ by 
$$
\delta(b\circ c) = \big(b\circ c^{(1)}\big)\circ \big(c^{(2)}\big)\circ \big(1\big) - \big(b\big)\circ \big(c^{(1)}\big)\circ \big(c^{(2)}\big)\qquad\qquad \text{in } \PP_B\circ_B \CC_{B+}\circ_B \PP_B.
$$
Here $\Delta^{(1)}(c)=c^{(1)}\circ c^{(2)}$ is the infinitesimal cocomposition. This derivation $\delta$ \mbox{(graded-)} commutes with the internal differential $d$ and using coassociativity of the infinitesimal cocomposition, one sees that $\delta^2=0$. Taking the sum, this makes $\mc{M}_B$ a bimodule over $\PP_B$ (with its total differential).
\begin{remark}
Taking $B=k$, the resulting $\PP$-bimodule $\mc{M}_k$ is the usual bimodule such that $\mc{M}_k\circ_\PP \mf{g}$ takes the bar-cobar resolution of a $\PP$-algebra. The symmetric sequence $\mc{M}_B$ is isomorphic to $B\circ \mc{M}_k$, but the right $B$-action is nontrivial and involves the distributive law.
\end{remark}
\begin{definition}
The \emph{bar-cobar construction} of a $\PP_B$-algebra $\mf{g}$ is the $\PP_B$-algebra $\mc{M}_B\circ_{\PP_B}\mf{g}$.
\end{definition}
\begin{proposition}\label{prop:cobar corepresents oo-maps}
Let $\mf{g}$ and $\mf{h}$ be $\PP_B$-algebras. Then there is a natural bijection between the set of $\infty$-morphisms $\mf{g}\rightsquigarrow \mf{h}$ and the set of strict morphisms $\mc{M}_B\circ_{\PP_B}\mf{g}\rt \mf{h}$.
\end{proposition}
\begin{proof}
Note that the unit $k\to B$ induces a natural map of $k$-operads $\PP\rt \PP_B$ and a map of bimodules over it $\mc{M}_k\rt \mc{M}_B$. These induce a map of $\PP$-algebras $\mc{M}_k\circ_{\PP}\mf{g}\rt \mc{M}_B\circ_{\PP_B} \mf{g}$. Without the differential, this map can be identified with the quotient map
$$
\PP\circ \CC_+\circ \mf{g} \rt \PP_B\circ_B \CC_{B+}\circ_B \mf{g}
$$
from a free $\PP$-algebra to the quotient of a free $\PP_B$-algebra. This map sends maps the space of $\PP$-algebra generators $\CC_+\circ \mf{g}$ onto the module of $\PP_B$-algebra generators $\CC_{B+}\circ_B \mf{g}$. Consequently, the set of $\PP_B$-algebra maps $\mc{M}_B\circ_{\PP} \mf{g}\rt \mf{h}$ is a subset of the set of $\PP$-algebra maps $\mc{M}_k\circ_{\PP}\mf{g}\rt \mf{h}$; the latter is naturally isomorphic to the set of $k$-linear $\infty$-morphisms of $\PP$-algebras from $\mf{g}$ to $\mf{h}$.

Furthermore, a map of $\PP$-algebras $\mc{M}_k\circ_{\PP}\mf{g}\rt \mf{h}$ descends to a map of $\PP_B$-algebras $\mc{M}_B\circ_{\PP_B}\mf{g}\rt \mf{h}$ if and only at the level of generators, it descends from a $k$-linear map $\CC_+\circ \mf{g}\rt \mf{h}$ to a $B$-linear map $\CC_{B+}\circ_B \mf{g}\rt \mf{h}$. This means precisely that the components of the $\PP$-algebraic $\infty$-morphism are $B$-multilinear.
\end{proof}
\begin{proposition}\label{prop:cobar resolution of algebras}
Let $\mf{g}$ be a $\PP_B$-algebra whose underlying $B$-module is cofibrant. Then the natural map of complete $\PP_B$-algebras
$$
\mc{M}_B\circ_{\PP_B} \mf{g}\rt \mf{g}
$$ 
induces a quasi-isomorphism on the associated graded.
\end{proposition}
\begin{proof}
Since taking the associated graded commutes with taking relative composition products and the cobar construction, it suffices to prove the analogous statement in the setting where $\CC$, $B$ and $\mf{g}$ are graded objects. In particular, $\CC$ is a graded conilpotent cooperad and has an exhaustive coradical filtration. One can identify 
$$
\mc{M}_B\circ_{\PP_B}\mf{g}\cong \big(\PP_B\circ_B \CC_{B+}\big)\circ_B \mf{g}.
$$
The cobar differential applies the cocomposition to a vertex in $\CC_{B+}$ and then either moves the bottom vertex to $\PP=\Omega(\CC)$ or acts by the top vertex on $\mf{g}$. In particular, this differential preserves the exhaustive filtration on $\mc{M}_B\circ_{\PP_B}\mf{g}$ induced by the coradical filtrations on $\CC_+$, resp.\ on $\CC[-1]\subseteq \Omega(\CC)$. 

The associated graded can be associated with the composition product $X\circ_B \mf{g}$, where $X=B\circ \Omega(\CC)\circ \CC_+$, together with the differential taking an element $c$ in $\cat{C}_+$, replacing it by $c\circ 1$ and ``moving'' $c$ to $\Omega(\CC)$ while increasing the degree by $1$. As in Proposition \ref{propdef:bar algebras}, $X$ admits a contracting homotopy and the natural map $X\rt B$ is a (graded) quasi-isomorphism. Since $\mf{g}$ is cofibrant as a left $B$-module, the natural map $X\circ_B \mf{g}\rt \mf{g}$ is then a quasi-isomorphism as well.
\end{proof}

As indicated by the terminology, under good conditions the bar-cobar construction provides a cofibrant replacement $\mc{M}_B\circ_{\PP_B}\mf{g} \to \mf{g}$. Let us proof a slightly stronger version of this fact.
\begin{proposition}\label{prop:cobar is cofibrant}
Let $\mf{g}$ be a $\PP_B$-algebra whose underlying $B$-module is cofibrant. Then the natural map $\PP_B\circ_B \mf{g}\rt \mc{M}_B\circ_{\PP_B}\mf{g}$ is a cofibration of $\PP_B$-algebras. In particular, $\mc{M}_B\circ_{\PP_B}\mf{g}$ is a cofibrant $\PP_B$-algebra.
\end{proposition}
We start by proving an auxiliary lemma.
\begin{lemma}\label{lem:cobar is B-cofibrant}
Let $\mf{g}$ be a $\PP_B$-algebra whose underlying $B$-module is cofibrant. Then the natural map $v\colon \mf{g}\rt \mc{M}_B\circ_{\PP_B}\mf{g}$ induced by the unit map $B\rt \mc{M}_B$ is a cofibration of $B$-modules. In particular, the bar-cobar construction $\mc{M}_B\circ_{\PP_B}\mf{g}$ is cofibrant as a $B$-module. 
\end{lemma}
\begin{proof}
By Proposition \ref{prop:cofibrations of complete modules}, it suffices to verify that $v$ is a summand inclusion of quasiprojective $B$-modules and that the associated graded of $\mc{M}_B\circ_{\PP_B}\mf{g}$ is a cofibrant weight-graded module over $\Gr(B)$. For the first, recall that without differential we simply have that
$$
\mc{M}_B\circ_{\PP_B} \mf{g}\cong \PP_B\circ_B \CC_{B+}\circ_B \mf{g}.
$$
Since all symmetric sequences in the above expression are quasiprojective as left $B$-modules, their composition product is quasiprojective as well. Note that $v\colon \mf{g}\rt \mc{M}_B\circ_{\PP_B} \mf{g}$ is simply the inclusion of the summand corresponding to the summand $B\cdot 1\subseteq \PP_B\circ_B \CC_{B+}$.

To see that the associated graded of $\mc{M}_B\circ_{\PP_B}\mf{g}$ is cofibrant over $\Gr(B)$, it suffices to work entirely at the graded level, since taking the associated graded preserves bar/cobar constructions, relative composition products and the cofibrancy of modules. Let us therefore work in the setting where $\CC, B$ and $\mf{g}$ are all weight-graded.
In this graded setting, we can filter $\mc{M}_B\circ_{\PP_B} \mf{g}$ using the coradical filtration on the cooperad $\CC$ (as in the proof of Proposition \ref{prop:cobar resolution of algebras}). This is an exhaustive increasing filtration whose graded is given in degree $i$ by
$$
\Big(B\otimes \Gr^i\big(\Omega(\CC)\circ \CC_+\big)\Big)\circ_{B} \mf{g}.
$$
Since $\mf{g}$ is cofibrant as a weight-graded $B$-module by assumption and each $B\otimes \Gr^i\big(\Omega(\CC)\circ \CC\big)$ is manifestly cofibrant as a weight-graded $B$-module, the associated graded of our increasing filtration consists of cofibrant weight-graded $B$-modules. This implies that $\mc{M}_B\circ_{\PP_B} \mf{g}$ is a cofibrant weight-graded $B$-module itself as well.
\end{proof}
\begin{proof}[Proof (of Proposition \ref{prop:cobar is cofibrant})]
The proof is essentially the same as that of Proposition \ref{prop:bar cobar resolutions are cofibrant}. Using  Proposition \ref{prop:cobar corepresents oo-maps}, it suffices to show the following assertion: let $p\colon \mf{h}\twoheadrightarrow \mc{M}_B\circ_{\PP_B}\mf{g}$ be an acyclic fibration, equipped with a $B$-linear map $s\colon \mf{g}\rt \mf{h}$ such that $ps\colon \mf{g}\rt \mc{M}_B\circ_{\PP_B}\mf{g}$ is the linear map underlying the universal $\infty$-morphism $v_\infty\colon \mf{g}\rightsquigarrow \mc{M}_B\circ_{\PP_B}\mf{g}$. We have to provide a refinement of $s$ to an $\infty$-morphism $s_\infty$ such that $ps_\infty=v_\infty$. 

%\ricardoline{A shorter proof starts here}
First, note that we can always replace $\mf{h}$ by a cofibrant resolution $q\colon \mf{h}'\stackrel{\sim}{\twoheadrightarrow} \mf{h}$ and lift the map $s$ to a map with values in $\mf{h}'$. Using this, we may therefore assume that $\mf{h}$ is cofibrant as a $B$-module (by Lemma \ref{lem:cofibrant => B-cofibrant}).
Now, the underlying linear map $v\colon \mf{g}\rt \mc{M}_B\circ_{\PP_B}\mf{g}$ is a cofibration of $B$-modules by Lemma \ref{lem:cobar is B-cofibrant}, so that we can find a $B$-linear section $i\colon \mc{M}_B\circ_{\PP_B}\mf{g}\rt \mf{h}$ extending $v$. By the same argument as in Lemma \ref{lem:acyclic fib is def retract}, $i$ and $p$ are part of a deformation retract satisfying the side conditions. Lemma \ref{lem:sections from htt}, using the Homotopy Transfer Theorem \ref{thm:htt-B}, then provides an extension of $i$ to an $\infty$-morphism $i_\infty$. The desired map is now given by $s_\infty=v_\infty i_\infty$.

\end{proof}

\subsubsection{The $\infty$-category of algebras}
For a graded mixed cdga $(B, \delta)$ and $\CC$ and $\PP$ as above, we can consider the following simplicially enriched category of $\PP_B$-algebras and $\infty$-morphisms between them:
\begin{definition}\label{def:oocat of algebras over B}
Let $\oocat{Alg}^\mm{cpl}_{\PP_B}$ denote the simplicially enriched category defined as follows:
\begin{enumerate}[start=0]
\item objects are complete $\PP_B$-algebras whose underlying complete $B$-module is cofibrant (in the model structure of Theorem \ref{thm:Model str}).
\item For any two objects $\mf{g}$ and $\mf{h}$, the simplicial set of morphisms between them is given in simplicial degree $n$ by the set of $\infty$-morphisms
$$
\Map_{\PP_B}(\mf{g}, \mf{h})_n = \big\{\mf{g}\rightsquigarrow \mf{h}\botimes \Omega[\Delta^n]\big\}.
$$
\end{enumerate}
\end{definition}
\begin{remark}
The same proof as for Lemma \ref{lem:kan enriched} shows that $\oocat{Alg}^\mm{cpl}_{\PP_B}$ is enriched in Kan complexes.
\end{remark}
\begin{proposition}\label{prop:oocat of algebras over B}
Let $\CC$ be a nonunital complete cooperad, $(B, \delta)$ a graded mixed cdga and let $\PP_B$ be the operad arising from a distributative of the form \eqref{eq:distributive law}. Then the functor $\Alg_{\PP_B}^\mm{cpl}\rt \oocat{Alg}_{\PP_B}^\mm{cpl}$ exhibits $\oocat{Alg}_{\PP_B}^\mm{cpl}$ as the $\infty$-categorical localization of the model category of complete $\PP_B$-algebras at the filtered quasi-isomorphisms.
\end{proposition}
\begin{proof}
The proof of Proposition \ref{prop:oocat of algebras over k} carries over to this situation.
\end{proof}

\subsection{Homotopy theory of curved Lie algebras}\label{sec:homotopy of curved Lie over B}
Let us now return to our situation of interest, the homotopy theory of curved $L_\infty$-algebras. If $(B, \delta)$ is a graded mixed cdga, we have now have three different $\infty$-categories of (types of) curved $L_\infty$-algebras over $B$:
\begin{definition}\label{def:oocategory of curved lie over B}
Let $(B, \delta)$ be a graded mixed cdga. We consider the following three simplicial categories:
\begin{enumerate}
\item The $\infty$-category $\oocat{cLie}_B$ of \emph{curved $L_\infty$-algebras over $B$}. Its objects are (classical) curved $L_\infty$-algebras over $B$ (Definition \ref{def:classical clie over B}) whose underlying $B$-module is quasiprojective, such that the associated graded is a cofibrant $\Gr(B)$-module. The simplicial set of maps $\Map_{\oocat{cLie}_B}(\mf{g}, \mf{h})$ consists of $\infty$-morphisms $\phi\colon \mf{g}\leadsto \mf{h}\otimes\Omega[\Delta^n]$ between the underlying $k$-linear curved $L_\infty$-algebras, such that each map
$$
\phi_n\colon \mf{g}^{\otimes n}\rt \mf{h}\otimes \Omega[\Delta^n]
$$
is $B$-multilinear (forgetting the differentials on both sides).

\item The $\infty$-category $\oocat{cLie}_B^{\mm{mix}}$ of \emph{mixed-curved $L_\infty$-algebras} over $(B, \delta)$ (Definition \ref{def:mixed-curved lie over B}), defined as the $\infty$-category associated to the model category of complete algebras over the operad $\cLie_{\infty, B}$ from Lemma \ref{lem:operad for mixed-curved Loo over B}. By Proposition \ref{prop:oocat of algebras over B}, this can also be described as the simplicial category whose objects are mixed-curved $L_\infty$-algebras whose underlying complete $B$-module is cofibrant, with morphisms given by simplicial sets of $\infty$-morphisms (see Definition \ref{def:oocat of algebras over B}).

\item The $\infty$-category $\oocat{cLie}_{B_{\gr}}^{\mm{gr-mix}}$ of graded mixed-curved $L_\infty$-algebras over $B$ (Definition \ref{def:graded mixed-curved lie over B}), defined as the $\infty$-category associated to the model category of graded algebras over the operad $\cLie_{\infty, B}^{\gr}$ (Remark \ref{rem:graded algebras over B}). An analogue of Proposition \ref{prop:oocat of algebras over B} shows that this $\infty$-category can also be described as the simplicial category whose objects are graded mixed-curved $L_\infty$-algebras whose underlying graded $B_{\gr}$-module is cofibrant, with morphisms given by simplicial sets of $\infty$-morphisms.
\end{enumerate}
\end{definition}
We now have the following analogues of Proposition \ref{prop:graded mixed = classical} and \ref{prop:graded mixed as pullback}:
\begin{theorem}\label{thm:homotopy theories of curved lie over B}
Let $(B, \delta)$ be a graded mixed cdga. Then there is a sequence of functors of presentable $\infty$-categories whose composite is an equivalence
$$\begin{tikzcd}
\oocat{cLie}_B^\mm{gr-mix}\arrow[r, "\Tot"] & \oocat{cLie}_B^\mm{mix}\arrow[r, "\blend"] & \oocat{cLie}_B.
\end{tikzcd}$$
Furthermore, the first functor fits into a pullback square of $\infty$-categories
$$\begin{tikzcd}
\oocat{cLie}_B^\mm{gr-mix}\arrow[r, "\Tot"]\arrow[d, "\mm{forget}"{left}] & \oocat{cLie}_B^\mm{mix}\arrow[d, "\mm{forget}"]\\
\oocat{Mod}_{B_{\gr}}^\gr\arrow[r, "\Tot"{below}] & \oocat{Mod}_B^\mm{cpl}.
\end{tikzcd}$$
\end{theorem}
\begin{proof}
The functor $\Tot$ sends a graded mixed-curved $L_\infty$-algebra to its total complex and the functor $\blend$ is defined in exactly the same way like \eqref{eq:functor summing differential}: it sends a mixed-curved $L_\infty$-algebra $(\mf{g}, d, \ell_i)$ to the curved $L_\infty$-algebra $\big(\mf{g}, \ell_1'=d+\ell_1, \ell_{i\neq 1}\big)$, and similarly on $\infty$-morphisms one sums up the component $\phi_\mm{lin}$ and $\phi_1$. 
The proofs of Proposition \ref{prop:graded mixed = classical} and Proposition \ref{prop:graded mixed as pullback} now carry over verbatim: the functor $\oocat{cLie}_B^{\mm{gr-mix}}\rt \oocat{cLie}_B$ is an equivalence because every curved Lie algebra over $B$ whose underlying $B$-module is quasiprojective admits a splitting, i.e.\ arises as the totalization of a graded $B$-module (which can then be endowed with a canonical graded mixed-curved $L_\infty$-structure). 

Furthermore, the above pullback square is a strict pullback square of simplicially enriched categories. To verify that it is a homotopy pullback, it suffices to verify that the right vertical functor is a fibration, which follows from the Homotopy Transfer Theorem \ref{thm:htt-B}.
%\ricardo{add details}\joost{how about now?}\ricardo{fine}
\end{proof}
%
%
% We then have the following:
%\begin{proposition}
%There is a homotopy pullback square of $\infty$-categories
%$$\begin{tikzcd}
%\Alg_{\PP_B^\gr} \arrow[r] \arrow[d] & \Alg_{\PP_B}^\mm{cpl}\arrow[d]\\
%\Mod_B^\gr\arrow[r] & \Mod_B^\mm{cpl}
%\end{tikzcd}$$
%in which the horizontal functors send a graded object $\mf{g}\mapsto \hat{\bigoplus}\mf{g}(i)$.
%\end{proposition}
%\begin{proof}
%Exactly as Proposition \ref{prop:graded mixed as pullback}.
%\end{proof}
%
%

\section{Curved Lie algebroids}\label{sec:Curved Lie algebroids}
In this section we recall the notion of a Lie algebroid (or $L_\infty$-algebroid) over $A$, and introduce an obvious curved analogue as well. The main result (Theorem \ref{thm:curved Lie algebroids conceptually}) provides an equivalence between the $\infty$-category of such curved Lie algebroids and a certain $\infty$-category of weight-graded non-curved Lie algebroids. Restricting to subcategories of Lie algebras, this gives yet another description of the homotopy theory of curved Lie algebras, purely in terms of Lie algebras without curvature.

\subsection{Categories of (curved) Lie algebroids}
From now on, $A$ will denote a cdga in nonnegative cohomological degrees. Everything below will typically only be homotopically sound when $A$ is furthermore cofibrant or smooth. We will denote by $T_A$ the $A$-module of derivations of $A$, which comes with a $k$-linear Lie algebra structure given by the commutator bracket satisfying the Leibniz rule
\begin{equation}\label{eq:leibniz}
\hspace{1cm} [v, a \cdot w] = a\cdot [v, w] + \mc{L}_v(a)\cdot w, \hspace{1.5cm} a\in A, v,w\in T_A.
\end{equation}
\begin{definition}\label{def:loo algebroid}
An \emph{$L_\infty$-algebroid} over $A$ (relative to $k$) is given by an $A$-module $\Loid$ 
%\ricardo{Lie algebroids are now called just $\Loid$. If you think this is too bland, they are all called $\backslash$\texttt{Loid} }
whose underlying complex carries a $k$-linear $L_\infty$-algebra structure, together with an \emph{anchor map}
$$
\rho\colon \Loid\rt T_A.
$$
This data has to satisfy the following conditions:
\begin{itemize}
\item The anchor map preserves both the $A$-module and $L_\infty$-algebra structure.
\item All brackets $\ell_n$ of arity $n\geq 3$ are $A$-multilinear and the binary bracket satisfies the Leibniz rule \eqref{eq:leibniz}.
\end{itemize}
A \emph{Lie algebroid} is an $L_\infty$-algebroid whose brackets in arity $\geq 3$ vanish. 
%We will denote by $\cat{Lie}(A/k)$ the category of $L_\infty$-algebroids over $A$ (relative to $k$).
\end{definition}
There are evident filtered and weight-graded analogues of the above definition, where one treats $T_A$ (and $A$) as being concentrated in filtration degree (weight-grading) $0$. In particular, all elements of filtration degree $\geq 1$ (resp.\ weight $\neq 0$) are contained in the kernel of the anchor map. In addition, one can add curvature to the above definition, where the curvature (being of filtration degree 1) is contained in the kernel of the anchor map:
\begin{definition}
A \emph{curved $L_\infty$-algebroid} over $A$ is given by a $k$-linear curved $L_\infty$-algebra $\Loid$ (Definition \ref{def:cLie classical}) equipped with the structure of a complete graded $A$-module and an anchor map
$$
\rho\colon \Loid\rt T_A\wgt{0}
$$
to $T_A$, concentrated in filtration degree $0$. This data is required to satisfy two conditions:
\begin{itemize}
\item The anchor map is a map of complete graded $A$-modules and preserves the curved $L_\infty$-algebra structure (strictly).
\item The brackets $\ell_n$ of arity $n\geq 3$ are $A$-multilinear, the binary bracket satisfies the Leibniz rule \eqref{eq:leibniz} and $\ell_1$ is an $A$-module derivation.
\end{itemize}
We define mixed-curved and graded mixed-curved $L_\infty$-algebroids analogously.
Note that none of these versions of $L_\infty$-algebroids can be seen as algebras over an operad.
\end{definition}
\begin{example}
The terminal (curved) $L_\infty$-algebroid is $T_A$ itself (in filtration degree $0$), with the usual commutator bracket. Orthogonally, (curved) $L_\infty$-algebroids with zero anchor map are precisely (curved) $L_\infty$-algebras over $A$. 
\end{example}
\begin{example}\label{ex:free curved lie algebroid}
Let $\mf{g}\rt T_A$ be a (strict) map of $k$-linear curved $L_\infty$-algebras. Then $A\otimes \mf{g}$ has the structure of a curved $L_\infty$-algebroid over $A$, where the anchor map and the $\ell_n$ with $n\geq 3$ are extended $A$-multilinearly, while $\ell_2$ is extended according to the Leibniz rule \eqref{eq:leibniz}. This construction defines a left adjoint to the forgetful functor from (curved) $L_\infty$-algebroids to $k$-linear (curved) $L_\infty$-algebras over $T_A$.
\end{example}
\begin{example}\label{ex:tensoring of curved Loo algebroids}
If $\Loid$ is an $L_\infty$-algebroid, then the tensor products with differential forms on simplices 
$$
\Loid\boxtimes \Omega[\Delta^n]\coloneqq \Loid\otimes \Omega[\Delta^n]\times_{T_A\otimes \Omega[\Delta^n]} T_A
$$
again have the structure of $L_\infty$-algebroids (cf.\ \cite[Construction 5.23]{nuiten2019homotopicalalgebra} and \cite{vezzosi2015}). This provides a simplicial enrichment of the category of $L_\infty$-algebroids, with simplicial sets of maps consisting of (strict) maps $\Loid\rt \Loide\boxtimes \Omega[\Delta^n]$. The same thing applies to curved, mixed-curved and graded mixed-curved $L_\infty$-algebroids.
\end{example}
\begin{example}\label{ex:change of base}
Let $k\rt k'$ be a map of rings and let $A'=k'\otimes_k A$. Then there is an adjoint pair $\cat{cLie}(A/k)\leftrightarrows \cat{cLie}(A'/k')$ between (graded mixed, mixed) curved $L_\infty$-algebroids over $A$ relative to $k$ and curved $L_\infty$-algebroids over $A'$ relative to $k'$. The left adjoint sends $\Loid\rt T_{A/k}$ to $k'\otimes_k \Loid\rt k'\otimes_k T_{A/k}\rt T_{A'/k'}$ while the right adjoint sends $\Loide\rt T_{A'/k'}$ to the pullback $\Loide\times_{T_{A'/k'}} T_{A/k}$.
\end{example}
\begin{example}\label{ex:filtrations on Lie algebroids}
Let $\Loid$ be an $L_\infty$-algebroid. There are three canonical (complete) filtrations that one can put on $\Loid$:
\begin{enumerate}[start=0]
\item \emph{trivial filtration $\Loid\wgt{0}$:} the filtered $L_\infty$-algebroid with $F^1\Loid\wgt{0}=0$, $F^0\Loid\wgt{0}=\Loid$. 

\item \emph{Hodge filtration $\Loid\wgt{1}$:} given by $F^0\Loid\wgt{1}=0$, $F^{-1}\Loid\wgt{1}=\Loid$.

\item \emph{anchor filtration $\Loid^\mm{anc}$:} the fiber product $\Loid\wgt{0}\times_{T_A\wgt{0}} T_A\wgt{1}$. Explicitly, this is a filtered Lie algebroid over $T_A\wgt{1}$ given by
$$\begin{tikzcd}[row sep=0.4pc]
\text{Filtration:}& 1 & 0 & -1 & -2 & \\
\dots \arrow[r] & 0\arrow[ddd]\arrow[r, hook] & \mf{n}\arrow[ddd]\arrow[r] & \Loid\arrow[ddd, "\rho"]\arrow[r, "="] & \Loid\arrow[ddd]\arrow[r, "="] & \dots\\
\\
\\
\dots \arrow[r] & 0\arrow[r] & 0\arrow[r] & T_A\arrow[r, "="] & T_A\arrow[r, "="] & \dots
\end{tikzcd}$$ 
where $\mf{n}$ is the kernel of the anchor map.
\end{enumerate}
\end{example}
\begin{definition}\label{def:CE algebra}
Let $\Loid$ be a curved $L_\infty$-algebroid. Its \emph{Chevalley--Eilenberg complex} is given by the complete graded vector space
$$
C^*(\Loid)\coloneqq \Hom_A\big(\Sym_A(\Loid[1]), A\big)
$$
equipped with the Chevalley--Eilenberg differential, given (modulo Koszul signs) by
\begin{align*}
(d_{\mm{CE}}\alpha)(x_1,\dots, x_n) &=d_A\big(\alpha(x_1, \dots, x_n)\big)+\sum_i \mc{L}_{\rho(x_i)} \alpha(x_1, \dots, \hat{x_i}, \dots, x_n)\\
&- \sum_{k\geq 0} \sum_{\sigma\in \mm{Sh}^{-1}(k, n-k)}\alpha\big(\ell_k(x_{\sigma(1)}, \dots, x_{\sigma(k)}), x_{\sigma(k+1)}, \dots, x_{\sigma(n)}\big).
\end{align*}
The usual exterior product endows $C^*(\Loid)$ with the structure of a complete $k$-linear cdga.
\end{definition}
\begin{remark}
Abstractly, the Chevalley--Eilenberg complex can be identified as follows. The Lie algebra action of $T_A$ on $A$ (by derivations) endows $\mm{Sym}^c_k(T_A[1])\otimes_k A$ with a differential $\delta$ making it a dg-comodule over the ``bar construction'' $\mm{Sym}^c_k(T_A[1])$ (cf.\ \cite{halperin1992}). On the other hand, the anchor map $\rho$ induces a map of (complete) dg-coalgebras $\mm{Sym}^c_k(\Loid[1])\rt \mm{Sym}^c_k(T_A[1])$. Corestricting along this map, we obtain a dg-comodule structure on $\mm{Sym}^c_k(\Loid[1])\otimes_k A$. 

The complex of comodule maps $\mm{Sym}^c_k(\Loid[1])\rt \mm{Sym}^c_k(\Loid[1])\otimes_k A$ can then be identified with the complete graded vector space $\Hom_k(\mm{Sym}^c_k(\Loid[1]), A)$, endowed with the Chevalley--Eilenberg differential $d_{CE}$ described above. The subspace $\Hom_A(\mm{Sym}^c_A(\Loid[1]), A)$ of $A$-multilinear maps is closed under this differential.
\end{remark}
\begin{example}\label{ex:chevalley-eilenberg}
Let $\Loid$ be an ordinary $L_\infty$-algebroid over $A$. Then $C^*(\Loid)=C^*(\Loid\wgt{0})$ is the usual Chevalley--Eilenberg complex, concentrated in filtration weight $0$. On the other hand, $C^*(\Loid\wgt{1})$ is the Chevalley--Eilenberg complex of $\Loid$ equipped with the \emph{Hodge filtration}, in which $p$-forms are of filtration weight $p$. In particular, $C^*(\Loid\wgt{1})$ has a canonical graded mixed structure, with the grading given by form degree.

Finally, $C^*(\Loid^\mm{anc})$ is the Chevalley--Eilenberg complex of $\Loid$, endowed with the \emph{anchor filtration} where a form has filtration weight $p$ if it is zero when applied to $\geq p$ elements coming from the kernel of $\rho\colon \Loid\rt T_A$.
\end{example}

\subsection{Homotopy theory of curved Lie algebroids}
The categories of mixed-curved and graded mixed-curved $L_\infty$-algebroids almost carry a model structure:
\begin{theorem}\label{thm:model structure lie algebroids}
The category of mixed-curved $L_\infty$-algebroids over $A$ carries a (left) semi-model structure (cf.\ \cite[Section 12.1]{fresse2009}) such that:
\begin{enumerate}
\item weak equivalences are weak equivalences between the underlying complete $A$-modules.
\item fibrations are maps that induce surjections in each filtration weight.
\end{enumerate}
Furthermore, every cofibrant object is (in particular) cofibrant as a complete $A$-module. 

Similarly, there are (left) semi-model structures on the categories of graded mixed-curved $L_\infty$-algebroids and (weight-graded/complete) $L_\infty$-algebroids over $A$.
\end{theorem}
One does not obtain model structures in the strict sense because there are (mixed-curved) $L_\infty$-algebroids which do not admit a fibrant replacement $\Loid\rto{\sim} \Loid^\mm{fib}$ (i.e.\ one for which the anchor map is surjective) \cite[Example 3.2]{nuiten2019homotopicalalgebra}. However, this does not pose a problem from the point of view of $\infty$-categories: the associated $\infty$-categories still have all expected properties, e.g.\ limits and colimits that are computed as homotopy limits and colimits.
\begin{proof}
For ordinary $L_\infty$-algebroids, without filtrations or curvature, this is proven in \cite{nuiten2019homotopicalalgebra}. The proofs of loc.\ cit.\ carry over verbatim to this case; we will briefly outline the argument in the mixed-curved case, the other cases are easier. The desired semi-model structure is obtained by transfer along the free-forgetful adjunction $\mm{Free}\colon \Mod^\mm{cpl}_k/T_A\leftrightarrows \cat{cLie}(A/k)^\mm{mix}\colon U$. To establish the existence of the semi-model structure, it suffices to verify that for every cofibrant $\Loid$ and every contractible complete complex $Z$ over $T_A$, the map $\Loid\rt \Loid\amalg \mm{Free}(Z)$ is a trivial cofibration of complete $A$-modules \cite[Theorem 12.1.4]{fresse2009}. We will prove something stronger: let us say that an object $\Loid$ is \emph{good} if it satisfies the following two conditions:
\begin{enumerate}[label={(\alph*)}]
\item Without differential, it is the retract of a free mixed-curved $L_\infty$-algebroid $\mm{Free}(V_0)$ on a complete graded vector space over $T_A$.
\item The functor $\Loid\amalg \mm{Free}(-)$ sends (trivial) cofibrations of complete complexes over $T_A$ to (trivial) cofibrations of complete $A$-modules.
\end{enumerate}
Furthermore, we will say that a map $\Loid\rt \Loide$ is good if both $\Loid$ and $\Loide$ are good, and for every complete complex $X$, the map $\Loid\amalg\mm{Free}(X)\rt \Loide\amalg\mm{Free}(X)$ is a cofibration of complete $A$-modules. We now claim that every cofibration with cofibrant domain is a good map, which implies the existence of the semi-model structure, as well as the fact that all cofibrant objects are (in particular) cofibrant as complete $A$-modules.

To verify the claim, note that good morphisms are closed under transfinite compositions and retracts. Next, suppose that $\Loid$ is good and let $V\rt W$ be a cofibration of complete complexes. Then any pushout $\Loid\rt \Loid\amalg_{\mm{Free}(V)}\mm{Free}(W)$ is good. To see this, let $X$ be any other complex over $T_A$ an consider the map 
\begin{equation}\label{eq:attaching cells of curved lie algebroids}
\Loid\amalg\mm{Free}(X)\rt \Loid\amalg_{\mm{Free}(V)} \mm{Free}(W\oplus X).
\end{equation}
Note that without differential, the inclusion $V\rt W$ is the inclusion of a summand. Using Example \ref{ex:free curved lie algebroid} to compute free mixed-curved Lie algebroids in terms of free mixed-curved Lie algebras, the above map then takes the form
$$
A\otimes \cLie_\infty(V_0\oplus X) \rt A\otimes \cLie_\infty(V_0\oplus W/V\oplus X)
$$
where the target has some differential. As a map of complete $A$-modules, we can filter this map by word length in $W/V$. The associated graded can then be identified with $\Loid\amalg \mm{Free}(W/V\oplus X)$. Since $\Loid$ was good by assumption, this gives that the map \eqref{eq:attaching cells of curved lie algebroids} is a cofibration of complete $A$-modules, and that its codomain is a good object as well.

This implies that all cofibrations with a good domain are themselves good. It now remains to verify that the initial mixed-curved $L_\infty$-algebroid is good, i.e.\ that $\mm{Free}(-)=A\otimes\cLie_{\infty}(-)$ sends (trivial) cofibrations of complete complexes to (trivial) cofibrations of complete $A$-modules. This is immediate.
\end{proof}
%\begin{remark}\label{rem:cofibrant lie algebroids}
%It follows from the small object argument that every cofibrant (filtered, weight-graded) $L_\infty$-algebroid is given by a retract of some $A\otimes\Lie_\infty(x_\alpha)$, where the $x_\alpha$ form a well-ordered set of generators (equipped with maps to $T_A$) such that the differential of $x_\alpha$ is a word in its precursors. In particular, all cofibrant $L_\infty$-algebroids are cofibrant as $A$-modules. The same assertion holds for all other types of $L_\infty$-algebroids appearing in Theorem \ref{thm:model structure lie algebroids}.
%\end{remark}
%\begin{proposition}[{see also [HA for LA]}]\label{prop:preservation of sifted colimits}
%The forgetful functor $\Lie(A/k)\rt \Mod_A/T_A$ preserves sifted homotopy colimits, and likewise in the complete and weight-graded setting.
%\end{proposition}
%\begin{proof}
%It suffices to verify that the forgetful functor preserves homotopy colimits indexed by each category $\cat{I}$ with finite coproducts [HNP]. Since the forgetful functor preserves (strict) sifted colimits, it suffices to verify that every projectively cofibrant $\cat{I}$-diagram of $L_\infty$-algebroids is projectively cofibrant as a diagram of $A$-modules. Since the projective model structure on $\cat{I}$-diagrams of $A$-modules is compatible with the pointwise tensor product (because $\cat{I}$ has coproducts), this follows from the small object argument (analogous to Remark \ref{rem:cofibrant lie algebroids}); we refer to \cite{HA for LA} for a more detailed argument.
%\end{proof}
Our next goal will be to give a more explicit description of the $\infty$-categories associated to the model categories from Theorem \ref{thm:Model str}, in terms of $\infty$-morphisms:
\begin{definition}
Let $\Loid$ and $\Loide$ be curved $L_\infty$-algebroids. An $\infty$-morphism $\phi\colon \Loid\rightsquigarrow \Loide$ is an $\infty$-morphism between the underlying $k$-linear curved $L_\infty$-algebras satisfying the following two conditions:
\begin{enumerate}
\item The composite $\infty$-morphism $\Loid\rightsquigarrow \Loide\rto{\rho_{\Loide}} T_A$ agrees with the strict morphism $\rho_\Loid$.
\item Each component $\phi_n$ defines an $A$-multilinear map $\phi_n\colon \mm{Sym}^n_A(\Loid[1])[-1]\rt \Loide$.
\end{enumerate}
The same definition applies to (graded) mixed-curved $L_\infty$-algebroids and (complete, weight-graded) $L_\infty$-algebroids.
\end{definition}
We have the following version of the Homotopy Transfer Theorem (see also \cite{pymsafronov2020,camposhomotopy2019}):
\begin{theorem}[Homotopy Transfer Theorem]\label{thm:htt lie algebroids}
Let $\Loid$ be a mixed-curved $L_\infty$-algebroid over $A$ and consider a deformation retract of complete $A$-modules $\begin{tikzcd}[sep=small, cells={nodes={}}]
V \arrow[yshift=-.5ex,swap]{r}{i}  
& \Loid \arrow[yshift=.5ex,swap]{l}{p} 
\arrow[loop right, distance=2em, 
%start anchor={[yshift=1ex]east}, end anchor={[yshift=-1ex]east}
]{}{h} 
\end{tikzcd}$
relative to $T_A$ (i.e.\ both $i$ and $p$ commute with the projection to $T_A$), satisfying the side conditions $ph=0$, $hi=0$ and $h^2=0$. Then the $A$-module structure on $V$ extends to a \emph{transferred} mixed-curved $L_\infty$-algebroid structure, and $i$ extends to an $\infty$-morphism $i_\infty$ of mixed-curved $L_\infty$-algebroids.
\end{theorem}
\begin{proof}
The proof is similar to Theorem \ref{thm:htt-B}: we apply the Homotopy Transfer Theorem for $k$-linear mixed-curved $L_\infty$-algebras to obtain a transferred mixed-curved $L_\infty$-structure on $V$ and a $k$-linear $\infty$-morphism $i_\infty\colon V\rightsquigarrow \Loid$. Note that the homotopy $h$ takes values in the kernel of the anchor map $\rho_\Loid$, since $\rho_\Loid h= \rho_{V}(ph)$. The formula for $i_\infty$ then implies that all nonlinear components of $\rho_\Loid\circ i_\infty$ vanish, i.e.\ the underlying ($A$-linear) map $\rho_\Loid\circ i$ is a strict map of mixed-curved $L_\infty$-algebras. 

Next, recall that all operations $\ell_n$ for $n\neq 2$ on $\Loid$ are $A$-linear. On the other hand, using that $h$ takes values in the kernel of the anchor map and satisfies the side conditions, one sees that $h\circ (\ell_2\circ_1 h)$, $p\circ (\ell_2\circ_1 h)$ and $h\circ \ell_2\circ (i, i)$ are all $A$-linear as well. This implies that $i_\infty$ is $A$-linear and that there is only one term in the transferred structure that is not $A$-multilinear, namely the term $p\circ \ell_2\circ(i, i)$ in the formula for the transferred operation $\ell_2'$. Since $p\circ \ell_2\circ(i, i)$ precisely satisfies the Leibniz rule (and all other terms contributing to $\ell_2$ are $A$-bilinear), the transferred $k$-linear curved-mixed $L_\infty$-structure makes $V$ a curved-mixed $L_\infty$-algebroid.
\end{proof}
\begin{proposition}\label{prop:cobar resolution for loo algebroids}
Let $\Loid$ be a mixed-curved $L_\infty$-algebroid. Then there exists a unique mixed-curved $L_\infty$-algebroid $Q(\Loid)$ together with a natural bijection
$$
\Big\{\text{structure-preserving maps }Q(\Loid)\rt \Loide\Big\}\cong \Big\{\infty\text{-morphisms }\Loid\leadsto \Loide\Big\}.
$$
If $\Loid$ is cofibrant as a complete $A$-module, then $Q(\Loid)$ is a cofibrant mixed-curved $L_\infty$-algebroid and the natural map $Q(\Loid)\rt \Loid$ is a weak equivalence. Similarly for graded mixed-curved $L_\infty$-algebroids and (complete, weight-graded) $L_\infty$-algebroids.
\end{proposition}
\begin{proof}
The proof is similar to Proposition \ref{prop:cobar is cofibrant} and is slightly different from the non-curved, unfiltered case treated in \cite[Section 5]{nuiten2019homotopicalalgebra}. First, the existence and uniqueness of the object $Q(\Loid)$ follows from category theoretic reasons: the functor sending $\Loide$ to the set of $\infty$-morphisms $\Loid\rightsquigarrow \Loide$ preserves limits and filtered colimits, and is hence corepresentable. In particular, there is a canonical map $\pi\colon Q(\Loid)\rt \Loid$ corresponding to the identity $\Loid\rightsquigarrow \Loid$ and a canonical $\infty$-morphism $v_\infty\colon \Loid\rightsquigarrow Q(\Loid)$ corresponding to the identity on $Q(\Loid)$.

\medskip

\textit{Underlying graded $A$-module.} 
Next, let us identify the linear map $v_\mm{lin}\colon \Loid\rt Q(\Loid)$ as a map of complete graded $A$-modules, ignoring the differentials. We start by noting that the universal property of $Q(\Loid)$ realizes it as a certain quotient of the mixed-curved $L_\infty$-algebroid $A\otimes \Omega\Bar(\Loid)$, where $\Loid$ is viewed as a $k$-linear mixed-curved $L_\infty$-algebra (cf.\ Example \ref{ex:free curved lie algebroid}). Without differential, $Q(\Loid)$ can therefore be identified with the quotient of  $A\otimes \big(\cLie_\infty\circ \ucocom^{\mm{mix}}_+\circ \Loid\big)$ by the following relation: viewing elements in this complex as certain height 2 trees with root labeled by $A$ and leaves labeled by $\mf{g}$, rescaling a leaf by $a\in A$ is equivalent to rescaling the root by $a$. Using that $\Loid$ is (the retract of) a free $A$-module, this implies that $Q(\Loid)$ is a quasiprojective complete $A$-module. 

%The map $v\colon \mf{g}\rt Q(\mf{g})$ is then a summand inclusion, corresponding to the tree labeled by the unit elements of $\cLie_\infty$ and $\ucocom^{\mm{mix}}_+$.

\medskip 

\textit{Underlying $A$-module.}
We will now show that $Q(\Loid)$ is cofibrant as a complete $A$-module and that the map $\pi\colon Q(\Loid)\rt \Loid$ is a weak equivalence. 

%\joostline{how about this:}
To this end, let us endow $\Loid$ with an \emph{additional, increasing} filtration such that $F_0(\Loid)=0\subseteq \Loid=F_1(\Loid)$. This induces a nonnegative increasing filtration on $Q(\Loid)$ and the map $\pi$ respects these filtrations. In now suffices to show that the associated graded of $Q(\Loid)$ is cofibrant as a complete $A$-module and that $\pi$ induces an equivalence on the associated graded. 

This is most easily seen using the Rees construction, sending a complete $k$-module $V$ with an increasing filtration to the $\hbar$-torsion free $k[\hbar]$-module $\bigoplus_{n} \hbar^nF_n(V)$. The associated graded is then the fiber at $\hbar=0$. The Rees construction of $\Loid$ can then be identified with the mixed-curved $L_\infty$-algebroid $\Loid_{\hbar}$ over $A[\hbar]$ given by $A[\hbar]\otimes_A \Loid$, with brackets given by $\hbar\cdot \ell_n$ and anchor map given by 
$$
\hbar\cdot \rho\colon A[\hbar]\otimes_A \Loid\rt A[\hbar]\otimes_A T_A\subseteq T_{A[\hbar]}.
$$
Likewise, the Rees construction of the map $\pi\colon Q(\Loid)\rt \Loid$ coincides with the natural map $\pi_\hbar\colon Q(\Loid_{\hbar})\rt \Loid_\hbar$. We therefore have to show that $\pi_{\hbar}$ induces a weak equivalence between cofibrant complete $A$-modules after setting $\hbar=0$.

Using the adjunction from Example \ref{ex:change of base}, ones sees that after setting $\hbar=0$, the map $\pi_\hbar$ coincides with the map $Q(\Loid_0)\rt \Loid_0$ for the \emph{trivial} mixed-curved $L_\infty$-algebroid $\Loid_0$, i.e.\ $\Loid$ with zero brackets and zero anchor map. For these, $Q(\Loid_0)$ coincides with the bar-cobar construction for mixed-curved $L_\infty$-algebras in complete $A$-modules; this is indeed cofibrant as an $A$-module and equivalent to $\Loid_0$ (cf.\ Proposition-Definition \ref{propdef:bar algebras} and Section \ref{sec:cofibrant resolutions over B}).

\medskip

\textit{Cofibrancy.}
We have shown  that $Q(\Loid)$ is cofibrant as a complete $A$-module and it remains to verify that it is also cofibrant as a mixed-curved $L_\infty$-algebroid. We now argue as in Proposition \ref{prop:bar cobar resolutions are cofibrant} and Proposition \ref{prop:cobar is cofibrant}: it suffices to verify that for any acyclic fibration $p\colon \Loide\rt Q(\Loid)$, there exists an $\infty$-morphism $s_\infty\colon \Loid\rt \Loide$ such that $ps_\infty=v_\infty$ is the universal $\infty$-morphism. In fact, we can replace $\Loide$ by a cofibrant resolution and hence assume that it is cofibrant as an $A$-module. 

Since $Q(\Loid)$ is cofibrant as an $A$-module, there exists an $A$-linear section $i$ of $p$ and an $A$-linear homotopy that realizes $Q(\Loid)$ as a deformation retract of $\Loide$ relative to $T_A$. We can now apply the argument from Lemma \ref{lem:sections from htt}, using the Homotopy Transfer Theorem \ref{thm:htt lie algebroids}.
\end{proof}
\begin{definition}\label{def:simplicial cat of lie algebroids}
We will denote by $\oocat{cLie}(A/k)^{\mm{mix}}$ the $\infty$-category corresponding to the following simplicially enriched category:
\begin{enumerate}[start=0]
\item objects are \emph{fibrant} mixed-curved $L_\infty$-algebroids over $A$ whose underlying complete $A$-module is cofibrant.
\item simplicial sets of morphisms consist of $\infty$-morphisms $\Loid\rightsquigarrow \Loide\boxtimes \Omega[\Delta^n]$, using the tensoring from Example \ref{ex:tensoring of curved Loo algebroids}.
\end{enumerate}
Likewise, we will write $\oocat{Lie}(A/k), \oocat{Lie}(A/k)^{\gr}$ and $\oocat{cLie}(A/k)^{\mm{gr-mix}}$ for the $\infty$-categories of $L_\infty$-algebroids, weight-graded $L_\infty$-algebroids and graded mixed-curved $L_\infty$-algebroids, respectively. In each case, objects are required to have a cofibrant underlying (weight-graded) $A$-module and a surjective anchor map, and morphisms are $\infty$-morphisms.
\end{definition}
\begin{corollary}\label{cor:models for algebroids}
The simplicially enriched category $\oocat{cLie}(A/k)^{\mm{mix}}$ presents the $\infty$-categorical localization of the category of mixed-curved $L_\infty$-algebroids at the weak equivalences. In particular, $\oocat{cLie}(A/k)^{\mm{mix}}$ is a presentable $\infty$-category. The same assertions holds for graded mixed-curved $L_\infty$-algebroids and (weight-graded) $L_\infty$-algebroids.
\end{corollary}
\begin{proof}
Exactly as Proposition \ref{prop:oocat of algebras over k}. Note that the $\infty$-categorical localization of the semi-model category $\cat{cLie}(A/k)^\mm{mix}$ is presentable because it is equivalent to that of the (Quillen equivalent) combinatorial model category $\tilde{0}/\cat{cLie}(A/k)^\mm{mix}$ of mixed-curved $L_\infty$-algebroids under a fibrant-cofibrant replacement of the initial object.
\end{proof}
\begin{definition}
The $\infty$-category $\oocat{cLie}(A/k)$ of curved $L_\infty$-algebroids over $A$ is the $\infty$-category associated to the simplicially enriched category whose:
\begin{enumerate}[start=0]
\item objects are curved $L_\infty$-algebroids such that the anchor $\Loid\rt T_A\wgt{0}$ is surjective, the underlying complete $A$-module is quasiprojective and $\Gr(A)$ is a cofibrant graded $A$-module.
\item simplicial sets of morphisms consist of $\infty$-morphisms $\Loid\rightsquigarrow \Loide\boxtimes \Omega[\Delta^n]$.
\end{enumerate}
\end{definition}
\begin{proposition}\label{prop:classical curved Lie algebroid=graded mixed}
There is a sequence of functors between presentable $\infty$-categories
$$\begin{tikzcd}
\oocat{cLie}(A/k)^{\mm{gr-mix}}\arrow[r, "\Tot"] & \oocat{cLie}(A/k)^\mm{mix} \arrow[r, "\blend"] & \oocat{cLie}(A/k)
\end{tikzcd}$$
whose composite is an equivalence. In particular, the $\infty$-category of curved $L_\infty$-algebroids is presentable.
\end{proposition}
\begin{proof}
As Proposition \ref{prop:graded mixed = classical} and Theorem \ref{thm:homotopy theories of curved lie over B}; the assumption that all objects are quasiprojective as complete $A$-modules implies that we can split their filtration $A$-linearly.
\end{proof}
Let us conclude with the following observation about the three canonical filtrations on an $L_\infty$-algebroid from Example \ref{ex:filtrations on Lie algebroids}:
\begin{proposition}\label{prop:adding filtrations fully faithful}
The three filtrations from Example \ref{ex:filtrations on Lie algebroids} determine fully faithful right adjoint functors of $\infty$-categories
$$\begin{tikzcd}[row sep=0pc, column sep=3.2pc]
\oocat{Lie}(A/k)\arrow[r, "\Loid\mapsto \Loid\wgt{0}"] & \oocat{cLie}(A/k)_{\hphantom{/T_A\wgt{1}}} \\
\oocat{Lie}(A/k)\arrow[r, "\Loid\mapsto \Loid\wgt{1}"] & \oocat{cLie}(A/k)_{/T_A\wgt{1}} \\
\oocat{Lie}(A/k)\arrow[r, "\Loid\mapsto \Loid^\mm{anc}"] & \oocat{cLie}(A/k)_{/T_A\wgt{1}}.
\end{tikzcd}$$
\end{proposition}
\begin{proof}
We will only treat the `anchor filtration', the others are similar but easier. We will present the $\infty$-functor $\Loid\mapsto \Loid^\mm{anc}$ by a right Quillen functor; to do this, let $\mf{t}_A\rto{\sim} T_A$ be an equivalent $L_\infty$-algebroid whose underlying $A$-module is cofibrant, and let $\tilde{0}_A$ denote the free $L_\infty$-algebroid generated by the map of $A$-modules $\mf{t}_A[0, -1]\rt \mf{t}_A\rt T_A$ from the path space of $\mf{t}_A$. Then $\tilde{0}_A\simeq 0$ is weakly equivalent to the initial $L_\infty$-algebroid.

Consider the category ${\scriptstyle\tilde{0}_A/}\cat{Lie}(A/k){\scriptstyle /\mf{t}_A}$ of $L_\infty$-algebroids that fit into a diagram $\tilde{0}_A\rt \Loid\rt \mf{t}_A$. Equivalently, this is the category of $\Loid\rt \mf{t}_A$ which come equipped with an $A$-linear section (which need not preserve differentials), inducing a decomposition $\Loid=\mf{t}_A\oplus \mf{n}$. This carries a model structure induced from the semi-model structure on all $L_\infty$-algebroids, and forgetting the maps from $\tilde{0}_A$ and to $\mf{t}_A$ relate the two by a zig-zag of Quillen equivalences.

Likewise, let $\cat{cLie}(A/k)^{\mm{gr-mix}}{\scriptstyle/\mf{t}_A\wgt{1}}$ denote the category of graded mixed-curved $L_\infty$-algebroids with a map to $\mf{t}_A\wgt{1}$. The forgetful functor to all graded mixed-curved $L_\infty$-algebroids is a right Quillen equivalence. The functor $\Loid\mapsto \Loid^\mm{anc}$ can then be presented by the right Quillen functor
%\begin{equation}\label{eq:basic hodge filtration very rectified}
$$\begin{tikzcd}
{\scriptstyle\tilde{0}_A/}\cat{Lie}(A/k){\scriptstyle /\mf{t}_A}\arrow[r] & \cat{cLie}(A/k)^{\mm{gr-mix}}{\scriptstyle /\mf{t}_A\wgt{1}}; &\big(\Loid=\mf{t}_A\oplus \mf{n}\to \mf{t}_A\big)\arrow[r, mapsto] & \Loid^\mm{anc}
\end{tikzcd}$$
%\end{equation}
where the graded mixed-curved $L_\infty$-algebroid $\Loid^\mm{anc}$ is given in weight $0$ by $\mf{n}$ and in weight $-1$ by $\mf{t}_A$ (and the mixed-curved $L_\infty$-structure on the total complex $\mf{n}\oplus \mf{t}_A$ is that of $\Loid$). The left adjoint $\Phi$ sends a graded mixed-curved $L_\infty$-algebroid $\Loide$ to (a) its quotient by the ideal generated by all $\Loide\wgt{p}$ with $p\neq 0, 1$ and by $\ker\big(\Loide\wgt{1}\rt \mf{t}_A\wgt{1}\big)$ (in particular, the result has no curvature), and (b) then takes the associated total $L_\infty$-algebroid (forgetting the filtration).

Now notice that the set of (curved) $\infty$-morphisms $\Loid^\mm{anc}\rt \Loide^\mm{anc}$ is isomorphic to the set of $\infty$-morphisms $\Loid\rt \Loide$. In particular, there is no $\phi_0$ because $\Loide^{\mm{anc}}$ is zero is positive weights. By adjunction, this means that the functor $\Phi$ preserves the `bar-cobar resolution' of Proposition \ref{prop:cobar resolution for loo algebroids}: $\Phi(Q(\Loid^\mm{anc}))\cong Q(\Loid)$. Since $Q(\Loid^\mm{anc})\rt \Loid^\mm{anc}$ is a cofibrant resolution whenever $\Loid$ is cofibrant as an $A$-module, this implies that the derived counit $\mathbb{L}\Phi(\Loid^\mm{anc})\rt \Loid$ is an equivalence, hence $(-)^\mm{anc}$ induces a fully faithful functor of $\infty$-categories.
%
%To see that this functor is fully faithful, note that (using Proposition \ref{prop:classical curved Lie algebroid=graded mixed}) its derived functor can be modeled by a functor of simplicially enriched categories
%$$\begin{tikzcd}
%\oocat{Lie}(A/k)_{/\mf{t}_A}\arrow[r] & \oocat{cLie}(A/k)_{/\mf{t}_A(1)}; & \mf{g}\arrow[r, mapsto] & \mf{g}^\mm{Hdg}=\mf{g}\times_{\mf{t}_A\wgt{0}} \mf{t}_A\wgt{1}.
%\end{tikzcd}$$
%Here the domain consists of $A$-cofibrant $L_\infty$-algebroids with a surjection to $\mf{t}_A$ and the target of $A$-cofibrant curved $L_\infty$-algebroids with a surjection to $\mf{t}_A\wgt{1}$. On can verify that this construction preserves $\infty$-morphisms. One then sees that the simplicial sets of $\infty$-morphisms $\mf{g}\rightsquigarrow \mf{h}$ are isomorphic to the simplicial sets of $\infty$-morphisms $\mf{g}^\mm{Hdg}\rightsquigarrow \mf{h}^\mm{Hdg}$: this uses that $\mf{h}^\mm{Hdg}$ is in filtration degrees $\leq 0$, so that all (curved) $\infty$-morphisms have nullary component $\phi_0=0$.
\end{proof}

\subsection{Curved Lie algebroids as non-curved Lie algebroids}
The goal of this section is to give a more intrinsic description of the $\infty$-category of curved $L_\infty$-algebroids in terms on \emph{non-curved} $L_\infty$-algebroids, using (the Koszul dual of) the \emph{Rees construction}. Note that taking $A=k$ the base field, this also gives a description of the $\infty$-category of curved $L_\infty$-algebras studied in Section \ref{sec:curved algebras}.
\begin{definition}\label{def:formal affine line}
Let us denote by $\mc{R}(T_A)$ the weight-graded Lie algebroid over $A$
$$
\mc{R}(T_A)\coloneqq T_A\ltimes A\wgt{-1}[-1]
$$
given by the direct sum of $T_A$ (in weight $0$) and the free $A$-module on a generator $\theta$ of weight $1$ and degree $1$, such that $[\theta, \theta]=0$ and $[v, a\cdot \theta]=\mc{L}_v(a)\cdot \theta$ for $v\in T_A$. 
\end{definition}
\begin{theorem}\label{thm:curved Lie algebroids conceptually}
There are equivalences of presentable $\infty$-categories
$$\begin{tikzcd}
\oocat{cLie}(A/k) & \oocat{cLie}(A/k)^\mm{gr-mix}\arrow[r, "\sim"]\arrow[l, "\sim"{swap}] & \oocat{Lie}(A/k)^\gr{\scriptstyle / \mc{R}(T_A)}
\end{tikzcd}$$ 
between the $\infty$-categories of curved $L_\infty$-algebroids, graded mixed-curved $L_\infty$-algebroids and graded $L_\infty$-algebroids over $\mc{R}(T_A)=T_A\ltimes A\wgt{-1}[-1]$. 
\end{theorem}
\begin{remark}\label{rem:curved Lie algebroids FMP interpretation}
From a geometric perspective, one can informally think of weight-graded Lie algebroids over $A$ as maps of formal stacks
$$
\Spec(A)\times B\mathbb{G}_m\rt Y\rt \Spec(A)_{\dR}\times B\mathbb{G}_m
$$
where the first (equivalently second) map is a nil-isomorphism. The structure map to $B\mathbb{G}_m$ gives rise to the weight-grading. The weight-graded Lie algebroid $\mc{R}(T_A)$ then corresponds to the formal stack $\Spec(A)_{\dR}\times \widehat{\mathbb{A}}^1/\mathbb{G}_m$. In other words, curved $L_\infty$-algebroids over $A$ describe nil-isomorphisms of formal stacks
$$
\Spec(A)\times B\mathbb{G}_m\rt Y \rt \Spec(A)_{\dR}\times \widehat{\mathbb{A}}^1/\mathbb{G}_m.
$$
In particular, the structure morphism to $\widehat{\mathbb{A}}^1/\mathbb{G}_m$ endows $Y$ with a complete filtration, and $Y$ maps to $\Spec(A)_{\dR}$ equipped with the \emph{trivial} filtration.
\end{remark}
Let us start by constructing the functor from graded mixed-curved $L_\infty$-algebroids to graded $L_\infty$-algebroids over $\mc{R}(T_A)$.
\begin{construction}\label{cons:Rees}
Suppose that $\Loid$ is a graded mixed-curved $L_\infty$-algebroid over $A$ and consider the weight-graded (non-differential) graded $A$-module 
$$
\mc{R}(\Loid)=\Loid\oplus A\wgt{-1}[-1].
$$
The anchor map of $\Loid$ induces a map $\mc{R}(\Loid)\rt T_A\ltimes A\wgt{-1}[-1]$. Recall from Definition \ref{def:graded mixed-curved Loo} that $\Loid$ comes with brackets
$$
\ell_n^p\colon \Loid^{\otimes n}\rt \Loid
$$ 
of weight $p\geq 0$ (except for $\ell_0^p$ and $\ell_1^p$, which are only defined for weights $p\geq 1$). Using these, we define $n$-ary operations $\ell_n$ of weight $0$ on $\mc{R}(\Loid)$ by
$$
\ell_n\big(\theta, \dots, \theta, x_{p+1}, \dots, x_n\big)\coloneqq p!\cdot \ell^p_{n-p}\big(x_{p+1}, \dots, x_n\big)
$$
for $p$ copies of $\theta$ and $x_{p+1}, \dots, x_n\in \Loid$. In particular, all maps $\ell_n$ take values in $\Loid\subseteq \mc{R}(\Loid)$.
\end{construction}
\begin{proposition}\label{prop:rees construction}
The operations $\{\ell_n\}$ make the map $\mc{R}(\Loid)\rt \mc{R}(T_A)$ a map of weight-graded $L_\infty$-algebroids if and only if the operations $\{\ell_n^p\}$ make $\Loid$ a curved $L_\infty$-algebroid.
\end{proposition}
\begin{proof}
Note that in the weight-graded case, all $\ell_n^p$ are $A$-multilinear except $\ell_2^0$, which satisfies the Leibniz rule. One easily sees that this is equivalent to the $\ell_n$ all being $A$-linear, except for $\ell_2$ which satisfies the Leibniz rule.

It then suffices to verify that the $\{\ell_n\}$ define a $k$-linear $L_\infty$-structure on $\mc{R}(\Loid)$ if and only if the $\ell_n^p$ define a $k$-linear graded mixed-curved $L_\infty$-algebra structure on $\Loid$. To see this, note that unshuffles $\sigma$ of an $n$-element set $(\theta, \dots, \theta, x_{p+1}, \dots, x_n)$ are in 1-1 correspondence with pairs consisting of an unshuffle $\tau$ of the $p$-element set $(\theta, \dots, \theta)$  and an unshuffle $\sigma'$ of the set $(x_{p+1}, \dots, x_n)$. Denoting $m=n-p$, the $L_\infty$-condition for the $\ell_n$ then translates into
\begin{align*}
\big[d, \ell_n\big](\theta, \dots, \theta, x_{p+1}, \dots, x_n) &\stackrel{!}{=} \sum_{\substack{i+j=n+1}} \sum_{\sigma\in \mm{Sh}^{-1}_{i-1, j}} \pm \big(\ell_i\circ_1 \ell_j\big)^\sigma(\theta, \dots, \theta, x_{p+1},\dots, x_n)\\
&= \sum_{\substack{i'+j'=m+1\\ q+r=p}} \sum_{\substack{\sigma'\in \mm{Sh}^{-1}_{i'-1, j'}\\ \tau\in \mm{Sh}^{-1}_{q, r}}} \pm q!r!\big(\ell_{i'}^q\circ_1 \ell_{j'}^r\big)^{\sigma'}(x_{p+1}, \dots, x_n)\\
&= \sum_{q+r=p}\sum_{i'+j'=m+1} \sum_{\sigma'\in \mm{Sh}^{-1}_{i'-1, j'}} \pm p!\cdot \big(\ell_{i'}^q\circ_1 \ell_{j'}^r\big)^{\sigma'}(x_{p+1}, \dots, x_n)
\end{align*}
where the $\pm$ signs are determined by Remark \ref{rem:signs}. 
Here we used that the unshuffles $\tau$ (which only permutes copies of $\theta$) leave the values invariant. Since $\big[d, \ell_n\big](\theta, \dots, x_n)=p!\cdot \big[d, \ell_m^p\big](x_{p+1}, \dots, x_n)$, the above equation can be identified with the graded mixed-curved $L_\infty$-equation for $\Loid$.
\end{proof}
\begin{example}\label{ex:hodge stack}
Let $T_A\wgt{1}$ be as in Example \ref{ex:filtrations on Lie algebroids}. Viewing $T_A\wgt{1}$ as a graded mixed-curved Lie algebroid (with trivial curvature), the Rees construction of $T_A\wgt{1}$ is given by $T_A\wgt{1}\ltimes A\wgt{-1}[1]$. The Chevalley--Eilenberg algebra of this graded Lie algebroid is given by the $\hbar$-de Rham complex (cf. Example \ref{ex:de Rham})
$$
{\dR}(A)[\![\hbar]\!], \qquad\qquad\qquad d\alpha=d_A(\alpha)+\hbar\cdot d_{\dR}\alpha, \qquad d\hbar=0,
$$
 where $\hbar$ has degree $0$ (weight $-1$) and $1$-forms have weight $1$. Using the equivalence between graded $k[\![\hbar]\!]$-algebras and filtered algebras (via the classical Rees construction), this corresponds to the de Rham algebra of $A$, endowed with the Hodge filtration.
From the perspective of Remark \ref{rem:curved Lie algebroids FMP interpretation}, the weight-graded $L_\infty$-algebroid $\mc{R}\big(T_A(1)\big)$ corresponds to the \emph{Hodge stack}
$$
\Spec(A)\times B\mathbb{G}_m\rt \Spec(A)^{\mm{Hodge}}\rt \Spec(A)_{\dR}\times \widehat{\mathbb{A}}^1/\mathbb{G}_m.
$$ 
This is the formal stack whose structure map to $\widehat{\mathbb{A}}^1/\mathbb{G}_m$ encodes the Hodge filtration on the de Rham complex.
\end{example}
\begin{remark}
Consider the zero curved $L_\infty$-algebroid $0$ and notice that $0$ is \emph{not} the initial curved $L_\infty$-algebroid: instead, the category of curved $L_\infty$-algebroids under $0$ is simply the category of (non-curved) $L_\infty$-algebroids. In fact, the space of $\infty$-morphisms $0\rightsquigarrow \Loid$ can be identified with the space of Maurer--Cartan elements in $F^1\ker(\Loid\rt T_A)$ (cf.\ Section \ref{sec:oo-morphism of clie}). The Rees construction $\mc{R}(0)$ is the trivial graded Lie algebroid $A\wgt{-1}[1]$, with zero bracket and anchor map. The corresponding Chevalley--Eilenberg algebra is simply $A[\![\hbar]\!]$ with $\hbar$ of degree $0$ and weight $-1$. Geometrically, this corresponds to $\Spec(A)\times \widehat{\mathbb{A}}^1/\mathbb{G}_m$. Informally, we can therefore think of the curvature of a curved $L_\infty$-algebroid $\Loid$ in terms of the corresponding formal stack $Y$ as the obstruction to finding a dotted lift
$$\begin{tikzcd}
\Spec(A)\times B\mathbb{G}_m\arrow[r]\arrow[d] & Y\arrow[d]\\
\Spec(A)\times \widehat{\mathbb{A}}^1/\mathbb{G}_m\arrow[r]\arrow[ru, dotted] & \Spec(A)_{\dR}\times \widehat{\mathbb{A}}^1/\mathbb{G}_m.
\end{tikzcd}$$
\end{remark}
\begin{proof}[Proof of Theorem \ref{thm:curved Lie algebroids conceptually}]
The equivalence $\oocat{cLie}(A/k)^\mm{gr-mix}\rt \oocat{cLie}(A/k)$ already appeared in Proposition \ref{prop:classical curved Lie algebroid=graded mixed}. The equivalence $\oocat{cLie}(A/k)^\mm{gr-mix}\rt \oocat{Lie}(A/k)^\gr{\scriptstyle /\mc{R}(T_A)}$ arises from a Quillen equivalence. Indeed, let $\tilde{0}$ denote the free weight-graded $L_\infty$-algebroid generated by the contractible complex $k(-1)[-1, -2]$ in weight $1$. There is a canonical map $\tilde{0}\rt \mc{R}(T_A)$ sending the degree $1$ generator of $\tilde{0}$ to the element $\theta\in A(-1)[-1]{/ \mc{R}(T_A)}$. 

Now consider the category $\cat{C}$ of weight-graded $L_\infty$-algebroids over $A$ equipped with maps $\tilde{0}\rt \Loid\rt\mc{R}(T_A)$. Equivalently, this is the category of weight-graded $L_\infty$-algebroids $\Loid$ over $\mc{R}(T_A)$, equipped with a compatible $A$-linear decomposition $\Loid\cong \Loid'\oplus A\wgt{-1}[-1]$ (which need not respect differentials). This carries a semi-model structure induced by the semi-model structure on all weight-graded $L_\infty$-algebroids from Theorem \ref{thm:model structure lie algebroids}. Since $\tilde{0}$ is cofibrant and weakly contractible, forgetting the splitting induces a Quillen equivalence to the category of weight-graded $L_\infty$-algebroids over $\mc{R}(T_A)$, whose associated $\infty$-category is exactly $\oocat{Lie}(A/k)^{\gr}{\scriptstyle /\mc{R}(T_A)}$.

On the other hand, using Proposition \ref{prop:rees construction} one sees that the functor $\Loid\mapsto \mc{R}(\Loid)$ defines an equivalence of categories from the category of graded mixed-curved $L_\infty$-algebroids to $\cat{C}$. This equivalence identifies the model structures on both sides, from which the result follows.
%
%Consider the category $^{\mf{f}}\big/\cat{Lie}(A/k)^\gr\big/_{\mc{R}(T_A)}$ of weight-graded $L_\infty$-algebroid in between $\mf{f}$ and $\mc{R}(T_A)$; this is equivalent to the category of weight-graded $L_\infty$-algebroids with a map to $\mc{R}(T_A)$, together with a specified inverse image of $\theta\in A(-1)[-1]\subseteq \mc{R}(T_A)$ in $\mf{g}$. Proposition \ref{prop:rees construction} implies that the Rees construction gives an equivalence of categories
%$$\xymatrix{
%\mc{R}\colon \cat{cLie}(A/k)^\mm{gr-mix}\ar[r]^\simeq & ^{\mf{f}}\big/\cat{Lie}(A/k)^\gr\big/_{\mc{R}(T_A)}.
%}$$
%The model structure on graded mixed-curved $L_\infty$-algebroids now arises from the model structure on weight-graded $L_\infty$-algebroids under this equivalence. Finally, since $\mf{f}$ is cofibrant and contractible, the forgetful functor 
%$$\xymatrix{
%\cat{cLie}(A/k)^\mm{gr-mix}\simeq {}^{\mf{f}}\big/\cat{Lie}(A/k)^\gr\big/_{\mc{R}(T_A)}\ar[r] & \cat{Lie}(A/k)^\gr\big/_{\mc{R}(T_A)}
%}$$
%is a Quillen equivalence.
\end{proof}
\begin{remark}
Although we do not need this, one can verify that the Rees construction (Construction \ref{cons:Rees}) is compatible with $\infty$-morphisms as well.
\end{remark}

\section{The equivalence between curved Lie algebras and Lie algebroids}\label{sec:The equivalence between curved Lie algebras and Lie algebroids}
The goal of this section is to prove that for a nonpositively graded cdga $A$ satisfying a certain finiteness condition, there is an equivalence between the $\infty$-category of Lie algebroids over $A$ and the $\infty$-category of certain types of curved Lie algebras over its (completed) de Rham algebra $\dR(A)$, endowed with the Hodge filtration. 

To this end, let us recall that a nonpositively graded cdga is said to be \emph{locally of finite presentation} if it is a compact object in the category of nonpositively graded cdgas; equivalently, it is the retract of a cdga with finitely many generators and relations. Note that a nonpositively graded cdga $A$ is cofibrant and locally of finite presentation if and only if it is the retract of a quasi-free cdga with finitely many generators, each of which is in nonpositive degree. Indeed, one can always realize such a cofibrant $A$ as a retract of a nonpositively graded quasi-free cdga with infinitely many generators; if $A$ is a compact object, it must already be a retract of a subalgebra on finitely many generators.
\begin{theorem}\label{thm:comparison tangent case}
Let $A$ be a nonpositively graded cdga. Suppose that $A$ is smooth or cofibrant and locally of finite presentation, and let $\dR(A)$ denote the de Rham algebra of $A$, endowed with the Hodge filtration. Then there is a fully faithful functor of $\infty$-categories
$$\begin{tikzcd}
\curv\colon \oocat{Lie}(A/k)\arrow[r, hook] & \oocat{cLie}_{\dR(A)}
\end{tikzcd}$$
whose essential image consists of curved $L_\infty$-algebras $\mf{g}$ over $\dR(A)$ such that the canonical map $\gr_0(\mf{g})\otimes_{\gr_0(\dR(A))} \gr(\dR(A))\rt \gr(\mf{g})$ is an equivalence.
\end{theorem}
In fact, we will deduce this from a more general result about curved $L_\infty$-algebroids over complete filtered algebras of the form $C^*(\tang)$, where $\tang$ is a complete $L_\infty$-algebroid over $A$; for the filtration on $C^*(\tang)$ to behave like the Hodge filtration on $\dR(A)$, we require $F^0(\tang)=0$. We will then prove the following result:
\begin{theorem}\label{thm:comparison}
Let $A$ be a nonpositively graded cdga and let $\tang$ be a complete $L_\infty$-algebroid over $A$ such that $F^0(\tang)=0$ and each $F^i(\tang)$ is finitely generated quasiprojective as an $A$-module. Then there is an equivalence of $\infty$-categories
$$\begin{tikzcd}
\curv\colon \oocat{cLie}(A/k){\scriptstyle /\tang}\arrow[r, "\sim"] & \oocat{cLie}_{C^*(\tang)}.
\end{tikzcd}$$
\end{theorem}
\begin{example}\label{ex:comparison for tangent}
In the situation of Theorem \ref{thm:comparison tangent case}, one can take $\tang=T_A\wgt{1}$ the tangent complex of $A$, put in filtration degree $-1$ (Example \ref{ex:filtrations on Lie algebroids}), whose Chevalley--Eilenberg algebra is precisely the de Rham algebra $\dR(A)$, equipped with the Hodge filtration (Example \ref{ex:chevalley-eilenberg}). Theorem \ref{thm:comparison} provides an equivalence between curved $L_\infty$-algebras over $\dR(A)$ (with the Hodge filtration) and curved $L_\infty$-algebroids over $A$ with a map to $T_A\wgt{1}$. By Theorem \ref{thm:curved Lie algebroids conceptually} one can also identify this second $\infty$-category with the $\infty$-category of weight-graded $L_\infty$-algebroids over the Rees construction of $T_A\wgt{1}$, which corresponds to the Hodge stack of $\Spec(A)$ (Example \ref{ex:hodge stack}).
\end{example}
The majority of this section is devoted to the proof of Theorem \ref{thm:comparison}. Throughout, we will assume that $\tang$ is a complete $L_\infty$-algebroid over $A$ as in the theorem, and we will write
$$
B\coloneqq C^*(\tang)\cong \Hom_A\big(\Sym_A(\tang[1]), A\big)\cong \Sym_A\big(\tang^\vee[-1]\big)
$$
for its (filtered) Chevalley--Eilenberg algebra. The last isomorphism of complete graded algebras holds because $F^0\tang=0$ and each $F^i\tang$ is a dualizable $A$-module, so that each $F^i\Sym_A(\tang[1])$ is a dualizable $A$-module as well.

\subsection{Curved $L_\infty$-algebras from curved $L_\infty$-algebroids}\label{sec:formulas for curv}
In this section we will show how to associate a curved $L_\infty$-algebra over $B$ to a curved $L_\infty$-algebroid over $\tang$, equipped with a linear section $\sigma\colon \tang\rt \Loid$. Furthermore, we will show how $\infty$-morphisms between such curved $L_\infty$-algebroids over $\tang$ give rise to $\infty$-morphisms between the associated curved $L_\infty$-algebras over $B$. In Section \ref{sec:proof of main theorem}, we will then show how these constructions give rise to a functor between simplicially enriched categories which presents the desired functor `$\curv$' in Theorem \ref{thm:comparison}.

\subsubsection{Definition on objects}
Let us start by describing how to associate a curved $L_\infty$-algebra over $B$  to certain curved $L_\infty$-algebroids with a map to $\tang$. More precisely, our goal will be to prove the following:
\begin{proposition}\label{prop:curved algebroid vs curved algebra}
Let $\keranc$ be a complete graded $A$-module and let $\pi\colon \tang\oplus \keranc\rt \tang
$ denote the canonical projection. Then there is a natural bijection (given by Construction \ref{cons:curved algebroid vs curved algebra}) between:
\begin{enumerate}
\item curved $L_\infty$-algebroid structures on $\tang\oplus \keranc$ such that $\pi$ is a (strict) map of curved $L_\infty$-algebroids.
\item curved $L_\infty$-algebra structures on the complete graded $B$-module $B\otimes_A\keranc$.
\end{enumerate}
\end{proposition}
To prove this, let us start with some preliminary observations.
\begin{remark}\label{rem:curv(g)asforms}
Because $F^0\tang=0$ and each $F^i\tang$ is a dualizable complete $A$-module by assumption, we have that $\Sym_A(\tang[1])$ is a dualizable complete $A$-module as well, whose dual is $\Sym_A(\tang^\vee[-1])$. Consequently, for every complete graded $A$-module $\keranc$, there are isomorphisms of complete graded vector spaces
$$
B\otimes_A \keranc \cong \mm{Sym}_A(\tang^\vee[-1])\otimes_A \keranc \cong \Hom_A\big(\mm{Sym}_A(\tang[1]), \keranc\big)=\prod_p \Hom_A\big(\mm{Sym}_A^p(\tang[1]), \keranc\big).
$$
In other words, we can identify $B\otimes_A \keranc$ with $\keranc$-valued forms on $\tang$.
\end{remark}
To compare the structure maps of $\tang\oplus \keranc$ and $B\otimes_A \keranc$, let us consider the sets of maps
\begin{align}\label{eq:derivations}
\mc{H}(p)&\coloneqq\Hom_B\Big(\mm{Sym}_B^p\big(B\otimes_A\keranc[1]\big), B\otimes_A \keranc[2]\Big)\nonumber\\
\mc{H}(1)&\coloneqq\mm{Der}_B\Big(B\otimes_A\keranc[1], B\otimes_A\keranc[2]\Big)
\end{align}
where in the second line we take (graded) derivations with respect to the Chevalley--Eilenberg differential on $B=C^*(\tang)$. The structure maps of a curved $L_\infty$-structure on $B\otimes_A \keranc$ are exactly contained in these sets.
\begin{lemma}\label{lem:further decompositions}
For $p\neq 1$, there are bijections
$$
\mc{H}(p)\cong \prod_i \Hom_A\Big(\Sym^i_A(\tang[1])\otimes_A \mm{Sym}_A^p\big(\keranc[1]\big), \keranc[2]\Big).
$$
Furthermore, there is an inclusion
$$\begin{tikzcd}
\mc{H}(1)\arrow[r, hook] & \prod_i \Hom_k\Big(\Sym^i_A(\tang[1])\otimes \keranc[1], \keranc[2]\Big).
\end{tikzcd}$$
whose image consists precisely of tuples of maps $f^{(i)}_1\colon \mm{Sym}_A^i(\tang[1])\otimes \keranc[1]\rt \keranc[2]$ that are $A$-multilinear for $i\neq 0, 1$, while for $x\in \keranc[1]$, $t\in \tang[1]$:
\begin{align}\label{eq:extendability to derivation}
f^{(0)}_1(a\cdot x)&= d_A(a)\cdot x + a\cdot f^{(1)}_1(x) \nonumber\\
f^{(1)}_1(a_1\cdot t, a_2\cdot x)&=a_1\mc{L}_{\rho(t)}(a_2)\cdot x + a_1a_2\cdot f^{(1)}_1(t, x).
\end{align}
\end{lemma}
In the above lemma and throughout, we will use the following convention: maps with $p$ inputs from $\keranc[1]$ and $i$ inputs from $\tang[1]$ will be denoted by $f_p^{(i)}$.
\begin{proof}
The domain of a map in $\mc{H}(p)$ is the free $B$-module on $\Sym^p_A(\keranc[1])$. In particular, for $p\neq 1$ we can identify maps $f\in \mc{H}(p)$ with $A$-linear maps $f\colon \Sym^p_A(\keranc[1])\rt B\otimes_A \keranc[2]$. Using Remark \ref{rem:curv(g)asforms}, such $A$-linear maps $f$ correspond by adjunction to tuples of $A$-linear maps
$$\begin{tikzcd}
f_p^{(i)}\colon\mm{Sym}_A^i(\tang[1])\otimes_A\Sym_A^p(\keranc[1])\arrow[r] & \keranc[2].
\end{tikzcd}$$
For $p=1$, note that a derivation $f_1\colon B\otimes_A \keranc[1]\rt B\otimes_A \keranc[1]$ is uniquely determined by a map $f_1\colon \keranc[1]\rt B\otimes_A \keranc[2]$ with the property that for all $x\in \keranc[1]$, $a\in A$
$$
f_1(a\cdot x) = d_B(a)\cdot x+ a\cdot f_1(x).
$$
Again, it follows from Remark \ref{rem:curv(g)asforms} that $f_1$ is determined by adjunction by a family of maps $f_1^{(i)}\colon \Sym^i_A(\tang[1])\otimes_k \keranc[1]\rt \keranc[2]$. Using that $d_B(a)=d_A(a)+\mc{L}_{\rho(-)}(a)$, the above derivation rule then translates into the equations \eqref{eq:extendability to derivation}.
\end{proof}
\begin{remark}\label{rem:ell1extensionasderivation}
Let $\{f_1^{(i)}\}_{i \geq 0}$ be a family of maps $f_1^{(i)}\colon \Sym_A^i(\tang[1])\otimes\keranc[1]\rt \keranc[2]$  as in Lemma \ref{lem:further decompositions}. Then the induced derivation $f_1$ on $B\otimes_A \keranc[1]$ can be computed explicitly as follows: for $\alpha\colon \mm{Sym}_A(\tang[1])\rt \keranc$ an $\keranc$-valued form and $t_i\in \tang[1]$, we have
\begin{align*}
f_1(\alpha)(t_1,\dots, t_n) =\sum_{i\geq 1} \sum_{\sigma\in \mm{Sh}^{-1}_{i, n-i}} \Big( & f^{(i)}_{1}\big(t_{\sigma(1)}, \dots, t_{\sigma(i)}, \alpha(t_{\sigma(i+1)}, \dots, t_{\sigma(n)})\big)\\
&- \alpha\big(\ell_{\tang}^{(i)}(t_{\sigma(1)}, \dots, t_{\sigma(i)}), t_{\sigma(i+1)}, \dots, t_{\sigma(n)}\big)\Big).
\end{align*}
Here $\ell_{\tang}^{(i)}\colon \Sym^{n}(\tang[1])\rt \tang[2]$ denotes the $L_\infty$-algebroid structure on $\tang$.
\end{remark}
\begin{observation}\label{obs:decomposing maps}
The collection of $\mc{H}(p)$ has a structure similar to a convolution Lie algebra, in the sense that for any $f_p\in \mc{H}(p)$ and $f_q\in \mc{H}(q)$, there is a $B$-linear map
$$
[f_p, f_q]\colon \Sym^{p+q-1}_B\big(B\otimes_A \keranc[1]\big)\rt B\otimes_A \keranc[3]
$$
given by the graded commutator
$$
[f_p, f_q]= \sum_{\sigma\in\mm{Sh}^{-1}_{p-1, q}} \big(f_p\circ_1 f_q\big)^\sigma + \sum_{\tau\in \mm{Sh}^{-1}_{q-1, p}} \big(f_q\circ_1 f_p\big)^{\tau}.
$$
Let us identify this commutator in terms of the decomposition from Lemma \ref{lem:further decompositions}, i.e.\ identifying $f_p=\sum_i f_p^{(i)}$ as a sum of maps with $i$ inputs from $\tang$ and $p$ inputs from $\keranc$. We have for $p, q\neq 1$, that
$$
[f_p, f_q] = \sum_{i, j}\sum_{\tau\in \mm{Sh}^{-1}_{i, j}}\bigg(\sum_{\sigma\in \mm{Sh}^{-1}_{p-1, q}} \big(f_p^{(i)}\circ_{\keranc} f_q^{(j)}\big)^{\sigma\times \tau}+ \sum_{\sigma\in \mm{Sh}^{-1}_{q-1, p}} \big(f_q^{(j)}\circ_{\keranc} f_p^{(i)}\big)^{\sigma\times \tau}\bigg).
$$
Here we symmetrize with respect to the inputs coming from $\keranc$, as well as all inputs from $\tang$, and use $f_p^{(i)}\circ_{\keranc} f_q^{(j)}$ to denote partial composition along the first $\keranc$-variable.

The description of the commutator $[f_p, f_1]$ is slightly more involved, since $f_1$ is extended from its restriction to $\keranc[1]$ as a derivation. By Remark \ref{rem:ell1extensionasderivation}, we have that
$$
[f_p, f_1] = \sum_{i, j}\bigg(\sum_{\tau\in\mm{Sh}^{-1}_{i, j}}\bigg( \big(f_1^{(j)}\circ_\keranc f_p^{(i)}\big)^{\tau} + \sum_{\sigma\in\mm{Sh}^{-1}_{p-1, 1}} \big(f_p^{(i)}\circ_\keranc f_1^{(j)}\big)^{\sigma\times\tau}\bigg)+ \sum_{\tau\in \mm{Sh}^{-1}_{i-1, j}}\big(f_p^{(i)}\circ_{\tang} \ell_{\tang}^{(j)}\big)^{\tau}\bigg).
$$
Here the first term is just the commutator (suitably symmetrized in the $\tang$-variables), and in the second term $-\circ_{\tang}\ell_{\tang}^{(j)}$ takes the partial composition in the first $\tang$-variable with a structure map of $\tang$.
\end{observation}
Let us now turn to the main construction behind Proposition \ref{prop:curved algebroid vs curved algebra}:
\begin{construction}\label{cons:curved algebroid vs curved algebra}
Let $\Loid =\tang\oplus \keranc$ be a complete graded $A$-module, let $\pi\colon \Loid\rt \tang$ be the natural projection and let $\ell_{\tang}^{(n)}$ denote the $n$-ary bracket of the $L_\infty$-algebroid structure on $\tang$. Then a curved $L_\infty$-algebroid structure on $\Loid$ such that $\pi$ is a (strict) map of curved $L_\infty$-algebroids has structure maps of the form
$$\begin{tikzcd}
\big(\ell_{\tang}^{(s)}\circ \pi, \mc{L}_s\big)\colon \Sym^s_k\big(\tang[1]\oplus \keranc[1]\big)\arrow[r] & \tang[2]\oplus \keranc[2].
\end{tikzcd}$$
In other words, the $\tang$-component of the bracket on $\tang\oplus\keranc$ is given by the brackets of $\tang$. Expanding binomially, such a curved $L_\infty$-algebroid structure is therefore determined by maps
$$
\ell_p^{(i)} \colon \mm{Sym}_k^i\big(\tang[1]\big)\otimes \mm{Sym}_k^p\big(\keranc[1]\big)\rt \keranc[2],
$$ 
such that
\begin{equation}\label{eq:bracket binomial}
\mc{L}_s=\sum_{i}\sum_{\gamma\in \mm{Sh}^{-1}_{i, s-i}} \big(\ell^{(i)}_{s-i}\big)^{\gamma}.
\end{equation}
(Here we view maps from a symmetric power as symmetric functions from a tensor power; the sums over unshuffles then guarantee that the above indeed gives a symmetric function.) 
Note that $\mc{L}_s$ is $A$-multilinear for $s\neq 1, 2$ and that $\mc{L}_1$ and $\mc{L}_2$ have the derivation properties
$$
\mc{L}_1(a\cdot x) = d_A(a)\cdot x + a\cdot \mc{L}_1(x) \qquad\qquad \mc{L}_2(x, a\cdot y)= a\cdot \mc{L}_2(x, y)+ \mc{L}_{\rho(x)}(a)\cdot y.
$$
This is equivalent to all maps $\ell_p^{(i)}$ being graded $A$-linear, except for the maps
$$
\ell^{(0)}_1\colon \keranc\rt \keranc[1]\qquad \text{and}\qquad \ell^{(1)}_{1}=[-, -]\colon \tang\otimes \keranc\rt \keranc
$$
which satisfy equation \eqref{eq:extendability to derivation}. In other words, Lemma \ref{lem:further decompositions} implies that the maps $\mc{L}_s$ are $A$-multilinear (resp.\ derivations for $n=1, 2$) if and only if each 
$$
\ell_p\coloneqq \sum_i \ell^{(i)}_p
$$
defines an element in $\mc{H}(p)$. 
\end{construction}

\begin{proof}[Proof (of Proposition \ref{prop:curved algebroid vs curved algebra})]
In light of Lemma \ref{lem:further decompositions} and Construction \ref{cons:curved algebroid vs curved algebra}, it suffices to verify that the maps $\ell_p=\sum_i \ell_p^{(i)}$ define a curved $L_\infty$-structure on $B\otimes_A \keranc$ if and only if the maps $\mc{L}_s=\sum_i\sum_\gamma\big(\ell_{s-i}^{(i)}\big)^{\gamma}$ define a curved $L_\infty$-structure on $\tang\oplus \keranc$.

To this end, let us consider all $\ell_p$ together as a single element $\sum_p \ell_p\in \prod_p \mc{H}(p)$, with $\mc{H}(p)$ as in \eqref{eq:derivations}. The condition of being a curved $L_\infty$-algebra translates into the equation (where we can suppress all additional signs by working with the shift $\keranc[1]$, see Remark \ref{rem:signs})
$$
\frac{1}{2}\sum_{p, q} [\ell_p, \ell_q]=\sum_{p, q}\sum_{\sigma\in\mm{Sh}^{-1}_{p-1, q}} \big(\ell_p\circ_1 \ell_q\big)^\sigma=0.
$$
Let us unravel this equation in terms of the maps $\ell_p^{(i)}\colon \mm{Sym}_A^i(\tang[1])\otimes \mm{Sym}_A^p(\keranc[1])\rt \keranc[1]$, as in Construction \ref{cons:curved algebroid vs curved algebra}. Using the formulas for partial composition and commutator from Observation \ref{obs:decomposing maps}, this yields
$$
\sum_{p, q, i, j}\sum_{\substack{\sigma\in \mm{Sh}^{-1}_{p-1, q}\\ \tau\in\mm{Sh}^{-1}_{i, j}}} \big(\ell_p^{(i)}\circ_{\keranc} \ell_q^{(j)}\big)^{\sigma\times \tau}
+\sum_{p, i, j}\sum_{\tau\in \mm{Sh}^{-1}_{i-1, j}}\big(\ell_p^{(i)}\circ_{\tang} \ell_{\tang}^{(j)}\big)^{\tau}=0.
$$
The above equation is equivalent to a certain system of equations $E(r, k)=0$, collecting all terms consisting of maps with $r$ inputs from $\keranc$ and $k$ inputs from $\tang$. In turn, this is equivalent to a system of equations
\begin{align*}
0=\sum_{r, k} \sum_{\gamma\in \mm{Sh}^{-1}_{r-k, r}} E(r-k, k)^\gamma &=\sum_{p, q, i, j}\sum_{\gamma\in \mm{Sh}^{-1}_{p+q-1, i+j}}\sum_{\substack{\sigma\in \mm{Sh}^{-1}_{p-1, q}\\ \tau\in\mm{Sh}^{-1}_{i, j}}}\Big(\big(\ell_p^{(i)}\circ_{\keranc} \ell_q^{(j)}\big)^{\sigma\times \tau}\Big)^{\gamma} \\
&+\sum_{p, i, j}\sum_{\gamma\in \mm{Sh}^{-1}_{p, i-1+j}} \sum_{\tau\in \mm{Sh}^{-1}_{i-1, j}}\Big(\big(\ell_p^{(i)}\circ_{\tang} \ell_{\tang}^{(j)}\big)^{\tau}\Big)^{\gamma}.
\end{align*}
For fixed $r$, this consists of maps $\mm{Sym}^r_k(\tang[1]\oplus \keranc[1])\rt \keranc[1]$. By our definition of the maps $\ell_{r-k}^{(k)}$ in terms of the curved $\Lie_\infty$-structure on $\tang\oplus\keranc$ \eqref{eq:bracket binomial}, the above equation is then equivalent to the equation
$$
\sum_{\substack{\\ t, s\geq 1}}\sum_{\sigma\in \mathrm{Sh}_{s-1,t}^{-1}}(\mc{L}_s\circ_{\keranc} \mc{L}_t)^\sigma + \sum_{\sigma\in \mm{Sh}^{-1}_{s-1, t}}  \big(\mc{L}_s\circ_{\tang} \ell_{\tang}^{(t)}\big)^{\sigma}=0.
$$
Here the first term involves partial composition along $\keranc$, while the second term involves partial composition along $\tang$. Unraveling the definitions, the above equation means precisely that the pair $\big(\ell_{\tang}^{(s)}\circ \pi, \mc{L}_s\big)\colon \mm{Sym}^s_k(\tang[1]\oplus \keranc[1])\rt \tang[2]\oplus \keranc[2]$ defines a curved $L_\infty$-structure on $\tang\oplus \keranc$.
\end{proof}

\subsubsection{Behaviour on $\infty$-morphisms}
Next, let us discuss how Proposition \ref{prop:curved algebroid vs curved algebra} interacts with $\infty$-morphisms. To this end, let $\tang$ and $B=C^*(\tang)$ be as in Theorem \ref{thm:comparison} and let
$$
\Loid=\tang\oplus \keranc, \qquad\qquad\qquad \Loide=\tang\oplus \kerance.
$$
Assume that $\Loid$ and $\Loide$ come equipped with curved $L_\infty$-algebroid structures such that the projection to $\tang$ is a (strict) map of curved $L_\infty$-algebroids and consider an $\infty$-morphism of curved $L_\infty$-algebroids which fits into a commuting diagram
\begin{equation}\label{diag:oo-morphisms over t}\begin{tikzcd}
\Loid\arrow[rr, rightsquigarrow, "\Phi"]\arrow[rd, "\pi_\mf{g}"{swap}] & &  \Loide\arrow[ld, "\pi_\mf{h}"]\\
& \tang.
\end{tikzcd}\end{equation}
In particular, this $\infty$-morphism is uniquely determined by $A$-linear maps of the form
$$\begin{tikzcd}[row sep=0.5pc]
\Phi_s=(0, \Phi'_s)\colon \Sym^s_A(\tang[1]\oplus \keranc[1])\arrow[r] & \tang[1]\oplus \kerance[1] & p\neq 1\\
\Phi_1=\big(\pi_{\tang}, \Phi'_1\big)\colon \tang[1]\oplus \keranc[1]\arrow[r] & \tang[1]\oplus \kerance[1]
\end{tikzcd}$$
where $\pi_{\tang}$ projects onto $\tang$. As in Construction \ref{cons:curved algebroid vs curved algebra}, we decompose $\Phi'_s=\sum_i \sum_{\gamma\in \mm{Sh}^{-1}_{i, s-i}}\big(\phi_{s-i}^{(i)}\big)^{\gamma}$ binomially into maps of complete graded $A$-modules
$$
\phi_p^{(i)}\colon \Sym_A^{i}(\tang[1])\otimes_A \Sym_A^{p}(\keranc[1])\rt \kerance[1].
$$
Since $\tang$ is a dualizable $A$-module, arguing as in Lemma \ref{lem:further decompositions} shows that such families of $A$-linear maps correspond bijectively to families of $B$-linear maps
$$
\phi_p\colon \mm{Sym}^p_B\big(B\otimes_A\keranc[1]\big)\rt B\otimes_A \kerance[1].
$$
Here $\phi_p$ denotes (using Remark \ref{rem:curv(g)asforms}) the $B$-linear extension of the map
$$
\sum \phi_p^{(i)}\colon \mm{Sym}^p_A(\keranc[1])\rt B\otimes_A \kerance[1]\cong \prod_i \Hom_A\big(\Sym_A^i(\tang[1]), \kerance[1]\big).
$$
\begin{proposition}\label{prop:curv is compatible with oo-maps}
Suppose that $\Loid$ and $\Loide$ are curved $L_\infty$-algebroids over $\tang$ as above. Then the maps $\phi_p$ constructed above define an $\infty$-morphism $\phi \colon \curv(\Loid) \rightsquigarrow \curv(\Loide)$ of curved $L_\infty$-algebras over $B$ if and only if the maps $\Phi_s$ define an $\infty$-morphism of curved $L_\infty$-algebroids as in \eqref{diag:oo-morphisms over t}.
\end{proposition}
Throughout, let us write $\ell_p$ and $\kappa_q$ for the curved $L_\infty$-structure maps on $B\otimes_A \keranc$ and $B\otimes_A \kerance$, respectively. Likewise, the structure maps of $\Loid=\tang\oplus \keranc$ and $\Loide=\tang\oplus\kerance$ have the form $(\ell_{\tang}^{(s)}, \mc{L}_s)$, respectively $(\ell_{\tang}^{(s)}, \mc{K}_s)$, where $\ell_{\tang}^{(s)}$ is the $L_\infty$-structure on $\tang$.
\begin{observation}\label{obs:B-linear extension}
Let $\phi_p\colon \Sym^p_B(B\otimes_A \keranc[1])\rt B\otimes_A \kerance[1]$. For $q\neq 1$ and $k\neq 1$, the composite maps
$$
\sum_{\sigma\in \mathrm{Sh}_{p-1,q}^{-1}}\big(\phi_p\circ_1 \ell_q\big)^{\sigma} \qquad\text{and}\qquad \sum_{\sigma\in \mm{Sh}^{-1}_{(p_1, \dots, p_k)}}\Big(\kappa_k\big(\phi_{p_1}, \dots, \phi_{p_k}\big)\Big)^{\sigma}
$$
are $B$-multilinear. In particular, they are determined by the maps 
$$\begin{tikzcd}[row sep=0.4pc]
\phi_p^{(i)}\colon \Sym^i_A(\tang[1])\otimes_A\Sym_A^p(\keranc[1])\arrow[r] & \kerance[1]\\
\ell_q^{(j)}\colon \Sym^j(\tang[1])\otimes\Sym_A^q\big(\keranc[1])\arrow[r] & \keranc[2]\\
\kappa_k^{(j)}\colon \Sym^j(\tang[1])\otimes\Sym_A^k\big(\kerance[1])\arrow[r] & \kerance[2]
\end{tikzcd}$$
via
\begin{align*}
\sum_{\sigma\in \mathrm{Sh}_{p-1,q}^{-1}}\big(\phi_p\circ_1 \ell_q\big)^{\sigma}&=\sum_{i, j} \sum_{\substack{\sigma\in \mathrm{Sh}_{p-1,q}^{-1}\\ \tau\in \mm{Sh}^{-1}_{i,j}}}\big(\phi^{(i)}_p\circ_{\keranc} \ell_q^{(j)}\big)^{\sigma\times \tau}\\
\sum_{\sigma\in \mm{Sh}^{-1}_{(p_1, \dots, p_k)}}\Big(\kappa_k\big(\phi_{p_1}, \dots, \phi_{p_k}\big)\Big)^{\sigma} &= \sum_{i, j_1, \dots, j_k} \sum_{\substack{\sigma\in\mm{Sh}^{-1}_{(p_1, \dots, p_k)}\\ \tau\in\mm{Sh}^{-1}_{(i, j_1, \dots, j_k)}}} \Big(\kappa^{(i)}_k\circ_{\kerance}\big(\phi^{(j_1)}_{p_1}, \dots, \phi^{(j_k)}_{p_k}\big)\Big)^{\sigma\times \tau}
\end{align*}
where $\circ_{\keranc}$ denotes partial composition via the first $\keranc$-variable and $\circ_{\kerance}$ denotes total composition along the $\kerance$-variable.

On the other hand, the maps $\sum_{\sigma}\big(\phi_p\circ_1 \ell_1\big)^{\sigma}$ and $\kappa_1\circ \phi_p$ both define maps 
$$
\Sym^p(B\otimes_A \keranc[1])\rt B\otimes_A \kerance[2]
$$
that are derivations over the Chevalley--Eilenberg differential on $B$ in each variable. Such derivations are again determined uniquely by their restriction $\Sym^p_A(\keranc[1])\rt B\otimes_A \kerance[2]$. Using Remark \ref{rem:ell1extensionasderivation}, these can be expressed in terms of the maps $\phi^{(i)}_p, \ell_1^{(j)}$ and $\kappa_1^{(j)}$ as
\begin{align*}
\sum_{\sigma\in \mathrm{Sh}_{p-1,1}^{-1}}\big(\phi_p\circ_1 \ell_1\big)^{\sigma}&=\sum_{i, j} \sum_{\substack{\sigma\in \mathrm{Sh}_{p-1,1}^{-1} \\ \tau\in \mm{Sh}^{-1}_{i,j}}}\big(\phi^{(i)}_p\circ_{\keranc} \ell_1^{(j)}\big)^{\sigma\times \tau}\\
\kappa_1\circ \phi_p &= \sum_{i, j} \sum_{\tau\in \mm{Sh}^{-1}_{j, i}}\big(\kappa^{(j)}_1\circ_{\kerance} \phi_{p}^{(i)}\big)^{\tau} - \sum_{i, j}\sum_{\tau\in \mm{Sh}^{-1}_{i-1, j}} \big(\phi_p^{(i)}\circ_{\tang} \ell_{\tang}^{(j)}\big)^{\tau}.
\end{align*}
Here $\ell_{\tang}^{(j)}$ is the $L_\infty$-structure on $\tang$ and $\circ_{\tang}$ takes the partial composition in the $\tang$-variable.
\end{observation}
\begin{proof}
By Section \ref{sec:oo-morphism of clie} and Definition \ref{def:oo-morphism B-linear}, the maps $\phi_p$ define an $\infty$-morphism of curved $L_\infty$-algebras over $B$ if and only if for each $n\geq 0$
\begin{equation}\label{eq:cLoo map}
0=\sum_{p+q=n+1}\sum_{\sigma\in \mathrm{Sh}_{p-1,q}^{-1}}\big(\phi_p\circ_1\ell_q\big)^\sigma-\sum_{\substack{k\geq 0\\ p_1+\dots +p_k=n}} \sum_{\sigma\in \mathrm{Sh}_{(p_1,\dots,p_k)}^{-1}}\frac{1}{k!} \kappa_k(\phi_{p_1}, \dots, \phi_{p_k})^\sigma
\end{equation}
The signs are determined by Remark \ref{rem:signs} and are exactly as written above if we work at the level of shifted objects.
As in the proof of Proposition \ref{prop:curved algebroid vs curved algebra}, the proof boils down to take the sum of these equations over $n$ and then appropriately rewriting the involved sums ``diagonally''.
More precisely, using Observation \ref{obs:B-linear extension}, the above equation is equivalent to a system of equations 
\begin{align*}
0&=\sum_{i, j, p, q}\sum_{\substack{\sigma\in \mathrm{Sh}_{p-1,q}^{-1}\\ \tau\in \mm{Sh}^{-1}_{i, j}}}\big(\phi_p^{(i)}\circ_\keranc\ell^{(j)}_q\big)^{\sigma \times \tau}\\
&-\sum_{\substack{i, k, j_1, \dots, j_k\\ p_1, \dots, p_k}} \frac{1}{k!}\sum_{\substack{\sigma\in\mm{Sh}^{-1}_{(p_1, \dots, p_k)}\\ \tau\in\mm{Sh}^{-1}_{(i, j_1, \dots, j_k)}}} \Big(\kappa^{(i)}_k\circ_{\kerance}\big(\phi^{(j_1)}_{p_1}, \dots, \phi^{(j_k)}_{p_k}\big)\Big)^{\sigma\times \tau}\\
&+ \sum_{i, j, p}\sum_{\tau\in \mm{Sh}^{-1}_{i-1, j}} \big(\phi_p^{(i)}\circ_{\tang} \ell_{\tang}^{(j)}\big)^{\tau}.
\end{align*}
As in the proof of Proposition \ref{prop:curved algebroid vs curved algebra}, this is a system of equations $E(r, m)=0$: for fixed $r$ and $m$, this is an equation between maps $\Sym^r(\keranc[1])\otimes \Sym^m(\tang[1])\rt \kerance[2]$. This is equivalent to the system of equations
\begin{align*}
0=\sum_{r, m} \sum_{\gamma\in \mm{Sh}^{-1}_{r-m, m}} E(r-m, m)^\gamma &=\sum_{i, j, p, q}\sum_{\gamma, \sigma, \tau}\Big(\big(\phi_p^{(i)}\circ_\keranc\ell^{(j)}_q\big)^{\sigma \times \tau}\Big)^{\gamma}\\
&-\sum_{\substack{i, k, j_1, \dots, j_k\\ p_1, \dots, p_k}} \frac{1}{k!}\sum_{\gamma, \sigma, \tau} \Big(\Big(\kappa^{(i)}_k\circ_{\kerance}\big(\phi^{(j_1)}_{p_1}, \dots, \phi^{(j_k)}_{p_k}\big)\Big)^{\sigma\times \tau}\Big)^{\gamma}\\
&+ \sum_{i, j, p}\sum_{\gamma, \tau} \Big(\big(\phi_p^{(i)}\circ_{\tang} \ell_{\tang}^{(j)}\big)^{\tau}\Big)^{\gamma}.
\end{align*}
Here we sum over unshuffles $\sigma$ of the variables within $\keranc$, unshuffles $\tau$ from the variables within $\tang$ and finally the unshuffles $\gamma$ of the sets of variables from $\tang$ and $\keranc$. For fixed $r$, this is an equation between maps $\mm{Sym}^r\big(\tang[1]\oplus \keranc[1]\big)\rt \kerance[1]$.

Now, the terms in the first and third line with $p+i=s$ and $q+j=t$ sum up to the maps
\begin{equation}\label{eq:term 1 and 3}
\sum_{\sigma\in \mm{Sh}^{-1}_{s-1, t}}\big(\Phi'_s\circ_1 \mc{L}_t\big)^{\sigma} \qquad\qquad\text{and} \qquad\qquad \sum_{\sigma\in \mm{Sh}^{-1}_{s-1, t}}\big(\Phi'_s\circ_1 \ell_{\tang}^{(t)}\big)^{\sigma}
\end{equation}
of the form $\Sym^{s+t-1}(\tang[1]\oplus\keranc[1])\rt \kerance[1]$. Likewise, the terms in the second line with fixed $k+i=s$ and $p_\alpha+j_\alpha=t_\alpha$ sum up to
\begin{equation}\label{eq:sum over a}\sum_a \sum_{\sigma\in\mm{Sh}^{-1}_{i, t_1, \dots, t_{s-i}}} \frac{1}{(s-i)!}\Big(\mc{K}_s\circ \big(\underbrace{\pi_{\tang}, \dots, \pi_{\tang}}_{i\times}, \Phi'_{t_1}, \dots, \Phi'_{t_{s-i}}\big)\Big)^{\sigma}.
\end{equation}
Here $\pi_{\tang}\colon \tang\oplus \keranc\rt \tang$ denotes the projection onto $\tang$. The sum of \eqref{eq:sum over a} over all $t_1, \dots, t_{s-a}$, can then be identified in terms of the structure maps $\Phi_p$ of \eqref{diag:oo-morphisms over t} with
\begin{equation}\label{eq:term 2}
\frac{1}{s!}\sum_{t_1, \dots, t_s}\sum_{\sigma\in \mm{Sh}^{-1}_{t_1, \dots, t_s}} \Big(\mc{K}_s\circ \big(\Phi_{t_1}, \dots, \Phi_{t_s}\big)\Big)^{\sigma}.
\end{equation}
Indeed, for each $i\geq 0$ there are ${s \choose i}$ terms in \eqref{eq:term 2} where $i$ different copies $t_\alpha$ are equal to $1$ (so that $\Phi_{t_\alpha}=\Phi_1=(\pi_{\tang}, \Phi'_1)$); in turn, each of these ${s \choose i}$ terms can be identified with $i!$ copies of the expression \eqref{eq:sum over a}, since the size of $\mm{Sh}^{-1}(1, \dots, 1, t_{1}, \dots, t_{s-i})$ is $i!$ times the size of $\mm{Sh}^{-1}(i, t_1, \dots, t_{s-i})$.

The sums of \eqref{eq:term 1 and 3} and \eqref{eq:term 2}  now precisely give the equation for $\Phi\colon \mf{g}\rightsquigarrow \mf{h}=\tang\oplus \kerance$ being an $\infty$-morphism of curved $L_\infty$-algebroids.
\end{proof}

\subsection{The functor $\curv$ and proof of the main theorem}\label{sec:proof of main theorem}
We will now use the explicit computations from Section \ref{sec:formulas for curv} to construct the desired equivalence of $\infty$-categories of Theorem \ref{thm:comparison}
$$\begin{tikzcd}
\curv\colon \oocat{cLie}(A/k){\scriptstyle/\tang}\arrow[r] & \oocat{cLie}_B.
\end{tikzcd}$$
Recall that both $\infty$-categories $\oocat{cLie}(A/k)$ and $\oocat{cLie}_B$ are modeled by concrete simplicially enriched categories (Definition \ref{def:simplicial cat of lie algebroids} and Definition \ref{def:oocat of algebras over B}). Likewise, the domain of the putative functor `$\curv$' can be modeled by an explicit simplicially enriched $\infty$-category, as follows:
\begin{definition}\label{def:slice of curved lie}
Let $\scat{C}$ denote the simplicially enriched category whose:
\begin{enumerate}[start=0]
\item objects are given by curved $L_\infty$-algebroids $\Loid$ over $A$, equipped with a (strict) map of $L_\infty$-algebroids $\pi\colon \Loid\rt \tang$ and a section of complete graded $A$-modules $\sigma\colon \tang\rt \Loid$. This section induces an equivalence $\Loid\cong \tang\oplus \keranc$ and we require $\keranc$ to be projective as a complete graded $A$-module.

\item simplicial sets of morphisms are given by the simplicial sets of $\infty$-morphisms
$$\begin{tikzcd}
\Loid\arrow[rr, rightsquigarrow, "\Phi"]\arrow[rd] & & \Loide\boxtimes_{\tang} \Omega[\Delta^\bullet]\coloneqq \Loide\otimes \Omega[\Delta^\bullet]\times_{\tang\otimes \Omega[\Delta^\bullet]} \tang\arrow[ld, start anchor={[xshift=-2.5pc]}]\\
& \tang.
\end{tikzcd}$$
\end{enumerate}
\end{definition}
\begin{lemma}\label{lem:C is slice}
There is a natural equivalence of $\infty$-categories $\scat{C}\rto{\sim} \oocat{cLie}(A/k){\scriptstyle/\tang}$.
\end{lemma}
\begin{proof}
Let us denote by $\oocat{cLie}(A/k)\hspace{1pt}\mathrlap{\hspace{-1pt}\smallsetminus}{\scriptstyle/\hspace{2pt}\tang}$ the simplicially enriched slice category of $\oocat{cLie}(A/k)$ over $\tang$: objects are maps $\mf{g}\rt \tang$ and simplicial maps of morphisms consist exactly of maps as in Definition \ref{def:slice of curved lie}. The simplicially enriched category $\oocat{cLie}(A/k)\hspace{1pt}\mathrlap{\vspace{-1pt}\hspace{-1pt}\smallsetminus}{\scriptstyle/\hspace{2pt}\tang}$ is \emph{not} a model for the slice $\infty$-category $\oocat{cLie}(A/k){\scriptstyle/\tang}$, but there is a comparison map of $\infty$-categories $\oocat{cLie}(A/k)\hspace{1pt}\mathrlap{\vspace{-1pt}\hspace{-1pt}\smallsetminus}{\scriptstyle/\hspace{2pt}\tang}\rt \oocat{cLie}(A/k){{\scriptstyle/\tang}}$ \cite[\href{https://kerodon.net/tag/01ZN}{Tag 01ZN}]{kerodon}. 

Composing this map with the natural simplicially enriched functor $\scat{C}\rt \oocat{cLie}(A/k)\hspace{1pt}\mathrlap{\vspace{-1pt}\hspace{-1pt}\smallsetminus}{\scriptstyle/\hspace{2pt}\tang}$ produces the desired map $\scat{C}\rt \oocat{cLie}(A/k){\scriptstyle/\tang}$. 
It is essentially surjective because every curved $L_\infty$-algebroid over $\tang$ is homotopy equivalent to one for which the projection $\Loid\rt \tang$ is surjective and $\Loid$ is a projective complete graded $A$-module (by Theorem \ref{thm:model structure lie algebroids}). Since $\tang$ was assumed to be projective as a graded complete $A$-module, such curved $L_\infty$-algebroids $\Loid$ admit a splitting $\Loid\cong \tang\oplus \keranc$ where $\keranc$ is projective as well. 

Furthermore, for two objects $\pi_1\colon \Loid\rt \tang$ and $\pi_2\colon \Loide\rt \tang$ in $\scat{C}$, the simplicial set of maps between them fits into a pullback square of simplicial sets
$$\begin{tikzcd}
\Map_{\scat{C}}(\Loid, \Loide)\arrow[r]\arrow[d] & \Map_{\oocat{cLie}(A/k)}(\Loid, \Loide)\arrow[d, "(\pi_2)_*"]\\
\{\pi_1\}\arrow[r] & \Map_{\oocat{cLie}(A/k)}(\Loid, \tang).
\end{tikzcd}$$
Unraveling the definitions, one sees that the right vertical map is a Kan fibration, so that the above square is homotopy cartesian. This implies that the functor $\scat{C}\rt \oocat{cLie}(A/k){\scriptstyle/\tang}$ is fully faithful (cf.\ \cite[\href{https://kerodon.net/tag/01ZT}{Tag 01ZT}]{kerodon}).
\end{proof}
\begin{construction}\label{con:curv}
For every map of curved $L_\infty$-algebroids $\Loid\rt \tang$, together with a splitting $\Loid\cong \tang\oplus \keranc$, let us now define
$$
\curv(\Loid)\coloneqq B\otimes_A \keranc
$$
equipped with the curved $L_\infty$-structure over $B$ from Proposition \ref{prop:curved algebroid vs curved algebra}. Observe that this definition is compatible with tensoring with forms on the simplex, i.e.\
\begin{equation}\label{eq:curv preserves cotensor}
\curv\big(\Loid\boxtimes_{\tang}\Omega[\Delta^n]\big)\cong \curv(\Loid)\otimes \Omega[\Delta^n].
\end{equation}
For any $\infty$-morphism $\Phi$ as in Definition \ref{def:slice of curved lie}, we then let
$$\begin{tikzcd}
\curv(\Phi)\colon \curv(\Loid)\arrow[r, rightsquigarrow] & \curv(\Loide)\otimes\Omega[\Delta^\bullet]
\end{tikzcd}$$
be the associated $\infty$-morphism of curved $L_\infty$-algebras over $B$ from Proposition \ref{prop:curv is compatible with oo-maps}.
\end{construction}
\begin{lemma}\label{lem:curv functor}
Construction \ref{con:curv} defines a simplicially enriched functor
$$\begin{tikzcd}
\curv\colon \scat{C}\arrow[r] & \oocat{cLie}_B
\end{tikzcd}$$
to the simplicial category of curved Lie algebras over $B$ from Definition \ref{def:oocat of algebras over B}.
\end{lemma}
\begin{proof}
	If $\Psi\colon \Loi \rightsquigarrow \Loid$ and $\Phi\colon \Loid \rightsquigarrow \Loide$ are $\infty$-morphisms of curved $L_\infty$-algebroids (over $\tang$), their composite $\Phi\circ \Psi$ has components (cf.\ \cite[Proposition 10.2.7]{LodayVallette2012})
	
	\[
	(\Phi\circ \Psi)_s = \sum_{k, s_1+\dots+s_k=s} \sum_{\sigma\in \mathrm{Sh}_{(s_1,\dots,s_k)}^{-1}} \frac{1}{k!}\Phi_k(\Psi_{s_1},\dots,\Psi_{s_k})^\sigma.  
	\]
	On the other hand, $\infty$-morphisms of curved $L_\infty$-algebras over $B$ have the same composition formula. Going through the arguments of Proposition \ref{prop:curv is compatible with oo-maps}, one then sees that $\curv(\Phi\circ \Psi) = \curv(\Phi) \circ \curv(\Psi)$. Alternatively, this can be seen in terms of coalgebras as follows: an $\infty$-morphism $\Loid=\tang\oplus\keranc\rightsquigarrow \tang\oplus\kerance=\Loide$ over $\tang$ is given by a certain map of graded complete $A$-linear cocommutative coalgebras 
$$\begin{tikzcd}
\Sym^c_A(\mf{g}[1])\cong \Sym^c_A(\tang[1])\otimes \Sym^c_A(\keranc[1])\arrow[r, "\Phi"] & \Sym^c_A(\tang[1])\otimes \Sym^c_A(\kerance[1])\cong \Sym^c_A(\mf{h}[1])
\end{tikzcd}$$
over $\Sym^c_A(\tang[1])$. Using that $\Sym^c_A(\tang[1])$ is a dualizable complete $A$-module, such a map of coalgebras over $\Sym^c(\tang[1])$ is equivalent to a map of $B$-linear coalgebras
$$\begin{tikzcd}
B\otimes_A \Sym^c_A(\keranc[1])= \Sym_A(\tang^\vee[-1])\otimes_A \Sym^c_A(\keranc[1])\arrow[r] & B\otimes_A \Sym^c_A(\keranc[1])
\end{tikzcd}$$
This is precisely the $B$-linear coalgebra map corresponding to the $\infty$-morphism of Proposition \ref{prop:curv is compatible with oo-maps}. The above construction manifestly preserves composition of coalgebra maps, which implies that $\curv$ preserves composition of $\infty$-morphisms.

For functoriality on $n$-simplices in the mapping spaces, note that the composition of $\Psi\colon \Loi \rightsquigarrow \Loid\boxtimes_{\tang}\Omega[\Delta^n]$ and $\Phi\colon \Loid \rightsquigarrow \Loide\boxtimes_{\tang}\Omega[\Delta^n]$ is given by
$$\begin{tikzcd}[column sep=1.5pc]
\Loi \arrow[r, rightsquigarrow, "\Psi"] & \Loid\boxtimes_{\tang}\Omega[\Delta^n]\arrow[r, rightsquigarrow, "\Phi"] & \Big(\Loide\boxtimes_{\tang}\Omega[\Delta^n]\Big)\boxtimes_{\tang}\Omega[\Delta^n] \cong \Loide\boxtimes_{\tang} \Omega[\Delta^n\times\Delta^n]\arrow[r, "\Delta^*"] & \Loide\boxtimes_{\tang} \Omega[\Delta^n]
\end{tikzcd}$$
where the last map is induced by the map of cdgas $\Delta^*\colon \Omega[\Delta^n\times\Delta^n]\rt \Omega[\Delta^n]$ restricting along the diagonal. By functoriality at the level of $\infty$-morphisms and Equation \eqref{eq:curv preserves cotensor}, one sees that $\curv$ also respects composition at the level of $n$-simplices.
\end{proof}
Finally, we turn to the proofs of Theorem \ref{thm:comparison} and Theorem \ref{thm:comparison tangent case}.
\begin{proof}[Proof of Theorem \ref{thm:comparison}]
Lemma \ref{lem:C is slice} and Lemma \ref{lem:curv functor} furnish a zig-zag of functors of $\infty$-categories
$$\begin{tikzcd}
\oocat{cLie}(A/k){\scriptstyle/\tang} & \scat{C}\arrow[l, "\sim"{above}] \arrow[r, "\curv"] & \oocat{cLie}_B.
\end{tikzcd}$$
Taking the inverse of the left functor and composing it with the right functor gives the desired functor $\curv\colon \oocat{cLie}(A/k){\scriptstyle/\tang}\rt \oocat{cLie}_B$. To see that it is an equivalence, it suffices to verify that the right functor is an equivalence. Indeed, Proposition \ref{prop:curved algebroid vs curved algebra} implies that it is essentially surjective and Proposition \ref{prop:curv is compatible with oo-maps} implies that it is (strictly) fully faithful as a map between simplicially enriched categories.
\end{proof}
\begin{proof}[Proof of Theorem \ref{thm:comparison tangent case}]
If $A$ is a cofibrant cdga locally of finite presentation (or smooth), then $T_A$ is a finitely generated projective graded $A$-module. Consequently, the filtered Lie algebroid $T_A\wgt{1}$ satisfies the conditions of Theorem \ref{thm:comparison} (see Example \ref{ex:comparison for tangent}). Let us now consider the composite functor
\begin{equation}\label{diag:lie algebroids into curved lie}\begin{tikzcd}[column sep=3pc]
\oocat{Lie}(A/k)\arrow[r, hook, "(-)^\mm{anc}"] & \oocat{cLie}(A/k){\scriptstyle/T_A\wgt{1}}\arrow[r, "\curv"{above}, "\sim"{below}] & \oocat{cLie}_{\dR(A)}
\end{tikzcd}\end{equation}
where $\dR(A)=C^*(T_A\wgt{1})$ is the de Rham algebra of $A$ with its natural filtration. The second functor is the equivalence of Theorem \ref{thm:comparison} and the first functor is the fully faithful inclusion of Proposition \ref{prop:adding filtrations fully faithful}. Unraveling the constructions, one sees that the composite sends an $L_\infty$-algebroid of the form $\Loid\simeq T_A\oplus \keranc$ to a curved $L_\infty$-algebra of the form $\dR(A)\otimes_A \keranc$, where $\keranc$ is in filtration weight $0$. Such curved $L_\infty$-algebras have the property that the natural map
$$\begin{tikzcd}
\Gr^0\big(\dR(A)\otimes_A \keranc\big)\otimes_{\Gr^0(\dR(A))}\Gr(\dR(A))\arrow[r] & \Gr^0\big(\dR(A)\otimes_A \keranc\big)
\end{tikzcd}$$
is an isomorphism. Conversely, suppose that $\mf{g}$ is a curved $L_\infty$-algebra over $B$ such that 
\begin{equation}\label{eq:graded generators in degree 0}
\Gr^0(\mf{g})\otimes_{\Gr^0(\dR(A))}\Gr(\dR(A))\rt \Gr(\mf{g})
\end{equation} is an equivalence. We may assume that $\mf{g}$ arises from a mixed-curved $L_\infty$-algebra over $B$ under the functor $\blend \colon \cat{cLie}^{\mm{mix}}_B\rt \cat{cLie}_B$ from Theorem \ref{thm:homotopy theories of curved lie over B}. There then exists a cofibrant dg-$A$-module $\keranc$ (in filtration degree $0$) together with a map 
$$
\phi\colon \dR(A)\otimes_A \keranc\rt \mf{g}
$$
which induces a quasi-isomorphism on the associated graded. Using the Homotopy Transfer Theorem \ref{thm:htt-B}, one can endow $\dR(A)\otimes_A \keranc$ with the structure of a mixed-curved $L_\infty$-algebra over $\dR(A)$ such that $\phi$ becomes an $\infty$-equivalence of mixed-curved $L_\infty$-algebras. Proposition \ref{prop:curved algebroid vs curved algebra} now implies that this equivalent curved $L_\infty$-algebra $\dR(A)\otimes_A \keranc$ is isomorphic to $\curv(T_A\wgt{1}\oplus \keranc)$, for a certain curved $L_\infty$-algebroid structure on $T_A\wgt{1}\oplus \keranc$. Such (curved) $L_\infty$-algebroids are precisely images of ordinary $L_\infty$-algebroids under the functor $(-)^\mm{anc}$. We conclude that the essential image of \eqref{diag:lie algebroids into curved lie} indeed consists of those curved $L_\infty$-algebras over $\dR(A)$ for which the map \eqref{eq:graded generators in degree 0} is an equivalence.
\end{proof}

\section{Examples and applications}\label{sec:Examples}
In this section we will give some concrete computations and applications of our main results.

In Section \ref{sec:explicit} we will establish an explicit formula for the functor $\curv$ for Lie algebroids with surjective anchor and discuss the less explicit case of a non-surjective anchor.

In Section \ref{sec:rogers} we will see that the formalism developed for curved Lie algebras produces non-trivial results even when restricting to uncurved Lie algebras, by exhibiting the Maurer--Cartan space as a mapping space. 

Finally, in Section \ref{sec:diffgeo} we show a $\mathcal{C}^\infty$ variant of our main result and also interpret that result in the case where $L_\infty$-algebroids arise from vector bundles.

\subsection{A more explicit functor $\curv$}\label{sec:explicit}

While the construction of the functor $\curv$ resorts to $\infty$-categorical methods, its description on objects can be made quite explicit when applied to a strict Lie algebroid $\rho\colon \Loid\to T_A$. 
The construction depends (even if it is homotopically independent) of a choice of a section of $\rho$ as a map of graded $A$-modules, so we first treat the case in which the anchor admits an $A$-linear section and then we show how to explicitly obtain such an algebroid when the anchor is not surjective.

\subsubsection*{If the anchor splits}

Suppose $\Loid$ is a Lie algebroid over a cofibrant cdga $A$ with surjective anchor $\rho\colon \Loid \twoheadrightarrow T_A$ and let us fix an $A$-linear section $s\colon T_A \to \Loid$ of $\rho$, which is not necessarily compatible with the differentials nor the Lie brackets.

Denoting $\mf{n} = \ker \rho$, using the section we can split $\Loid$ as a direct sum of (non-differential) $A$-modules $\Loid = T_A \oplus \mf{n}$.

The bracket and the differential with respect to this decomposition now take the form 

\begin{align*}
	d_{\Loid}(v, \xi)&=\big(dv, d\xi + \ell^{(1)}(v)\big) &\in T_A \oplus \mf n\\
	\big[(v, \xi), (w, \eta)\big]&= \big([v, w], [\xi, \eta]+ \nabla_v(\eta)-\nabla_w(\xi) + \ell_0^{(2)}(v, w)\big) &\in T_A \oplus \mf n,
\end{align*}

for some uniquely determined $k$-(bi)linear maps $\ell_0^{(1)}(-)$, $\ell_0^{(2)}(-,-)$ and $\nabla_-(-)$.

It follows from the Leibniz rule that $\nabla\colon T_A\otimes \mf{n}\rt \mf{n}$ is a connection and from  the Jacobi identity on $\Loid$ we observe that for a fixed $v\in T_A$ we have that  $\nabla_v\colon \mf{n}\rt \mf{n}$ is a derivation with respect to the Lie bracket.

Let us consider the  $\dR(A)$-module $\dR(A)\otimes_A \mf{n}$ equipped with the differential $d = d_{\dR(A)} + d_\mf{n}$. 
Consider the $\dR(A)$-linear endomorphism $\ell_1$ obtained by extending linearly the connection $\nabla\colon \mf{n}\rt \Omega^1_A\otimes_A\mf{n}$ and let us interpret $\ell_0^{(1)}\in \dR^1(A)\otimes_A \mf{n}$ and $\ell_0^{(2)}\in \dR^2(A)\otimes_A \mf{n}$.

\begin{proposition}\label{prop:curv for transitive lie}
The  $\dR(A)$-module $\dR(A)\otimes_A \mf{n}$, equipped with:
\begin{itemize}
\item the $\dR(A)$-linear extension of the Lie bracket on $\mf{n}$,

\item the pre-differential $d+\ell_1$,

\item and the curvature $\ell_0 = \ell_0^{(1)}+\ell_0^{(2)}$
\end{itemize}
is the curved Lie algebra obtained as the image of the functor $\curv$ from Theorem \ref{thm:comparison tangent case} applied to $\Loid$.
\end{proposition}

\begin{proof}
This is a particular case of Proposition \ref{prop:curved algebroid vs curved algebra}.
A direct verification that $\dR(A)\otimes_A \mf{n}$ is a curved Lie algebra  can be done using that:
\begin{enumerate}
\item[] $[d, \ell_0^{(1)}]=0$ since the differential on $\Loid$ squares to zero.

\item[] $[d, \nabla_v]-\nabla_{dv} = \mm{ad}_{\ell_0^{(1)}(v)}$ since  the differential on $\Loid$ derivation for the Lie bracket.

\item[] $[\nabla_v, \nabla_w]-\nabla_{[v, w]} = \mm{ad}_{\ell_0^{(2)}(v, w)}$ due to the Jacobi identity on $\Loid$.
\end{enumerate}
\end{proof}

\subsubsection*{If the anchor does not split}
Next, let us consider a Lie algebroid over $A$ whose anchor $\rho\colon \Loid \to T_A$ does not admit a splitting. Proposition \ref{prop:curv for transitive lie} and Proposition \ref{prop:curved algebroid vs curved algebra} do not apply in this situation, but they do apply to a fibrant replacement of $\Loid$. Such fibrant replacements always exist for formal reasons (Theorem \ref{thm:model structure lie algebroids}), but are typically big and not very explicit. In the case where $\Loid$ is a Lie algebroid  (or $L_\infty$-algebroid) whose underlying $A$-module is perfect, i.e.\ finitely generated and projective without differential, one can also construct a fibrant replacement by geometric means: this provides a fibrant replacement of $\Loid$ which is again perfect as an $A$-module. Let us briefly describe this construction, which is inspired by \cite{grady_gwilliam}.
\begin{construction}
Let $A$ be a smooth algebra or a cofibrant cdga which is locally of finite presentation (so that $\Omega^1_A$ is a dualizable $A$-module). Let us write $A\hat{\otimes} A$ for the complete filtered cdga given by the formal completion at the diagonal, i.e.\ by the adic completion of $A\otimes A$ at the kernel $I$ of the multiplication map. Note that since $I/I^2 \cong \Omega_A^1$, there is a canonical $\widehat\Sym_A(\Omega^1_A)\cong \Gr(A\hat{\otimes} A)$ with $\Omega^1_A$ of weight $1$. By a \emph{formal exponential} we will mean an isomorphism of complete graded algebras that fits into a commuting diagram
$$\begin{tikzcd}
A\arrow[d, "i_1"{swap}]\arrow[r] & \widehat\Sym_A(\Omega^1_A)\arrow[d]\\
A\hat{\otimes} A\arrow[r]\arrow[ru, dotted, "\exp^*"{description}] & A
\end{tikzcd}$$
and induces the canonical isomorphism on the associated graded (equivalently, on the first graded piece it induces the canonical identification $I/I^2\cong \Omega^1_A$). Geometrically, such an isomorphism identifies the formal completion of the diagonal with the tangent bundle of $A$.
%formal completion of the zero section of the normal bundle (with derivative the identity)\ricardo{this is wrong I think}.
  In particular, there then exists some differential on $\widehat\Sym_A(\Omega^1_A)$ so that $A\hat{\otimes} A\cong \widehat\Sym_A(\Omega^1_A)$ as complete dg-$A$-algebras, where $A\hat{\otimes} A$ is considered an $A$-algebra \emph{from the left}.  Unraveling the definitions, the data of a formal exponential (or rather its inverse) comes down to providing a left $A$-linear section of $F^1(A\hat{\otimes} A)\twoheadrightarrow \Gr^1(A\hat{\otimes} A)\cong \Omega^1_A$; in particular, such a section exists by our assumptions on $A$.

Likewise, if $E$ is a perfect dg-$A$-module, then a \emph{formal parallel transport} on $E$ is an isomorphism of graded $(A\hat{\otimes} A)$-modules
$$\begin{tikzcd}
\mathrm{PT}\colon E\otimes_A (A\hat{\otimes} A)\arrow[r, "\cong"] & (A\hat{\otimes} A)\otimes_A E 
\end{tikzcd}$$
which is the identity when tensored up along $A\hat{\otimes} A\rt A$. 
To provide some intuition, notice that if $E$ arises from a vector bundle, an isomorphism $E\otimes_A (A{\otimes} A)\cong (A{\otimes} A)\otimes_A E$ represents a continuous identification of the fibers $E_x\cong E_y$ for each point $(x, y)\in \mm{Spec}(A)\times \mm{Spec}(A)$. A formal parallel transport provides such an identification for infinitesimally close points.

A formal parallel transport is induced by a (left) $A$-linear section of the map $(A\hat{\otimes} A)\otimes_A E\twoheadrightarrow E$ (not necessarily preserving differentials), and hence exists since $E$ is perfect.
\end{construction}

These constructions can be recovered directly from connection data:
\begin{lemma}
Suppose that $A$ is a smooth algebra or a cofibrant cdga. Every connection $\nabla$ on $T_A$ (not required to preserve differentials) induces a natural formal exponential.
\end{lemma}
\begin{proof}
Notice that the complete filtered algebra $A\hat{\otimes} A$ is the (filtered) left $A$-linear dual of the universal enveloping algebra $\mc{U}(T_A)$, equipped with its PBW filtration \cite{rinehart63}; the duality arises from the canonical (filtered) left $A$-linear pairing 
$$\begin{tikzcd}
\langle-, - \rangle\colon \mc{U}(T_A)\hat{\otimes} (A\hat{\otimes} A)\arrow[r] & A; \quad \langle D, a\otimes b\rangle = a\cdot D(b).
\end{tikzcd}$$
On the associated graded, the PBW theorem \cite{rinehart63} identifies this with the canonical (nondegenerate) pairing $\Sym_A(T_A)\otimes \widehat{\Sym}_A(\Omega^1_A)$ with $T_A$ and $\Omega^1_A$ in filtration degrees $-1$ and $1$, respectively (see e.g.\ \cite[Sec. 4.4]{calaque2010hochschild}).

Following \cite{liao2019formal}, recall that the connection $\nabla$ gives rise recursively to an $A$-linear map $\exp\colon \Sym_A(T_A) \to~\mathcal U (T_A)$, which sends $X_i\in T_A$ to $X_i\in \mathcal U (T_A)$ and satisfies

\[ \exp(X_1\dots X_n) = \frac{1}{n} \sum_{k=1}^n\pm \left(X_k\exp(X_1\dots \hat X_k\dots X_n) -\exp\left(\nabla_{X_k}(X_1\dots \hat X_k\dots X_n)\right)\right).\]
At the associated graded level, $\exp$ coincides with the PBW isomorphism and is therefore an isomorphism itself \cite[Prop. 4.2]{liao2019formal}. Taking (filtered) $A$-linear duals, we obtain the left $A$-linear isomorphism $\exp^*\colon A\hat{\otimes} A \to \widehat\Sym_A(\Omega^1_A)$; this preserves the multiplication by \cite[Thm. 4.3]{liao2019formal}.%published version 
\end{proof}
\begin{lemma}
Let $A$ be a smooth algebra or a cofibrant cdga over $k$, equipped with a formal exponential. If $E$ is a perfect $A$-module, then every connection $\nabla^E$ on $E$ (not required to respect differentials) induces a formal parallel transport on $E$.
\end{lemma}
Geometrically, the algebraic construction in the proof below can be viewed as follows: for every $k$-point $x\in \mm{Spec}(A)$, one can restrict the (graded) vector bundle with connection $E$ to the formal neighbourhood $\mm{Spec}(A)^\wedge_x$ around $x$. The formal exponential map identifies $\mm{Spec}(A)^\wedge_x$ with the tangent space $T_x\mm{Spec}(A)$. We then produce a formal parallel transport by taking the usual parallel transport along straight lines in $T_x\mm{Spec}(A)$, starting at the origin.
\begin{proof}
The connection $\nabla^E$ induces a left $A$-linear connection on the base change $(A\otimes A)\otimes_A E$ (i.e.\ a connection with respect to the tangent bundle of $A\otimes A$ relative to $A\otimes k$). This connection induces a left $A$-linear connection on the complete $A\hat{\otimes} A$-module $(A\hat{\otimes} A)\otimes_A E$. Using the formal exponential, $(A\hat{\otimes} A)\cong \widehat{\Sym}_A(\Omega^1_A) \eqqcolon B$, we obtain a left $A$-linear connection on the perfect complete $B$-module $F\coloneqq B\otimes_A E$. Note that $F$ comes with the natural adic filtration, and $\gr^0(F)\cong E$ as $A$-modules. We have to provide a natural left $A$-linear section of the map $F\rt \gr^0(F)$.

To this end, let us work in a setting where there is a further ``\emph{projective}'' grading. Consider the algebra $\mm{Sym}_A(\Omega^1_A\oplus A\cdot t)$ where $A$ has projective weight $0$, $\Omega^1_A$ has weight $1$ and $t$ has weight $-1$, and let $B_t$ be its $t$-adic completion ($B_t=\mm{Sym}_A(\Omega^1_A)[\![t]\!]$ but with different projective grading). Then $B_t$ is a complete graded ring and there is a natural map of complete $A$-algebras
$$\begin{tikzcd}
B\arrow[r] & B_t; \quad \Omega^1_A\ni \alpha\arrow[r, mapsto] & t\cdot \alpha.
\end{tikzcd}$$
This map is an isomorphism onto the part of $B_t$ in projective weight $0$; in fact, $B_t$ is given in projective weight $0$ by $\prod \mm{Sym}^n_A(\Omega^1_A)\otimes t^{n}$. Now consider the perfect (projective graded) $B_t$-module $F_t\coloneqq B_t\otimes_{B} F$, which coincides with $F$ in projective weight $0$. Explicitly, it has elements $b(t)\otimes f$, with $b(t)\in B_t\cong \mm{Sym}_A(\Omega^1_A)[\![t]\!]$ and $f\in F$, subject to the relation 
$$
b(t)\otimes (\alpha\cdot f)\sim b(t)\cdot t\cdot \alpha\otimes f
$$
for every $\alpha\in \Omega^1_A$ (and extended by continuity for the adic topology). We now claim that $F_t$ has a canonical connection $\nabla_t\colon F_t\rt F_t\otimes dt$ in the $t$-direction. Indeed, let $\nabla\colon F\rt F\otimes_B \Omega^1_{B/A}$ be the $A$-linear connection we had on $F$, and denote $\nabla(f)=\nabla^{(1)}(f)\otimes \nabla^{(2)}(f)$. We then define
$$
\nabla_t\big(b(t)\otimes f\big) = \Big(\frac{db(t)}{dt}\otimes f + \big(b(t)\cdot t\cdot \nabla^{(2)}(f)\big)\otimes \nabla^{(1)}(f)\Big)dt.
$$
One can verify that this is well-defined and indeed defines a connection in the $t$-direction on $F_t$. But for $t$-connections on a perfect (projectively graded) module $F_t$ over a power series algebra $B_t=\mm{Sym}_A(\Omega^1_A)[\![t]\!]$, we know that the composite
$$\begin{tikzcd}
\ker(\nabla_t\colon F_t\rt F_t)\arrow[r] & F_t\arrow[r] & F_t\big/t\cdot F_t
\end{tikzcd}$$
is an isomorphism of projectively graded $\mm{Sym}_A(\Omega^1_A)$-modules. Finally, considering only the part in projective weight $0$, one obtains a sequence of left $A$-linear maps 
$$\begin{tikzcd}
\ker(\nabla_t)(0)\arrow[r] & F_t(0)\cong F\arrow[r] & (F_t/tF_t)(0) \cong \gr^0(F)
\end{tikzcd}$$
where the composition is an isomorphism. This provides the desired left $A$-linear section of $F\rt \gr^0(F)$, which concludes the proof.
%
%
%We are now in the following general situation: $A$ is a graded algebra, $V$ a finitely generated projective graded $A$-module ($V=T_A$) and $F$ is a perfect complete $\widehat{\Sym}_A(V^\vee)$-module equipped such that $\Gr(F)\cong \widehat{\Sym}_A(V^\vee)\otimes_A \Gr^0(F)$. We have to associate to an $A$-linear connection $\nabla$ on $F$ a canonical $A$-linear section of the map 
%$$
%\pi\colon F\rt A\otimes_{\widehat{\Sym}_A(V^\vee)} F\cong \Gr^0(F).
%$$
%To do this, note that there is a canonical $A$-linear map $V\rt T_{\widehat{\Sym}_A(V^\vee)/A}$ sending each $v\in V$ to the constant vector field in the direction $V$; here $V$ is concentrated in weight $-1$. The connection then determines an $A$-linear map $\nabla\colon V\otimes_A F\rt F$, or by adjunction, an $A$-linear map $\nabla\colon F\rt V^\vee\otimes_A F$. At the level of the associated graded, this map is simply given by the canonical map 
%$$\begin{tikzcd}
%d\otimes \mm{id}\colon \widehat{\Sym}_A(V^\vee)\otimes_A \Gr^0(F)\arrow[r] &  \big(\widehat{\Sym}_A(V^\vee)\otimes_A V^\vee\big)\otimes_A \Gr^0(F).
%\end{tikzcd}$$
%If $F^\nabla$ denotes the kernel of $\nabla\colon F\rt V^\vee\otimes_A F$, the composite $A$-linear map
%$$\begin{tikzcd}
%F^\nabla\arrow[r] & F\arrow[r, "\pi"] & A\otimes_{\widehat{\Sym}_A(V^\vee)} F
%\end{tikzcd}$$
%is then an isomorphism (this is clear on the associated graded). This yields the desired section of $\pi$.
\end{proof}

\begin{proposition}
Let $\Loid$ be a perfect $L_\infty$-algebroid over $A$ and fix a formal exponential and formal parallel transport on $\Loid$. Then the mapping cylinder $\Loid \oplus T_A[-1]\oplus T_A$ admits a natural a $L_\infty$-algebroid structure over $A$, which fits into a commuting diagram
\begin{equation}\label{diag:fibrant replacement mapping cylinder}\begin{tikzcd}
\Loid\arrow[rd, "\rho"{swap}]\arrow[rr, rightsquigarrow] & & \Loid\oplus T_A[-1]\oplus T_A\arrow[ld, "\pi"]\\
& T_A.
\end{tikzcd}\end{equation}
Here the $\infty$-morphism has linear part given by $(\mm{id}, 0, \rho)$ (in particular, it is a quasi-isomorphism) and $\pi$ is the projection onto the last factor.
\end{proposition}
\begin{proof}
Recall that for a perfect $A$-module $\Loid$ equipped with an $A$-linear map $\rho\colon \Loid\rt T_A$, there is an equivalence between $L_\infty$-algebroid structures on $\Loid$ with anchor $\rho$ and differentials on the complete filtered symmetric algebra $\widehat{\mm{Sym}}(\Loid^\vee[-1]\wgt{1})$ making 
$$\begin{tikzcd}
\dR(A)\arrow[r, "\rho^*"] & \widehat\Sym(\Loid^\vee[-1]\wgt{-1})\arrow[r] & A
\end{tikzcd}$$ 
a sequence of complete cdgas. In this case, the middle term agrees with the Chevalley--Eilenberg complex $C^*(\Loid\wgt{1})$. 

Now suppose that $\Loid$ is an $L_\infty$-algebroid and consider the map of complete cdgas $\dR(A)\otimes C^*(\Loid\wgt{1})\rt A\otimes A \rt A$. Let us write $\hat{R}$ for the adic completion of $\dR(A)\otimes C^*(\Loid\wgt{1})$ at the kernel of this map; unraveling the definitions, $\hat{R}$ can be identified without differential with the complete symmetric algebra
\begin{equation}\label{eq:Rhat}
\hat{R}=\widehat\Sym_{A\hat{\otimes} A}\Big(\pi_1^*\Omega^1_A[-1]\oplus \pi^*_2\Loid^\vee[-1]\Big).
\end{equation}
Here the filtration arises from the filtration on $A\hat{\otimes} A$ and 
$$
\pi_1^*\Omega^1_A=\Omega^1_A\otimes_A (A\hat{\otimes} A)\qquad\qquad \pi^*_2\Loid^\vee=(A\hat{\otimes} A)\otimes_A \Loid^\vee
$$
where $\Omega^1_A$ and $\Loid^\vee$ are both of filtration weight $1$. There are natural maps of complete cdgas $i_1\colon \dR(A)\rt \hat{R}$, induced by the left inclusion of $\dR(A)$, and $\phi\colon \hat{R}\rt C^*(\Loid\wgt{1})$ induced by
$$
(\rho^*, \mm{id})\colon \dR(A)\otimes C^*(\Loid\wgt{1})\rt C^*(\Loid\wgt{1}).
$$
In terms of \eqref{eq:Rhat}, $i_1$ includes $A$ into $A\hat{\otimes}A$ from the left and maps $\Omega^1_A$ into $\pi_1^*\Omega^1_A$. Furthermore, $\phi$ is given by the diagonal $A\hat{\otimes} A\rt A$, acts as $\rho^*$ on $\pi_1^*\Omega^1_A$ and as the identity on $\pi_2^*\Loid^\vee$.

Applying formal parallel transport, we can identify $\hat{R}$ with the complete graded algebra
$$
\hat{R}\cong \widehat\Sym_{A\hat{\otimes} A}\Big(\pi_1^*\Omega^1_A[-1]\oplus \pi^*_1\Loid^\vee[-1]\Big)
$$
In this presentation, $i_1$ is still the obvious inclusion, but the map $\hat{R}\rt C^*(\Loid\wgt{1})$ only acts as the identity on $\pi^*_1\Loid^\vee[-1]$ up to first order.

Next, using the formal exponential map we can identify $A\hat{\otimes}A \cong \widehat\Sym_A(\Omega^1_A)$ as complete $A$-algebras, where $A\hat{\otimes}A$ is seen as an $A$-algebra from the left. Using this, we find that the sequence of complete cdgas $\dR(A)\rt \hat{R}\rt C^*(\Loid\wgt{1})$ can be identified without differential with the sequence of $A$-algebra maps
\begin{equation}\label{diag:sequence of ce algebras}\begin{tikzcd}
\widehat\Sym_A\big(\Omega^1_A[-1]\big)\arrow[r, "i_1"] & \widehat\Sym_A\Big(\Omega^1_A\oplus \Omega^1_A[-1]\oplus \Loid^\vee[-1]\Big)\arrow[r, "\phi"] & \widehat\Sym_A\big(\Loid^\vee[-1]\big).
\end{tikzcd}\end{equation}
Here $i_1$ is simply induced by the summand inclusion $\Omega^1_A[-1]\rt \Omega^1_A\oplus \Omega^1_A[-1]\oplus \Loid^\vee[-1]$. The map $\phi$ is given on the summand $\Omega^1_A[-1]$ by $\rho^*$, while the restriction to $\Omega^1_A\oplus \Loid^\vee[-1]$ is given by the projection onto $\Loid^\vee[-1]$ up to higher order terms. 

The differential on $\hat{R}$ endows $\widehat\Sym_A\big(\Omega^1_A\oplus \Omega^1_A[-1]\oplus \Loid^\vee[-1]\big)$ with a differential, such that the associated graded is precisely the symmetric algebra on the mapping cylinder. This makes \eqref{diag:sequence of ce algebras} a diagram of cdgas, dual to the desired diagram of $L_\infty$-algebroids \eqref{diag:fibrant replacement mapping cylinder}.
\end{proof}
Even if $A$ is a smooth discrete algebra and $\Loid$ is a Lie algebroid over it concentrated in degree $0$, the mapping cylinder $\Loid\oplus T_A[-1]\oplus T_A$ will generally only come with an $L_\infty$-algebroid structure. The higher brackets depend on higher order terms appearing in the formal exponential map $\widehat\Sym_A(\Omega^1_A)\cong A\hat{\otimes} A$ and the formal parallel transport map $\Loid\otimes_A (A\hat{\otimes} A)\cong (A\hat{\otimes} A)\otimes_A \Loid$.
\begin{corollary}[cf.\ \cite{grady_gwilliam}]
Let $A$ be a smooth algebra or a cofibrant cdga of finite type and let $\Loid \to T_A$ be a Lie algebroid over $A$ whose underlying $A$-module is perfect. Then the complete $\dR(A)$-module $\dR(A)\otimes_A (\Loid\oplus T_A[-1])$ admits the structure of a curved $L_\infty$-algebra which gives a model for $\curv(\Loid)$.
\end{corollary}

\subsection{The Maurer--Cartan space}\label{sec:rogers}
In this section we show how to use our constructions at the level of curved Lie algebras to recover results concerning the classical homotopy theory of $L_\infty$-algebras over a field of characteristic $0$.

Let $\cat{Alg}_{L_\infty}^{c.p.f.}$ denote the category of complete positively filtered  $L_\infty$ algebras, with $\infty$-morphisms. 
Due to the assumption of positive filtration we can associate to such an $L_\infty$-algebra $\mathfrak g$ its Maurer--Cartan set

\[
\MC(\mf g) = \{x\in \mf g | dx + \frac{1}{2!}\ell_2(x,x) + \frac{1}{3!}\ell_3(x,x,x) + \dots =0\}.
\]

This construction extends to the Maurer--Cartan space, which is the functor
\begin{center}
\begin{tikzcd}[row sep=0]
\mathcal{M}\mathcal{C}\colon \cat{Alg}_{L_\infty}^{c.p.f.} \arrow[r] &\mathrm{Kan\ complexes}\\
\mf g & \MC(\mf g \otimes \Omega[\Delta^\bullet])
\end{tikzcd}
\end{center}

where $\Omega[\Delta^\bullet]$ denotes the polynomial differential forms on the simplex, see \cite{dolgushevrogers}.

It is not entirely trivial to see that the Maurer--Cartan functor is well defined on $\infty$-morphisms, let alone that it satisfies the right homotopical properties. 
It turns out that under our framework this functor takes the form of a mapping space if we consider the larger category of curved $L_\infty$-algebras, which establishes in particular that  $\mathcal M \mathcal C$ is a representable functor.

\begin{proposition}\label{prop:MC spaces are mapping spaces}
	Let $\mf g$ be a positively filtered $L_\infty$ algebra, seen as an object in the simplicial category of curved $L_\infty$-algebras $\oocat{cLie}$.
	Then $$\mathcal M \mathcal C (\mf g) = \Map_{\oocat{cLie}}(0,\mf g).$$
\end{proposition}

\begin{proof}
	Following the explicit form of an $\infty$-morphism as presented in section \ref{sec:oo-morphism of clie}, since $0$ is the trivial curved $L_\infty$-algebra, an $\infty$-morphism $\phi\colon 0\to \mf g$ can only have one non-trivial piece, namely $\phi_0$, which is represented by a filtration $1$, degree $1$ element satisfying precisely the Maurer--Cartan condition.	
	
	The result follows from $F^1 \mf g=\mf g$.
\end{proof}

Since in an $\infty$-category mapping spaces preserve weak equivalences, as a corollary we obtain the main result of \cite{dolgushevrogers}.
\begin{theorem}[Theorem 1.1 of \cite{dolgushevrogers}]
Let $\mf g \rightsquigarrow \mf h$ be a weak equivalence of positively filtered complete $L_\infty$-algebras.
Then, $\mathcal{M}\mathcal{C}(\mf g)\to \mathcal{M}\mathcal{C}(\mf h)$ is a homotopy equivalence of simplicial sets.
\end{theorem}

Furthermore, by expressing the Maurer--Cartan functor as a mapping space, the corollary above we can show the $\infty$-categorical version of a result due to Rogers \cite{rogers2020complete}.

\begin{theorem}[See Theorem 3 of \cite{rogers2020complete}]
	The Maurer--Cartan functor $\mathcal{M}\mathcal{C}\colon \cat{Alg}_{L_\infty}^{c.p.f.} \to \mathrm{Kan\ complexes}$ induces a left exact functor at the $\infty$-categorical level.
\end{theorem}

\begin{proof}
Following Proposition \ref{prop:MC spaces are mapping spaces}, at the $\infty$-categorical level $\mathcal{M}\mathcal{C}$ is a (covariant) mapping space functor, which is therefore left exact.
\end{proof}

\subsection{A differential-geometric variant}\label{sec:diffgeo}
Let us conclude with a brief discussion of Theorem \ref{thm:comparison tangent case} in the differential-geometric setting, where a similar construction has been studied in \cite{grady_gwilliam}. Let $M$ be a (Hausdorff, second countable) smooth manifold and let $A=\mc{C}^\infty(M)$ be the ring of $\mc{C}^\infty$-functions on $M$.
Let $\cat{Sh}_{\mc{O}_M}$ denote the category of sheaves of $\mc{O}_M$-modules on $M$, where $\mc{O}_M=\mc{C}^\infty(-)$ is the structure sheaf. Then there is an adjoint pair 
$$\begin{tikzcd}
	(-)^\sim\colon \Mod_{\mc{C}^\infty(M)}\arrow[r, yshift=1ex]\arrow[r, hookleftarrow, yshift=-1ex] & \cat{Sh}_{\mc{O}_M}(M) \colon \Gamma
\end{tikzcd}
$$
where the fully faithful right adjoint takes global sections and the left adjoint sends an $\mc{C}^\infty(M)$-module $V$ to the associated sheaf of $V\otimes_{\mc{C}^\infty(M)} \mc{C}^\infty(-)$ \cite[Section 5.4]{Joyce2019}. In other words, $\mc{O}_M$-module sheaves are completely determined by their global sections.
%This respects the tensor product in the sense that $(V\otimes_{\mc{C}^\infty(M)} W)^\sim\cong V^\sim\otimes_{\mc{O}_M} W^{\sim}$. 

Furthermore, the Serre--Swan theorem asserts that $\Gamma$ restricts to a monoidal equivalence between locally free sheaves -- i.e.\ vector bundles on $M$ -- and finitely generated projective $\mc{C}^\infty(M)$-modules. For example, the module $T_{\mc{C}^\infty(M)}$ of algebra derivations of $\mc{C}^\infty(M)$ agrees with the module $\Gamma(M, TM)$ of vector fields on $M$ and its $\mc{C}^\infty(M)$-linear dual agrees with the module $\Gamma(M, T^*M)$ of $1$-forms on $M$ (however, its natural \emph{pre}dual, consisting of K\"ahler differentials, is not finitely generated).

\begin{definition}
	Let $\mf{g}$ be an $L_\infty$-algebroid over $\mc{C}^\infty(M)$ in the sense of Definition \ref{def:loo algebroid}. We will say that $\mf{g}$ is:
	\begin{enumerate}
		\item a \emph{sheaf of $L_\infty$-algebroids} on $M$ if the $\mc{C}^\infty(M)$-module underlying $\mf{g}$ arises as the global sections of a complex of sheaves of $\mc{O}_M$-modules on $M$.
		
		\item a \emph{differential-geometric} $L_\infty$-algebroid if $\mf{g}$ is bounded above and each $\mf{g}_n$ arises from a vector bundle, i.e.\ it is a finitely generated projective $\mc{C}^\infty(M)$-module.
	\end{enumerate}
\end{definition}
To justify this terminology, let us point out that under the Serre--Swan theorem, differential-geometric $L_\infty$-algebroids over $M$ indeed correspond to the usual definition of an $L_\infty$-algebroid over a smooth manifold considered in the literature \cite{laurent2018universal}; the only possible exception is that we allow \emph{unbounded} complexes of vector bundles, while one typically only considers nonpositively graded complexes of vector bundles. In particular, $T_{\mc{C}^\infty(M)}$ itself corresponds to the usual tangent Lie algebroid $T_M$ and its Chevalley--Eilenberg complex $C^*(T_{\mc{C}^\infty(M)})\cong \Omega^*(M)$ is isomorphic to the usual de Rham complex of $M$.

On the other hand, suppose that $\mf{g}\rt T_{\mc{C}^\infty(M)}$ is a sheaf of $L_\infty$-algebroids in the above sense. For every open $U$, the map 
$$\begin{tikzcd}
	\mf{g}\otimes_{\mc{C}^\infty(M)} \mc{C}^\infty(U)\arrow[r] & T_{\mc{C}^\infty(M)}\otimes_{\mc{C}^\infty(M)} \mc{C}^\infty(U)\cong T_{\mc{C}^\infty(U)}
\end{tikzcd}$$
is the anchor map of a natural $L_\infty$-algebroid $\mf{g}\otimes_{\mc{C}^\infty(M)} \mc{C}^\infty(U)$ over $\mc{C}^\infty(U)$. Taking associated sheaves, one sees that $\mf{g}^{\sim}\rt T_{\mc{C}^\infty(M)}^\sim = \mc{X}(-)$ gives a sheaf of $L_\infty$-algebroids, with anchor map taking values in the sheaf of vector fields on $M$. By our assumption on $\mf{g}$, the global sections of this sheaf of $L_\infty$-algebroids coincides with $\mf{g}$ itself. In this way, sheaves of $L_\infty$-algebroids on $M$ embed fully faithfully in $L_\infty$-algebroids over $\mc{C}^\infty(M)$, by taking global sections.
\begin{theorem}\label{thm:comparison differentiable case}
	Let $M$ be a differentiable manifold and $\mc{C}^\infty(M)$. Then there is a fully faithful inclusion of $\infty$-categories
	$$\begin{tikzcd}
		\curv\colon \oocat{Lie}(\mc{C}^\infty(M)/\mathbb{R})\arrow[r, hook] & \oocat{cLie}_{\Omega^*(M)}
	\end{tikzcd}$$
	from the $\infty$-category of $L_\infty$-algebroids over $\mc{C}^\infty(M)$ to the $\infty$-category of curved $L_\infty$-algebras over the de Rham complex $\Omega^*(M)$, filtered by form degree. Furthermore, $\curv$ restricts to equivalence between the full subcategories of:
	\begin{enumerate}
		\item sheaves of $L_\infty$-algebroids on $M$ and curved $L_\infty$-algebroids $\mf{h}$ over $\Omega^*(M)$ such that $\Gr^i(\mf{h})\simeq \Omega^i(M)\otimes_{\mc{C}^\infty(M)} \Gr^0(\mf{h})$ is an equivalence and $\Gr^0(\mf{h})$ arises as the global sections of a complex of sheaves of $\mc{O}_M$-modules.
		
		\item differential-geometric $L_\infty$-algebroids on $M$ and curved $L_\infty$-algebras over $\Omega^*(M)$ equivalent to a curved $L_\infty$-algebra of the form $\Omega^*(M)\otimes_{\mc{C}^\infty(M)} E$, with $E$ a bounded above graded vector bundle on $M$.
	\end{enumerate} 
\end{theorem}
\begin{proof}
	The Lie algebroid $T_{\mc{C}^\infty(M)}$ satisfies the conditions of Theorem \ref{thm:comparison}. Furthermore, the $\infty$-category of (algebraic) $L_\infty$-algebroids over $T_{\mc{C}^\infty(M)}$ embeds into the $\infty$-category of curved $L_\infty$-algebroids over $T_{\mc{C}^\infty(M)}\wgt{1}$ via the functor $(-)^\mm{anc}$. The proof of Theorem \ref{thm:comparison tangent case} now carries over verbatim to show that the essential image of $\oocat{Lie}(\mc{C}^\infty(M)/\mathbb{R})$ consists of curved $L_\infty$-algebroids $\mf{h}$ with $\Gr^i(\mf{h})\simeq \Omega^i(M)\otimes_{\mc{C}^\infty(M)} \Gr^0(\mf{h})$.
	
	Now note from the construction of the functor $\curv$ that $\Gr^0(\curv(\mf{g}))$ is weakly equivalent to the mapping fiber of the anchor $\rho\colon \mf{g}\rt T_A$. In particular, $\Gr^0(\curv(\mf{g}))$ arises as the global sections of a complex of sheaves of $\mc{O}_M$-modules if and only if $\mf{g}$ does. 
	
	Likewise, $\Gr^0(\curv(\mf{g}))$ is weakly equivalent to a complex of vector bundles if and only if $\mf{g}$ is weakly equivalent to a complex of vector bundles. But if $\mf{h}$ is a curved $L_\infty$-algebra with $\Gr^0(\mf{h})$ weakly equivalent to a complex of vector bundles, then $\mf{h}$ is itself weakly equivalent to a curved $L_\infty$-algebra of the form $\Omega^*(M)\otimes_{\mc{C}^\infty(M)} E$ for a graded vector bundle $E$, by the Homotopy Transfer Theorem \ref{thm:htt-B}. Similarly, if $\mf{g}$ is an $L_\infty$-algebroid whose underlying $\mc{C}^\infty(M)$-module is weakly equivalent to a bounded above complex of vector bundles, then $\mf{g}$ is weakly equivalence to a differential-geometric $L_\infty$-algebroid by the Homotopy Transfer Theorem \ref{thm:htt lie algebroids}.
\end{proof}
\begin{remark}
	Part (2) of Theorem \ref{thm:comparison differentiable case} can be made more explicit as follows. Given a curved $L_\infty$-algebra over the filtered cdga $\Omega^*(M)$ whose underlying graded module is of the form $\Omega^*(M)\otimes_{\mc{C}^\infty(M)} E$ for a bounded above graded vector bundle $E$ on $M$. Proposition \ref{prop:curved algebroid vs curved algebra} shows that the graded vector bundle $TM\oplus E$ carries a natural $L_\infty$-algebroid structure, with the anchor given by the projection to $TM$.
	
	Conversely, let $\rho\colon \mf{g}\rt TM$ be a differential-geometric $L_\infty$-algebroid. If $\rho$ is surjective, we can choose a splitting $\mf{g}=TM\oplus \keranc$ and Proposition \ref{prop:curved algebroid vs curved algebra} determines a curved $L_\infty$-structure on $\Omega^*(M)\otimes_{\mc{C}^\infty(M)} \keranc$. 
	%The situation is more complicated when $\rho$ is not surjective, and is dealt with in \cite{grady_gwilliam}. Let us briefly outline how these constructions fit into our framework. 
	%
	%In loc.\ cit.\  the authors consider the complete cdga $\dR(J(C^*(\mf{g})))$ given by the de Rham complex of the jet algebra of $C^*(\mf{g})$ (endowed with the Hodge filtration). Choosing connection data identifies this with a complete commutative algebra of the form $\widehat{\Sym}\big(T^\vee_A[-1]\oplus T^\vee_A\oplus \mf{g}^\vee[-1]\big)$ with all generators in filtration weight $1$. The differential determines an $L_\infty$-algebroid structure on $T_A[-1]\oplus T_A\oplus \mf{g}$, such that the underlying complex has differential given by the sum of the internal differential on $\mf{g}$ and a differential sending $(0, v, \xi)$ to $(v-\rho(\xi), 0, 0)$. There exists a map of complete cdgas $\dR(J(C^*(\mf{g})))\rt C^*(\mf{g})$\joost{??} that dually determines an $\infty$-quasi-isomorphism $i_\infty\colon \mf{g}\rt \big(\mf{g}\oplus T_A\oplus T_A[-1], d\big)$ of $L_\infty$-algebroids. We can now apply the $\curv$ functor to the equivalent $L_\infty$-algebroid $\big(\mf{g}\oplus T_A\oplus T_A[-1], d\big)$ and obtain the curved $L_\infty$-algebra over $\Omega^*(M)$ denoted by $\mm{enh}(\mf{g})^\sigma$ in \cite[Theorem 4.2]{grady_gwilliam}.
\end{remark}

\begin{remark}
	The reader with derived inclinations may also take $A$ to be a nonpositively graded cdga of the form $\big(\mc{C}^\infty(M)[x_1, \dots, x_n], d\big)$ where the variables $x_1, \dots, x_n$ are of strictly negative degree. Such cdgas arise as the algebras of functions on derived manifolds, in which case the module of derivations $T_A$ indeed models the tangent sheaf of the derived manifold. Part (1) of Theorem \ref{thm:comparison differentiable case} holds in this setting as well.
\end{remark}

%\section{To Do list?}
%
%\begin{itemize}
%	
%	
%	\item revise section 5
%
%	
%	---------------For v2 ------------------
%	\item Can we make Proposition \ref{prop:graded mixed complexes} into a Quillen functor? What is the adjoint?
%	
%	\item  Consider including the (now deleted) section `Deforming Lie algebras'
%	
%	\item Elaborate the section on differential geometry, namely when the anchor is not surjective
%	%--------- Future project?--------
%	%
%	%\item FMPs for uCom? Obtain Joosts result directly from curved Lie algebras.
%	%\item Try to find an inverse to the Chevalley--Eilenberg
%	%
%	%\item Generalise \cite{calaque2019moduli} to curved operads. What would it look like for $uE_n$?
%
%	
%
%\end{itemize}

% \listoftodos

\bibliographystyle{alpha}
{\small\bibliography{biblio}}

\newcommand{\etalchar}[1]{$^{#1}$}
\begin{thebibliography}{CRvdB10}

\bibitem[AT20]{AmorimTu2020}
Lino Amorim and Junwu Tu.
\newblock The inverse function theorem for curved {L}-infinity spaces.
\newblock {\em arXiv preprint arXiv:2008.0147}, 2020.

\bibitem[Baa]{baarsmaphd}
Arjen Baarsma.
\newblock Deformations and {$L_\infty$}-algebras of {F}r\'echet type.
\newblock {\em PhD thesis, available at
  \url{https://dspace.library.uu.nl/handle/1874/386311}}.

\bibitem[BM03]{berger2003axiomatic}
Clemens Berger and Ieke Moerdijk.
\newblock Axiomatic homotopy theory for operads.
\newblock {\em Commentarii Mathematici Helvetici}, 78(4):805--831, 2003.

\bibitem[BMDC20]{Bellier-MillesDrummond-Cole2020}
Joan Bellier-Mill\`es and Gabriel~C. Drummond-Cole.
\newblock Homotopy theory of curved operads and curved algebras.
\newblock {\em arXiv preprint arXiv:2007.03004}, 2020.

\bibitem[Cam19]{camposhomotopy2019}
Ricardo Campos.
\newblock Homotopy equivalence of shifted cotangent bundles.
\newblock {\em J. Lie Theory}, 29(3):629--646, 2019.

\bibitem[CCT14]{CalaqueCaldararuTu2014}
Damien Calaque, Andrei C\u{a}ld\u{a}raru, and Junwu Tu.
\newblock On the {L}ie algebroid of a derived self-intersection.
\newblock {\em Adv. Math.}, 262:751--783, 2014.

\bibitem[CF07]{CattaneoFelder}
Alberto~S. Cattaneo and Giovanni Felder.
\newblock Relative formality theorem and quantisation of coisotropic
  submanifolds.
\newblock {\em Adv. Math.}, 208(2):521--548, 2007.

\bibitem[CG18]{calaque2018formal}
Damien Calaque and Julien Grivaux.
\newblock Formal moduli problems and formal derived stacks.
\newblock {\em arXiv preprint arXiv:1802.09556}, 2018.

\bibitem[CLM16]{ChuangLazarevMannan2016}
Joseph Chuang, Andrey Lazarev, and W.~H. Mannan.
\newblock Cocommutative coalgebras: homotopy theory and {K}oszul duality.
\newblock {\em Homology Homotopy Appl.}, 18(2):303--336, 2016.

\bibitem[Cos11]{costellowittenII}
Kevin Costello.
\newblock A geometric construction of the {W}itten genus, {II}.
\newblock {\em arXiv preprint arXiv:1112.0816}, 2011.

\bibitem[CPT{\etalchar{+}}17]{CPTVV2017}
Damien Calaque, Tony Pantev, Bertrand To\"{e}n, Michel Vaqui\'{e}, and Gabriele
  Vezzosi.
\newblock Shifted {P}oisson structures and deformation quantization.
\newblock {\em J. Topol.}, 10(2):483--584, 2017.

\bibitem[CRvdB10]{calaque2010hochschild}
Damien Calaque, Carlo~A. Rossi, and Michel van~den Bergh.
\newblock Hochschild (co)homology for {L}ie algebroids.
\newblock {\em Int. Math. Res. Not. IMRN}, (21):4098--4136, 2010.

\bibitem[CSLW19]{cirici2019model}
Joana Cirici, Daniela~Egas Santander, Muriel Livernet, and Sarah Whitehouse.
\newblock Model category structures and spectral sequences.
\newblock {\em Proceedings of the Royal Society of Edinburgh Section A:
  Mathematics}, pages 1--34, 2019.

\bibitem[CT13]{CaldararuTu}
Andrei C\u{a}ld\u{a}raru and Junwu Tu.
\newblock Curved {$A_\infty$} algebras and {L}andau-{G}inzburg models.
\newblock {\em New York J. Math.}, 19:305--342, 2013.

\bibitem[DR15]{dolgushevrogers}
Vasily~A. Dolgushev and Christopher~L. Rogers.
\newblock A version of the {G}oldman-{M}illson theorem for filtered
  {$L_\infty$}-algebras.
\newblock {\em J. Algebra}, 430:260--302, 2015.

\bibitem[DSV18]{DotsenkoShadrinVallette2018}
Vladimir Dotsenko, Sergey Shadrin, and Bruno Vallette.
\newblock The twisting procedure.
\newblock {\em arXiv preprint arXiv:1810.02941}, 2018.

\bibitem[FOOO09]{FOOO}
Kenji Fukaya, Yong-Geun Oh, Hiroshi Ohta, and Kaoru Ono.
\newblock {\em Lagrangian intersection {F}loer theory: anomaly and obstruction.
  {P}art {I} and {II}}, volume~46 of {\em AMS/IP Studies in Advanced
  Mathematics}.
\newblock American Mathematical Society, Providence, RI; International Press,
  Somerville, MA, 2009.

\bibitem[Fre04]{Fresse2004}
Benoit Fresse.
\newblock Koszul duality of operads and homology of partition posets.
\newblock In {\em Homotopy theory: relations with algebraic geometry, group
  cohomology, and algebraic {$K$}-theory}, volume 346 of {\em Contemp. Math.},
  pages 115--215. Amer. Math. Soc., Providence, RI, 2004.

\bibitem[Fre09]{fresse2009}
Benoit Fresse.
\newblock {\em Modules over operads and functors}, volume 1967 of {\em Lecture
  Notes in Mathematics}.
\newblock Springer-Verlag, Berlin, 2009.

\bibitem[Fuk03]{Fukaya2003}
Kenji Fukaya.
\newblock Deformation theory, homological algebra and mirror symmetry.
\newblock In {\em Geometry and physics of branes ({C}omo, 2001)}, Ser. High
  Energy Phys. Cosmol. Gravit., pages 121--209. IOP, Bristol, 2003.

\bibitem[Get18]{Getzler2018}
Ezra Getzler.
\newblock Maurer-cartan elements and homotopical perturbation theory.
\newblock {\em arXiv preprint arXiv:1802.06736}, 2018.

\bibitem[GG20]{grady_gwilliam}
Ryan Grady and Owen Gwilliam.
\newblock Lie algebroids as {$L_\infty$} spaces.
\newblock {\em J. Inst. Math. Jussieu}, 19(2):487--535, 2020.

\bibitem[GR17]{gaitsgory_rozenblyum_vol2}
Dennis Gaitsgory and Nick Rozenblyum.
\newblock {\em A study in derived algebraic geometry. {V}ol. {II}.
  {D}eformations, {L}ie theory and formal geometry}, volume 221 of {\em
  Mathematical Surveys and Monographs}.
\newblock American Mathematical Society, Providence, RI, 2017.

\bibitem[Hal92]{halperin1992}
Stephen Halperin.
\newblock Universal enveloping algebras and loop space homology.
\newblock {\em J. Pure Appl. Algebra}, 83(3):237--282, 1992.

\bibitem[Hin97]{hinich1997homological}
Vladimir Hinich.
\newblock Homological algebra of homotopy algebras.
\newblock {\em Comm. Algebra}, 25(10):3291--3323, 1997.

\bibitem[Hin16]{hinichlocalization}
Vladimir Hinich.
\newblock Dwyer-{K}an localization revisited.
\newblock {\em Homology Homotopy Appl.}, 18(1):27--48, 2016.

\bibitem[HM12]{HirshMilles2012}
Joseph Hirsh and Joan Mill\`es.
\newblock Curved {K}oszul duality theory.
\newblock {\em Math. Ann.}, 354(4):1465--1520, 2012.

\bibitem[Joy19]{Joyce2019}
Dominic Joyce.
\newblock Algebraic geometry over {$C^{\infty}$}-rings.
\newblock {\em Mem. Amer. Math. Soc.}, 260(1256):v+139, 2019.

\bibitem[Kon03]{Kontsevich2003}
Maxim Kontsevich.
\newblock Deformation quantization of {P}oisson manifolds.
\newblock {\em Letters in Mathematical Physics. A Journal for the Rapid
  Dissemination of Short Contributions in the Field of Mathematical Physics},
  66(3):157--216, 2003.

\bibitem[KS06]{HiroshigeStasheff2006}
Hiroshige Kajiura and Jim Stasheff.
\newblock Homotopy algebras inspired by classical open-closed string field
  theory.
\newblock {\em Comm. Math. Phys.}, 263(3):553--581, 2006.

\bibitem[LGLS18]{laurent2018universal}
Camille Laurent-Gengoux, Sylvain Lavau, and Thomas Strobl.
\newblock The universal {L}ie {$\infty$}-algebroid of a singular foliation.
\newblock {\em arXiv preprint arXiv:1806.00475}, 2018.

\bibitem[LS19]{liao2019formal}
Hsuan-Yi Liao and Mathieu Sti\'{e}non.
\newblock Formal exponential map for graded manifolds.
\newblock {\em Int. Math. Res. Not. IMRN}, (3):700--730, 2019.

\bibitem[Lur11a]{DAGX}
Jacob Lurie.
\newblock Derived algebraic geometry {X}: Formal moduli problems.
\newblock {\em preprint available at
  \url{https://www.math.ias.edu/~lurie/papers/DAG-X.pdf}}, 2011.

\bibitem[Lur11b]{LurieDAGXIII}
Jacob Lurie.
\newblock Derived algebraic geometry {XIII}: {R}ational and p-adic {H}omotopy
  {T}heory.
\newblock {\em preprint available at
  \url{https://www.math.ias.edu/~lurie/papers/DAG-XIII.pdf}}, 2011.

\bibitem[Lur18]{kerodon}
Jacob Lurie.
\newblock Kerodon.
\newblock \url{https://kerodon.net}, 2018.

\bibitem[LV12]{LodayVallette2012}
Jean-Louis Loday and Bruno Vallette.
\newblock {\em Algebraic operads}.
\newblock Number 346 in Grundlehren der Mathematischen Wissenschaften
  [Fundamental Principles of Mathematical Sciences]. Springer, 2012.

\bibitem[Mac87]{mackenzie87}
K.~Mackenzie.
\newblock {\em Lie groupoids and {L}ie algebroids in differential geometry},
  volume 124 of {\em London Mathematical Society Lecture Note Series}.
\newblock Cambridge University Press, Cambridge, 1987.

\bibitem[Mau15]{maunder2015unbased}
James Maunder.
\newblock Unbased rational homotopy theory: a {L}ie algebra approach.
\newblock {\em arXiv preprint arXiv:1511.07669}, 2015.

\bibitem[Mau17]{maunder2017Koszul}
James Maunder.
\newblock Koszul duality and homotopy theory of curved {L}ie algebras.
\newblock {\em Homology Homotopy Appl.}, 19(1):319--340, 2017.

\bibitem[Nui19a]{nuiten2019homotopicalalgebra}
Joost Nuiten.
\newblock Homotopical algebra for {L}ie algebroids.
\newblock {\em Appl. Categ. Structures}, 27(5):493--534, 2019.

\bibitem[Nui19b]{nuiten2019koszul}
Joost Nuiten.
\newblock {K}oszul duality for {L}ie algebroids.
\newblock {\em Advances in Mathematics}, 354:106750, 2019.

\bibitem[Pos11]{positselski2011Twokinds}
Leonid Positselski.
\newblock Two kinds of derived categories, {K}oszul duality, and
  comodule-contramodule correspondence.
\newblock {\em Mem. Amer. Math. Soc.}, 212(996):vi+133, 2011.

\bibitem[Pos18]{positselskiweaklycurved}
Leonid Positselski.
\newblock Weakly curved {${A}_{\infty}$}-algebras over a topological local
  ring.
\newblock {\em M\'{e}m. Soc. Math. Fr. (N.S.)}, (159):vi+206, 2018.

\bibitem[Pra67]{pradines67}
Jean Pradines.
\newblock Th\'{e}orie de {L}ie pour les groupo\"{\i}des diff\'{e}rentiables.
  {C}alcul diff\'{e}rentiel dans la cat\'{e}gorie des groupo\"{\i}des
  infinit\'{e}simaux.
\newblock {\em C. R. Acad. Sci. Paris S\'{e}r. A-B}, 264:A245--A248, 1967.

\bibitem[Pri10]{Pridham}
J.~P. Pridham.
\newblock Unifying derived deformation theories.
\newblock {\em Adv. Math.}, 224(3):772--826, 2010.

\bibitem[PS20]{pymsafronov2020}
Brent Pym and Pavel Safronov.
\newblock Shifted symplectic {L}ie algebroids.
\newblock {\em Int. Math. Res. Not. IMRN}, (21):7489--7557, 2020.

\bibitem[Qui69]{QuillenAnnals}
Daniel Quillen.
\newblock Rational homotopy theory.
\newblock {\em Ann. of Math. (2)}, 90:205--295, 1969.

\bibitem[Qui73]{QuillenKTheory}
Daniel Quillen.
\newblock Higher algebraic {$K$}-theory. {I}.
\newblock In {\em Algebraic {$K$}-theory, {I}: {H}igher {$K$}-theories ({P}roc.
  {C}onf., {B}attelle {M}emorial {I}nst., {S}eattle, {W}ash., 1972)}, pages
  85--147. Lecture Notes in Math., Vol. 341, 1973.

\bibitem[Rin63]{rinehart63}
George~S. Rinehart.
\newblock Differential forms on general commutative algebras.
\newblock {\em Trans. Amer. Math. Soc.}, 108:195--222, 1963.

\bibitem[Rog20]{rogers2020complete}
Christopher~L. Rogers.
\newblock Complete filtered {${L}_\infty$}-algebras and their homotopy theory,
  2020.

\bibitem[SS12]{SchlessingerStasheff2012}
Mike Schlessinger and Jim Stasheff.
\newblock Deformation theory and rational homotopy type.
\newblock {\em arXiv preprint arXiv:1211.1647}, 2012.

\bibitem[Val20]{vallette2014homotopy}
Bruno Vallette.
\newblock Homotopy theory of homotopy algebras.
\newblock {\em Annales de l'Institut Fourier}, 70(2):683--738, 2020.

\bibitem[Vez15]{vezzosi2015}
Gabriele Vezzosi.
\newblock A model structure on relative dg-{L}ie algebroids.
\newblock In {\em Stacks and categories in geometry, topology, and algebra},
  volume 643 of {\em Contemp. Math.}, pages 111--118. Amer. Math. Soc.,
  Providence, RI, 2015.

\bibitem[Yu17]{Yu2017Dolbeault}
Shilin Yu.
\newblock Dolbeault dga and {$L_\infty$}-algebroid of the formal neighborhood.
\newblock {\em Adv. Math.}, 305:1131--1162, 2017.

\end{thebibliography}

\end{document}